\documentclass[final,1p]{elsarticle}
\graphicspath{{./Figures/}}
\usepackage{amssymb,latexsym,amsmath,amsthm}
\usepackage[colorlinks,linkcolor=blue, anchorcolor=blue, citecolor=blue]{hyperref}
\usepackage[ruled,linesnumbered]{algorithm2e}
\usepackage{color,caption,cases}
\usepackage{bm}
\usepackage{graphicx,graphics}
\usepackage{booktabs}
\usepackage{tabularx}
\usepackage[T1]{fontenc}
\usepackage{multirow, makecell}
\usepackage{subfigure}
\biboptions{numbers,sort&compress} 
\numberwithin{equation}{section}
\allowdisplaybreaks
\newtheorem{theorem}{Theorem}[section]

\newtheorem{lemma}[theorem]{Lemma}

\theoremstyle{definition}

\newtheorem{example}[theorem]{Example}

\theoremstyle{remark}
\newtheorem{remark}[theorem]{Remark}

\newcommand{\f}{\frac}
\newcommand{\p}{\partial}

\newcommand{\mbb}{\mathbb}

\newcommand{\mM}{\mathcal{M}}
\newcommand{\mE}{\mathcal{E}}

\journal{CMAME}
\begin{document}    

\begin{frontmatter}

\title{Fully decoupled, linear and structure-preserving block-centered finite difference methods for the Keller–Segel chemotaxis system on staggered non-uniform grids}

\author[OUC]{Jie Xu} \ead{jxu129@163.com}
\author[OUC,LMM]{Hongfei Fu\corref{Fu}}\ead{fhf@ouc.edu.cn}

\address[OUC]{School of Mathematical Sciences, Ocean University of China, Qingdao 266100, P.R. China}
\address[LMM]{Laboratory of Marine Mathematics, Ocean University of China, Qingdao 266100, P.R. China}

\cortext[Fu]{Corresponding author.}

\begin{abstract}
In this paper, we propose two fully decoupled, linear and structure-preserving block-centered finite difference schemes for the classical Keller–Segel chemotaxis system on staggered non-uniform spatial grids. Both novel schemes are second-order accurate in space; one is first-order accurate in time, while the other achieves second-order temporal accuracy. Moreover, we show that the schemes preserve several inherent physical laws at the discrete level: (i) the positivity of both the cell density and the chemoattractant concentration; (ii) the conservation of total cell mass; and (iii) a discrete energy dissipation property for the first-order scheme. In particular, the temporally first-order scheme unconditionally preserves positivity, mass conservation, and energy dissipation, whereas the second-order scheme ensures positivity under a sufficient (but not necessary) time-step condition. The proposed methods yield more accurate and efficient simulations of chemotactic dynamics, especially in the presence of rapid blow-up phenomena, on specified non-uniform spatial grids. Numerical experiments are conducted to validate the theoretical findings and to illustrate the accuracy and reliability of the proposed schemes.
\end{abstract}


\end{frontmatter}

\section{ Introduction}
Chemotaxis is the directed movement of cells or organisms in response to a chemical stimulus, commonly modeled by a system of nonlinear partial differential equations (PDEs). In the 1970s, Keller and Segel \cite{KS'70, KS'71} introduced the so-called classical Keller--Segel chemotaxis system, in which the dimensionless model seeks two scalar functions $\rho = \rho(\bm{x}, t)$ and $c = c(\bm{x}, t)$ satisfying \cite{CC'08}:
 \begin{subequations}\label{model:ks}
	\begin{align}
          \p_t \rho  = \Delta \rho-  \nabla \cdot(\rho \nabla c),  
                      &\quad \text{in}~~ \Omega \times(0, T], \label{model:ks:rho}\\
         \varepsilon \p_t c  = \Delta c- \alpha c+ \rho,   &\quad \text{in}~~ \Omega \times(0, T], \label{model:ks:c}
	\end{align}
\end{subequations}
enclosed with no-flux/Neumann boundary conditions:
\begin{equation}\label{model:ks:bc}
	(\nabla \rho - \rho \nabla c) \cdot \bm{n}=0,\quad \nabla c \cdot \bm{n}=0, \qquad\text {on}~~ \p\Omega \times(0, T], 
\end{equation}
and initial conditions:
\begin{equation}\label{model:ks:ic}
\begin{aligned}
   \rho(\bm{x},0)  =\rho^o(\bm{x}), \quad c(\bm{x},0)=c^o(\bm{x}),   \quad \text {in} ~~\Omega,
\end{aligned}
\end{equation}
where  $\Omega $ is a convex, bounded and open domain in $\mbb R^2$, $\bm{n}$ represents the outward unit normal vector on the boundary $\p \Omega$. The unknown variables in \eqref{model:ks} can be interpreted biologically as follows: $\rho$ denotes the cell (or organism) density at position $\bm{x} \in \Omega$ and time $t \in [0, T]$, while $c$ represents the chemoattractant concentration, a chemical signal that induces cell migration. Both cells and the chemoattractant undergo diffusion within the spatial domain.
The parameters $\varepsilon$ and $\alpha$ are both non-negative.
Specifically, $\varepsilon$ describes the response rate of the chemoattractant concentration to the cell density, $\alpha$ is the reaction coefficient of the chemoattractant.
The model is a parabolic-parabolic system when $\varepsilon > 0$ and a parabolic-elliptic system when $\varepsilon = 0$. When $\varepsilon \ll 1$ the model is in a transition regime between the parabolic-parabolic and parabolic-elliptic cases. 
The mathematical theory of the classical Keller--Segel model is extensive; we refer the reader to \cite{GZ'98, CEM'14, CC'08, LWZ'18, HW'01, HP'09} and the references therein. 

The Keller--Segel system \eqref{model:ks}--\eqref{model:ks:ic} possesses three important fundamental physical laws. The first one is the conservation of total mass, i.e.,
\begin{equation}\label{model:ks:MC}
       M_\rho(t) := \int_\Omega \rho(\bm{x}, t)  d\bm{x} = \int_\Omega \rho^o(\bm{x})  d\bm{x}=M_\rho(0), \quad \forall t>0.
\end{equation}
The second physical law of the Keller-Segel model \eqref{model:ks}--\eqref{model:ks:ic}  is the preservation of positivity, that is, if the initial condition $\rho^o > 0$,  then we have $\rho(t) >0$; furthermore, if $c^o > 0$ also holds,  then we further have $c(t) >0$. Third, following the positivity of cell density, we define the free energy 
\begin{equation}\label{energy:continue}
\mE(t)= \mE[\rho(\cdot, t), c(\cdot, t))]
:=\int_{\Omega}\left[\rho \log \rho-\rho-\rho c+\frac{1}{2}|\nabla c|^2 + \frac{1}{2}\alpha c^2 \right] d x.
\end{equation}
Then, the nonlinear system \eqref{model:ks} under the given boundary conditions \eqref{model:ks:bc} can be reformulated into the following mixed conservative and non-conservative gradient flow
\begin{equation}\label{model:gradient}
\rho_t=\nabla \cdot\left(\rho \nabla \frac{\delta \mE}{\delta \rho}\right), \quad \varepsilon c_t=-\frac{\delta \mE}{\delta c},
\end{equation}
which reflects the gradient-flow structure of the model. Due to \eqref{energy:continue}--\eqref{model:gradient}, the following entropy-dissipation law of the Keller-Segel model \eqref{model:ks}--\eqref{model:ks:ic}  holds: 
\begin{equation}\label{model:ks:ED}
\frac{d}{d t} \mE(t)=-\int_{\Omega}\left[\rho|\nabla(\log \rho-c)|^2 + \varepsilon\left|\partial_t c\right|^2\right] d x \le 0, \quad \forall t>0,
\end{equation}
i.e., $\mE(t) \le \mE(s)$ for all $s < t$.

In addition, it is important to note that the production of chemoattractant by cells, to which the cells themselves are attracted, can lead to a phenomenon known as chemotactic collapse, potentially resulting in uncontrolled aggregation and ultimately finite-time blow-up \cite{CC'08, LWZ'18}. Specifically, a finite-time blow-up may occur when the initial density has a total mass exceeding the critical threshold of $8 \pi $ while maintaining a finite second moment. Conversely, if the initial total mass is strictly below this threshold, a continuous solution is guaranteed for all times. This blow-up phenomenon is also widely recognized as a key feature of the classical Keller--Segel model and poses a significant challenge for numerical simulations, making the ability to accurately capture it a rigorous test of any numerical method.

A large number of numerical methods have been developed for the Keller--Segel system \eqref{model:ks}--\eqref{model:ks:ic}, largely because analytical solutions are rarely available. For example, Li et al. \cite{LSY'17} proposed a local discontinuous Galerkin method and obtained optimal convergence rates before blow-up; they also introduced a positivity-preserving limiter to better capture the blow-up time. Xiao et al. \cite{XFH'19} developed a semi-implicit characteristic finite element method for blow-up, pattern formation, and aggregation dynamics, with second-order accuracy in both the $L^2$ and $H^1$ norms. Sulman and Nguyen \cite{SN'19} proposed an adaptive moving-mesh implicit-explicit finite element method, illustrating the effectiveness of non-uniform spatial grids. More recently, Shen and Xu \cite{SX'20} developed bound-preserving and energy-stable Galerkin schemes based on the conservative gradient flow structure. They proved that the resulting first-order scheme is mass conservative, bound preserving, uniquely solvable, and energy dissipative; while the second-order BDF2 scheme is not energy stable. However, these schemes require solving, at each time step, a nonlinear system which is a unique minimizer of a strictly convex functional. Later, utilizing the scalar auxiliary variable (SAV) approach and function transformation, Huang and Shen \cite{HS'21} constructed some temporal semi-discrete schemes, which are linear, bound/positivity-preserving, and unconditionally modified energy stable. At a broader methodological level, Lagrange multiplier technique provides an alternative to the construction of positivity-preserving structure-preserving schemes \cite{CS'22}. In 2025, Wang et al. \cite{WLF'24} proposed a first-order decoupled linear fully discrete finite element method, based on a log-transformation and a mass-recovery step, which is unconditionally positivity-preserving and mass-conserving. Though it is not energy dissipative, the authors proved the energy bound for the scheme, and meanwhile, optimal-order error estimates are also analyzed. Very recently, Ding et al. \cite{DWZ'25} developed a second-order positivity-preserving and energy dissipative finite difference scheme for the Keller--Segel system with various mobilities. This seem to be the first second-order accurate scheme in literature that could achieve both the numerical positivity and original energy dissipation, though the proposed scheme is nonlinear and has to be solved iteratively. Other approaches, including operator splitting methods \cite{TSL'19, M'03}, hybrid finite-volume/finite-difference methods \cite{CE'18}, finite difference methods \cite{S'09, EK'12, HZ'23}, and meshless methods \cite{BGG'20, DA'19}, are also discussed. In the context of more complex chemotaxis-fluid models, fully decoupled, linear, positivity-preserving finite element schemes based on the flux-corrected transport (FCT) algorithm have also been developed and analyzed for the chemotaxis–Stokes equations \cite{HFXW'21,FHW'21}.

Despite this progress, efficient and provably structure-preserving discretizations, especially finite difference methods, on \textit{non-uniform} grids remain limited. In particular, when blow-up occurs, the highly localized solution structures necessitates the use of local mesh refinement to be accurately captured. Otherwise, the computational cost becomes substantial, as very fine uniform grids are required to maintain accuracy. Motivated by this challenge, we recently proposed a mass-conserving block-centered finite difference (BCFD) scheme for the parabolic-parabolic Keller--Segel system on staggered non-uniform spatial grids \cite{XF'25}. In this scheme, the primal variables $\rho$ and $c$ are approximated at cell centers, while their fluxes are discretized at the midpoints of cell edges. This work presents the first investigation of the BCFD method for this model problem on general \textit{non-uniform} grids, establishing rigorous second-order superconvergence for both cell density and chemoattractant concentration. However, the preservation of positivity and energy dissipation are not addressed therein, which constitutes the main aim of the present paper.

To develop structure-preserving BCFD schemes on non-uniform grids, we adopt the Slotboom transformation \cite{Slotboom'73}--an approach widely employed in the discretization of Poisson--Nernst--Planck equations \cite{LW'14, HH'20, DWZ'19}. Specifically, this transformation reformulates the cell density equation \eqref{model:ks:rho} into the following equivalent symmetric form:
\begin{equation}\label{model:rho:rewrite}
	\partial_t \rho=\nabla \cdot\left(\mM \nabla\left(\frac{\rho}{\mM}\right)\right),~~ \mM:={e}^{\log \rho-\frac{\delta \mE}{\delta \rho}}={e}^{c},
\end{equation}
where $\rho/\mM = \rho e^{-c}$ is referred to as a generalized Slotboom variable \cite{Slotboom'73}, and the no-flux boundary condition for $\rho$ is now expressed as: 
\begin{equation}\label{model:rho:rewrite:bc}
\nabla\left(\frac{\rho}{\mM}\right) \cdot \bm{n}=0, \qquad\text {on}~~ \p\Omega \times(0, T].
\end{equation}
Besides, the total free energy \eqref{energy:continue} is now rewritten equivalently as
\begin{equation}\label{energy:continue:rewrite}
	\mE[\rho, c]=\int_{\Omega}\left[\rho \log \left(\frac{\rho}{\mM}\right)-\rho+\frac{1}{2}|\nabla c|^2 + \frac{1}{2}\alpha c^2 \right] d x.
\end{equation}
Note that the equivalent form \eqref{model:rho:rewrite} can be interpreted as a \textit{variable-coefficient} conservative diffusion equation, thereby enabling the construction of efficient positivity-preserving and energy-dissipative schemes.
For example, Liu et al. \cite{LWZ'18} constructed a structure-preserving finite difference scheme for the parabolic--parabolic Keller--Segel system \eqref{model:ks} with $\alpha=0$. Their method combines the
backward Euler discretization in time with central finite differences in space, and preserves mass conservation and positivity at the fully discrete level. Moreover, the scheme is designed from a symmetrized formulation of the density equation and is asymptotic-preserving in the quasi-static limit. Recently, Lu et al. \cite{LCLL'24} developed two alternating direction implicit (ADI) central finite difference schemes for the two-dimensional (2D) model with minimal computational cost, in which the temporal first-order scheme unconditionally preserves positivity, mass conservation and asymptotical energy dissipation laws; while the second-order one conditionally preserves positivity. Hu and Zhang \cite{HZ'23} introduced two positivity-preserving and energy-dissipative finite difference schemes for the parabolic--elliptic Keller--Segel system. 
Both schemes are first-order accurate in time, and second-order and fourth-order accurate in space. The first scheme is proved to be  unconditionally positivity preserving and energy dissipative, while the second one achieves such properties under a mild time-step  constraint. It is worth noting that all these finite difference schemes are formulated on \textit{uniform} spatial grids. 
In this paper, starting from the Slotboom reformulation \eqref{model:rho:rewrite}, we aim to develop two linearized and decoupled BCFD schemes on staggered \textit{non-uniform} grids. The first scheme is temporally first-order and is proven to unconditionally preserve positivity, mass conservation, and a discrete version energy-dissipation law; while the second scheme achieves second-order temporal accuracy but can only preserve positivity and mass conservation. To our knowledge, this work represents the first BCFD framework for the Keller--Segel system that jointly offers non-uniform grid flexibility and multiple structure-preserving properties.

The remainder of this paper is organized as follows. Section \ref{sec:not} introduces key notations and some preliminary results. Section \ref{sec:mc-bcfd} presents the construction of linearized, decoupled BCFD schemes of first- and second-order temporal accuracy for the reformulated system. Structure-preserving numerical analysis, including positivity-preserving, mass conservation, and energy dissipation, is carefully discussed in Section \ref{sec:sp}. Numerical experiments are provided in Section \ref{sec:num} to validate the theoretical findings and demonstrate the reliability of the proposed schemes. Finally, Some conclusions are addressed in Section \ref{sec:conclusion}.

\section{Notations and preliminaries}\label{sec:not}
 Let $N_t$ be a positive integer. Consider the non-uniform temporal partition $ 0 = t_{0} < t_{1} < \ldots < t_{N_t} = T$ with stepsize $ \tau_{n} = t_{n} - t_{n-1} $ for $ 1 \leq n \leq N_t$. Let $\tau= \max_{1 \leq n \leq N_t} \tau_{n}$ be the maximum stepsize. For temporal grid function  $\{\phi^n\}$, define
\begin{align*}
 d_t\phi^{n+1} = \f{\phi^{n+1}-\phi^{n}}{\tau_{n+1}},\quad
 \overline{\phi}^{n+1/2} = \f{\phi^{n+1} + \phi^{n}}{2}.
\end{align*}

For simplicity,  we consider the two-dimensional rectangular domain $\Omega:=(a^x, b^x)\times (a^y, b^y)$. Let $N_x$ and $N_y$ be the number of grids along the $x$ and $y$ coordinates, respectively. Similar to those used in \cite{WW'88}, staggered spatial grids are introduced, where the primal grid points $\Pi_x \times \Pi_y$ are denoted by
\begin{align*}
\Pi_x: \quad a^x=x_{1/2}< \ldots < x_{i-1/2}<x_{i+1/2} < \ldots <  x_{N_x+1/2}=b^x,\\
\Pi_y: \quad a^y=y_{1/2}< \ldots < y_{j - 1/2}<y_{j+1/2} < \ldots <  y_{N_y+1/2}=b^y,
\end{align*}
with grid sizes $\Delta x=\{\Delta x_i=x_{i+1/2}-x_{i-1/2}\mid i=1, \ldots, N_x\}$ and 
$\Delta y=\{\Delta y_j=y_{j+1/2}-y_{j-1/2}\mid j=1, \ldots, N_y\}$, and the auxiliary staggered grid points are denoted by
\begin{align*}
	 \Pi_x^*: \quad x_i= (x_{i-1/2}+x_{i+1/2} ) / 2,  ~~ i=1, \ldots, N_x,\\
     \Pi_y^*: \quad y_j= (y_{j-1/2}+ y_{j+1/2} ) / 2,  ~~ j=1, \ldots, N_y,
\end{align*}
with grid sizes  
$\Delta x_{i+1/2}=x_{i+1}-x_i= (\Delta x_{i+1} +\Delta x_i)/2$ for $i=1,\ldots, N_x-1$ and 
$\Delta y_{j+1/2}=y_{j+1}-y_j= (\Delta y_{j+1} +\Delta y_j)/2$ for $j=1,\ldots, N_y-1$. 
We use $\bm{h_x} := [\Delta x_1, \Delta x_2, \cdots, \Delta x_{N_x}]^\top$ to represent the $N_x$-dimensional column vector along $x$-direction, and similarly, we define $\bm{h_y}:= [\Delta y_1, \Delta y_2, \cdots, \Delta y_{N_y}]^\top$ as the $N_y$-dimensional column vector along $y$-direction. We also use $\Delta x:=\max \Delta x_i $ and $\Delta y:=\max \Delta y_j$ to represent the maximum grid sizes along $x$ and $y$ directions. Let $h:=\max \{\Delta x, \Delta y\}$ and assume that the spatial grid partition is regular, i.e., there exists a positive constant $\sigma$ such that
$ 	\max \{\f{h}{\min_{i } \Delta x_i}, \f{h}{\min_{j } \Delta y_j} \} \le \sigma.
$ 

Moreover, given spatial grid functions $g=\{g_{i, j}\}$, $\hat{g}=\{g_{i+1/2, j}\}$ and $\check{g}=\{g_{i, j+1/2}\}$ respectively defined on $\Pi_x^*\times \Pi_y^*$, $\Pi_x\times \Pi_y^*$ and $\Pi_x^*\times \Pi_y$,  we define
\begin{align*}
	& [d_x g]_{i+1/2, j}=\f{g_{i+1, j}-g_{i, j}}{\Delta x_{i+1/2}}, ~~	
        && [d_y g]_{i, j+1/2}=\f{g_{i, j+1}-g_{i, j}} {\Delta y_{j+1/2}}, \\
	& [D_x \hat{g}]_{i, j}=\f{\hat{g}_{i+1/2, j}-\hat{g}_{i-1/2, j}}{\Delta x_i}, ~~
	&&  [D_y \check{g}]_{i, j}=\f{\check{g}_{i, j+1/2}-\check{g}_{i, j-1/2}}{\Delta y_j}, 
\end{align*}
and denote $\bm{d}g = \left(d_xg, d_yg\right)$. Besides, we introduce the discrete inner products and norms on $\Pi_x^*\times \Pi_y^*$, $\Pi_x\times \Pi_y^*$ and $\Pi_x^*\times \Pi_y$ as follows:
\begin{align*}
	& (f, g)_{\rm M}=\sum_{i=1}^{N_x} \sum_{j=1}^{N_y} \Delta x_i \Delta y_j f_{i, j} g_{i, j},  &&\|f\|_{\rm M}^2 = (f, f)_{\rm M}, \\
	& (f, g)_x=\sum_{i=1}^{N_x-1} \sum_{j=1}^{N_y} \Delta x_{i+1/2} \Delta y_j f_{i+1/2, j} g_{i+1/2, j},  &&\|f\|_x^2 = (f, f)_x, \\
	& (f, g)_y=\sum_{i=1}^{N_x} \sum_{j=1}^{N_y-1} \Delta x_i \Delta y_{j+1/2} f_{i, j+1/2} g_{i, j+1/2},  &&\|f\|_y^2 = (f, f)_y,\\
	& (\bm {f}, \bm {g})_{\rm TM}  = (f^x, g^x )_x+ (f^y, g^y)_y,    &&\|\bm {f}\|_{\rm TM}^2 = (\bm{f}, \bm{f})_{\rm TM}.
\end{align*}

Next,  
for any point $(x, y) \in [x_i, x_{i+1}]\times [y_j, y_{j+1}]$, $i=1,\ldots, N_x-1$, $j=1,\ldots, N_y-1$, we define  $\Pi_h p(x, y)$ as the piecewise bilinear interpolation of $p(x, y)$ with values $\{p_{i,j} = p(x_i,y_j)\}$ such that
\begin{equation}\label{Operator:bi}
	\begin{aligned}
		\Pi_h p(x, y)& =    \f{(x_{i+1}-x)(y_{j+1}-y)}{(x_{i+1}-x_i)(y_{j+1}-y_j)}p_{i, j} 
		+ \f{(x-x_i)(y_{j+1}-y)}{(x_{i+1}-x_i)(y_{j+1}-y_j)} p_{i+1, j} \\
		& \quad + \f{(x_{i+1}-x)(y-y_j)}{(x_{i+1}-x_i)(y_{j+1}-y_j)}  p_{i, j+1} 
		+ \f{(x-x_i)(y-y_j)}{(x_{i+1}-x_i)(y_{j+1}-y_j)} p_{i+1, j+1}.
	\end{aligned}
\end{equation}

Finally, we give some useful lemmas that will play an important role in the subsequent numerical analysis. 
\begin{lemma} [\cite{WW'88}] \label{lemma:Dd}
 Let $q=\{q_{i, j}\}, v=\{v_{i+1/2, j}\}$ and $w=\{w_{i, j+1 / 2}\}$ be any grid functions defined on $\Pi_x^*\times \Pi_y^*$, $\Pi_x\times \Pi_y^*$ and $\Pi_x^*\times \Pi_y$, such that $v_{1/2, j}=$ $v_{N_x+1/2, j}=w_{i, 1 / 2}=w_{i, N_y+1 / 2}=0$. Then there holds
$$
\begin{aligned}
 (q, D_x v)_{\rm M}=-(d_x q, v)_x, ~~ (q, D_y w)_{\rm M}=-(d_y q, w)_y .
\end{aligned}
$$
\end{lemma}

\begin{lemma}[\cite{HZ'23}]\label{lem:M-matrix}
For a real square matrix $A$ with positive diagonal entries and non-positive off-diagonal entries,  it is a non-singular M-matrix if all the row sums of  $A$ are non-negative and at least one row sum is positive.
\end{lemma}

\begin{lemma}[\cite{P'77}]\label{lem:M-matrix:inverse}
Let $A$ be a real square matrix. If $A$ is a non-singular M-matrix, then it is inverse-positive. That is, $A^{-1}$ exists and $A^{-1} \geq 0.$
\end{lemma}

\section{Numerical schemes}\label{sec:mc-bcfd}
In this section, we propose two linear, fully decoupled BCFD schemes for the nonlinear Keller–Segel chemotaxis system \eqref{model:ks}–\eqref{model:ks:ic} on general non-uniform staggered spatial grids, which are first- and second-order accurate in time, respectively. Below, we use $\left\{\rho_h^n, c_h^n\right\}$ with $\rho_h^n=\{\rho_{h,i,j}^n\}$, $c_h^n=\{c_{h,i,j}^n\}$ that defined on $\Pi_x^*\times \Pi_y^*$ to represent the approximations of the exact solutions $\left\{\rho^n, c^n\right\}$ at time level $t^n$ ($0\le n\le M$). Moreover, whenever no confusion is caused, such symbols also refer to the two-dimensional $N_x$-by-$N_y$ matrices.

\subsection{The temporal first-order scheme}
 Noting that the primal variables $\rho$ and $c$ are discretized on $\Pi_x^*\times \Pi_y^*$, while the two different ${\mM}$ in the nonlinear cross-diffusion term $\nabla \cdot\left(\mM \nabla\left(\frac{\rho}{\mM}\right)\right)$ should be approximated on different spatial grids: the former one is discretized on $\Pi_x\times \Pi_y^*$ and $\Pi_x^*\times \Pi_y$ along $x$- and $y$-directions, respectively, while the latter one is approximated on $\Pi_x^*\times \Pi_y^*$.  Thus, we employ the piecewise bilinear interpolation operator $\Pi_h$ defined in \eqref{Operator:bi} to ensure these approximations are well-defined.
 
 The first-order backward Euler type BCFD (BE-BCFD) scheme for the Keller--Segel system in the form \eqref{model:ks:c}  and \eqref{model:rho:rewrite}--\eqref{model:rho:rewrite:bc} is proposed as follows: 
\begin{equation}\label{BE-BCFD}
\left\{
	\begin{aligned}
		& \varepsilon d_t c_{h}^{n+1} =  D_x (d_x c_h^{n+1}) + D_y (d_y c_h^{n+1})- \alpha c_{h}^{n+1} + \rho_{h}^{n}, \\
		& d_t \rho_{h}^{n+1} =  D_x \Big(\mM_x^{n+1} d_x \big( \f{\rho_h^{n+1}}{\mM_h^{n+1}}   \big) \Big)    +   D_y \Big(\mM_y^{n+1} d_y \big( \f{\rho_h^{n+1}}{\mM_h^{n+1}} \big) \Big),
	\end{aligned}
    \right.
\end{equation}
on $\Pi_x^*\times \Pi_y^*$, enclosed with the following boundary and initial conditions 
\begin{equation}\label{BE-BCFD:IBc}
  \left\{
	\begin{aligned}
		& \left[d_x c_h^{n+1}\right]_{i+1/2, j}
			=\Big[d_x \big(\f{\rho_h^{n+1}}{\mM_h^{n+1}}\big)\Big]_{i+1/2, j} =0, \quad i=\{0, N_x\}, ~ 1 \leq j \leq N_y, \\
		& \left[d_y c_h^{n+1}\right]_{i, j+1/2}
		    =\Big[d_y \big(\f{\rho_h^{n+1}}{\mM_h^{n+1}}\big)\Big]_{i, j+1/2} =0, \quad 1 \leq i \leq N_x, ~ j=\{0, N_y\},\\	
		& c_{h,i,j}^0 = c^o(x_i,y_j),\quad\rho_{h,i,j}^0 = \rho^o(x_i,y_j), \quad 1 \leq i \leq N_x, 1 \leq j \leq N_y,
	\end{aligned}
  \right.
\end{equation}
where $ \mM_{x}^{n+1}$, $\mM_{y}^{n+1}$ and $\mM_{h}^{n+1}$ are defined on staggered spatial grids
$\Pi_x\times \Pi_y^*$, $\Pi_x^*\times \Pi_y$ and $\Pi_x^*\times \Pi_y^*$, respectively, i.e.,
\begin{equation}\label{BE-BCFD:e3}
  \mM_{x,i+1/2,j}^{n+1} := {\rm e}^{[\Pi_h c_h]_{i+1/2,j}^{n+1}},~~
  \mM_{y,i,j+1/2}^{n+1} := {\rm e}^{[\Pi_h c_h]_{i,j+1/2}^{n+1}},~~
  \mM_{h,i,j}^{n+1} := {\rm e}^{c_{h,i,j}^{n+1}}.
\end{equation}

Next, we introduce two auxiliary variables $g_h^{n+1}:=\f{\rho_h^{n+1}}{\mM_h^{n+1}}$ and $\hat{g}_h^{n}:=\f{\rho_h^{n}}{\mM_h^{n+1}}$. Then, the BE-BCFD scheme \eqref{BE-BCFD}--\eqref{BE-BCFD:IBc} can be rewritten into an equivalent form:
\begin{subequations}\label{BE-BCFD:rewrite}
	\begin{numcases}{}
		\varepsilon	d_t c_{h}^{n+1} = D_x (d_x c_h^{n+1})
                 + D_y (d_y c_h^{n+1})-  \alpha c_{h}^{n+1} + \rho_{h}^{n}, \label{BE-BCFD:c:rewrite}\\
	\f{ \mM_h^{n+1}g_h^{n+1}  - \mM_h^{n+1}\hat{g}_h^{n} }{\tau_{n+1}}= D_x (\mM_x^{n+1} (d_x g_h^{n+1}) )  + D_y (\mM_y^{n+1} (d_y g_h^{n+1} )),  \label{BE-BCFD:rho:rewrite}	
	\end{numcases}
\end{subequations}
enclosed with boundary and initial conditions
\begin{equation}\label{BE-BCFD:IBc:rewrite}
\left\{
  \begin{aligned}
	& \left[d_x c_h^n\right]_{i+1/2, j}
		=[d_x g_h^{n}]_{ i+ 1/2, j}=0, \quad i=\{0, N_x\}, ~ 1 \leq j \leq N_y, \\
       & \left[d_y c_h^n\right]_{i, j+1/2}
		=[d_y g_h^{n}]_{i, j+1/2} =0, \quad 1 \leq i \leq N_x, ~ j=\{0, N_y\}, \\	
	 & c_{h,i,j}^0 = c^o(x_i,y_j),  \quad\rho_{h,i,j}^0 = \rho^o(x_i,y_j), \quad 1 \leq i \leq N_x, 1 \leq j \leq N_y. 
	\end{aligned}
  \right.
\end{equation}

\begin{remark} Note that the BE-BCFD scheme \eqref{BE-BCFD}--\eqref{BE-BCFD:IBc} is fully decoupled and linearly implicit. In practical computation, at each time level, one first solves \eqref{BE-BCFD:c:rewrite}--which involves a \textit{constant-coefficient} matrix--to obtain $c_h^{n+1}$. Next, the auxiliary variables $\mM_h^{n+1}$, $\mM_x^{n+1}$, and $\mM_y^{n+1}$ are evaluated using \eqref{BE-BCFD:e3}. Finally, \eqref{BE-BCFD:rho:rewrite} is solved for $g_h^{n+1}$, and $\rho_h^{n+1}$ is then recovered via $\rho_h^{n+1}=\mM_h^{n+1}g_h^{n+1}$. In summary, the fully discrete non-uniform BE-BCFD scheme \eqref{BE-BCFD:rewrite}--\eqref{BE-BCFD:IBc:rewrite} for the Keller–Segel system is implemented sequentially as follows.
\end{remark}
\begin{algorithm}
\caption{Implementation of the BE-BCFD scheme \eqref{BE-BCFD:rewrite}--\eqref{BE-BCFD:IBc:rewrite}}\label{algthm:BE-BCFD}
\KwData{Given initial values $\rho_h^0$ and $c_h^0$}
\For{$n\leftarrow 0$ \KwTo $N_t-1$}{
$c_h^{n+1} \leftarrow $ solve \eqref{BE-BCFD:c:rewrite} using the last time value $\rho_h^n$\;
$\mM_x^{n+1},~ \mM_y^{n+1},~ \mM_h^{n+1} \leftarrow $ compute \eqref{BE-BCFD:e3} with the obtained value $c_h^{n+1}$\;
$g_h^{n+1}\leftarrow$ solve \eqref{BE-BCFD:rho:rewrite} using the computed value $\hat{g}_h^n:=\frac{\rho_h^n}{\mM_h^{n+1}}$\;
$\rho_h^{n+1} \leftarrow $ update $\rho_h^{n+1}:=\mM_h^{n+1} g_h^{n+1}$\;}
\KwResult{Numerical solutions $\{\rho_h^n\}$ and $\{c_h^n$\}}
\end{algorithm}

\subsection{The temporal second-order scheme}
To achieve second-order accuracy in time, we propose a prediction-correction block-centered finite difference (PC-BCFD) scheme for the Keller--Segel system. To be specific, the scheme is divided into the following two steps:

\textbf{Step 1.} On $\Pi_x^*\times \Pi_y^*$, we first solve $(c_{h}^{n+1/2}, \rho_{h}^{n+1/2})$ such that
\begin{equation}\label{PC-BCFD:e1a}
\left\{
	\begin{aligned}
         & \varepsilon \f{ c_{h}^{n+1/2}- c_{h}^{n}}{\tau_{n+1}/2} = D_x (d_x c_h^{n+1/2}) 
                   + D_y (d_y c_h^{n+1/2})	-  \alpha c_{h}^{n+1/2} + \rho_{h}^{n},  \\
	& \f{ \rho_{h}^{n+1/2}- \rho_{h}^{n} }{\tau_{n+1}/2} = D_x \Big(\mM_x^{n+1/2} d_x  \big(\f{\rho_h^{n+1/2}}{\mM_h^{n+1/2}} \big)\Big) + D_y \Big(\mM_y^{n+1/2}  d_y  \big(\f{\rho_h^{n+1/2}}{\mM_h^{n+1/2}} \big)\Big), 
	\end{aligned}
  \right.
\end{equation}
enclosed with homogeneous boundary conditions 
\begin{equation}\label{PC-BCFD:IBc}
	\left\{
	\begin{aligned}
		&\left[d_x c_h^{n+1/2}\right]_{i+1/2, j}
		=\Big[d_x \Big(\f{\rho_h^{n+1/2}}{\mM_h^{n+1/2}}\Big)\Big]_{i+1/2, j} =0, \quad i=\{0, N_x\}, ~ 1 \leq j \leq N_y, \\
		& \Big[d_y c_h^{n+1/2}\Big]_{i, j+1/2}
		=\Big[d_y \Big(\f{\rho_h^{n+1/2}}{\mM_h^{n+1/2}}\Big)\Big]_{i, j+1/2} =0, \quad 1 \leq i \leq N_x, ~ j=\{0, N_y\},
	\end{aligned}
	\right.
\end{equation}
where $ \mM_{x}^{n+1/2}$, $\mM_{y}^{n+1/2}$ and $\mM_h^{n+1/2}$ are defined similarly as in \eqref{BE-BCFD:e3}.

\textbf{Step 2.} With the predicted value $\rho_h^{n+1/2}$, we then solve $(c_{h}^{n+1}, \rho_{h}^{n+1})$ such that
\begin{equation}\label{PC-BCFD:e1b}
    \left\{
	\begin{aligned}
		& \varepsilon d_t c_{h}^{n+1} = D_x (d_x \overline{c}_h^{n+1/2}) + D_y (d_y \overline{c}_h^{n+1/2})	-  \alpha \overline{c}_{h}^{n+1/2} + \rho_{h}^{n+1/2}, \\
	    & d_t \rho_{h}^{n+1} = D_x \Big(\mM_x^{n+1/2} d_x \Big(\overline{\f{\rho_h}{\mM_h}}^{n+1/2}\Big)\Big)
               + D_y \Big(\mM_y^{n+1/2}  d_y \Big(\overline{\f{\rho_h}{\mM_h}}^{n+1/2}\Big) \Big),  
	\end{aligned}
  \right.
\end{equation}
enclosed with the same boundary and initial conditions \eqref{BE-BCFD:IBc}.

By introducing another auxiliary variables $g_h^{n+1/2}:=\f{\rho_h^{n+1/2}}{\mM_h^{n+1/2}}$ and  $\check{g}_h^{n}:=\f{\rho_h^{n}}{\mM_h^{n+1/2}}$, the second-order PC-BCFD scheme \eqref{PC-BCFD:e1a}--\eqref{PC-BCFD:e1b} reduces to
\begin{subequations}\label{PC-BCFD:rewrite}
	\begin{numcases}{}
		\varepsilon \f{ c_{h}^{n+1/2}- c_{h}^{n}}{\tau_{n+1}/2} 
                = D_x (d_x c_h^{n+1/2}) + D_y (d_y c_h^{n+1/2})	-  \alpha c_{h}^{n+1/2} + \rho_{h}^{n}, \label{PC-BCFD:c_half:rewrite} \\
		\f{ \mM_h^{n+1/2}g_h^{n+1/2} - \mM_h^{n+1/2}\check{g}_h^{n}}{\tau_{n+1}/2} = D_x (\mM_x^{n+1/2} (d_x g_h^{n+1/2})) 
        + D_y (\mM_y^{n+1/2} (d_y g_h^{n+1/2})),  \label{PC-BCFD:rho_half:rewrite}\\
		\varepsilon d_t c_{h}^{n+1} = D_x (d_x \overline{c}_h^{n+1/2}) + D_y (d_y \overline{c}_h^{n+1/2})
		-  \alpha \overline{c}_{h}^{n+1/2} + \rho_{h}^{n+1/2}, \label{PC-BCFD:c:rewrite} \\
		\f{\mM_h^{n+1}g_h^{n+1} - \mM_h^{n+1}\hat{g}_h^{n}}{\tau_{n+1}} = D_x (\mM_x^{n+1/2} (d_x \overline{g}_h^{n+1/2}))
                +  D_y (\mM_y^{n+1/2} (d_y \overline{g}_h^{n+1/2}) ),  \label{PC-BCFD:rho:rewrite}
	\end{numcases}
\end{subequations}
enclosed with boundary and initial conditions
\begin{equation}\label{PC-BCFD:IBc:rewrite}
	\left\{
	\begin{aligned}
		& \left[d_x c_h^n\right]_{i+1/2, j}
		=[d_x g_h^{n}]_{ i+ 1/2, j}=0, \quad i=\{0, N_x\}, ~ 1 \leq j \leq N_y, \\
		& \left[d_x c_h^{n+1/2}\right]_{i+1/2, j}
		=[d_x g_h^{n+1/2}]_{ i+ 1/2, j}=0, \quad i=\{0, N_x\}, ~ 1 \leq j \leq N_y, \\
		& \left[d_y c_h^n\right]_{i, j+1/2}
		=[d_y g_h^{n}]_{i, j+1/2} =0, \quad 1 \leq i \leq N_x, ~ j=\{0, N_y\},\\
		& \left[d_y c_h^{n+1/2}\right]_{i, j+1/2}
		=[d_y g_h^{n+1/2}]_{i, j+1/2} =0, \quad 1 \leq i \leq N_x, ~ j=\{0, N_y\},\\	
		&c_{h,i,j}^0 = c^o(x_i,y_j),  \quad \rho_{h,i,j}^0 = \rho^o(x_i,y_j),	 \quad 1 \leq i \leq N_x, 1 \leq j \leq N_y.
	\end{aligned}
	\right.
\end{equation}

 \begin{remark} 
	Although the spatial local truncation errors of the BE-BCFD scheme \eqref{BE-BCFD:rewrite}--\eqref{BE-BCFD:IBc:rewrite} and PC-BCFD scheme \eqref{PC-BCFD:rewrite}--\eqref{PC-BCFD:IBc:rewrite} are only first-order under non-uniform spatial grids, the error estimates under discrete $L^2$-norm are proved to be second-order accurate in space; see, for example, \cite{RP'12,XXF'22}. To summarize, the fully discrete non-uniform finite difference scheme \eqref{PC-BCFD:rewrite}--\eqref{PC-BCFD:IBc:rewrite} for the Keller--Segel system  is implemented as follows:
\end{remark}
\begin{algorithm}
\caption{Implementation of the PC-BCFD scheme \eqref{PC-BCFD:rewrite}--\eqref{PC-BCFD:IBc:rewrite}}\label{algthm:PC-BCFD}
\KwData{Given initial point values $\rho_h^0$ and $c_h^0$}
\For{$n\leftarrow 0$ \KwTo $N_t-1$}{
$c_h^{n+1/2} \leftarrow $ solve \eqref{PC-BCFD:c_half:rewrite} using the last time value $\rho_h^n$\;
$\mM_x^{n+1/2},~ \mM_y^{n+1/2} \leftarrow$ compute \eqref{BE-BCFD:e3} with the obtained value $ c_h^{n+1/2}$\;
$g_h^{n+1/2}\leftarrow$ solve \eqref{PC-BCFD:rho_half:rewrite} using the computed value $\check{g}_h^n:=\frac{\rho_h^n}{\mM_h^{n+1/2}}$\;
$\rho_h^{n+1/2} \leftarrow $ compute $\rho_h^{n+1/2}=\mM_h^{n+1/2} g_h^{n+1/2}$\;
$c_h^{n+1} \leftarrow $ solve \eqref{PC-BCFD:c:rewrite} with the obtained value $\rho_h^{n+1/2}$\;
$\mM_x^{n+1},~ \mM_y^{n+1} \leftarrow $ compute \eqref{BE-BCFD:e3} with the obtained value $ c_h^{n+1}$\;
$g_h^{n+1}\leftarrow$ solve \eqref{PC-BCFD:rho:rewrite} using the computed value $\hat{g}_h^n:=\frac{\rho_h^n}{\mM_h^{n+1}}$\;
$\rho_h^{n+1} \leftarrow $ update $\rho_h^{n+1}=\mM_h^{n+1} g_h^{n+1}$\;}
\KwResult{Numerical solutions $\{\rho_h^n\}$ and $\{c_h^n\}$}
\end{algorithm}
\section{Structure-preserving numerical analysis} \label{sec:sp}
In this section,  both the BE-BCFD scheme \eqref{BE-BCFD:rewrite}--\eqref{BE-BCFD:IBc:rewrite} and PC-BCFD scheme \eqref{PC-BCFD:rewrite}--\eqref{PC-BCFD:IBc:rewrite} are analyzed with respect to the discrete mass conservation, positivity preservation and energy dissipation properties.
\subsection{Analysis of the BE-BCFD scheme}
We first show that a discrete version of mass conservation law \eqref{model:ks:MC} holds for the BE-BCFD scheme with respect to the $(\cdot,\cdot)_{\rm M}$ inner product.
\begin{theorem}[Discrete mass conservation]\label{thm:be-mass-cov}
Let $\left\{\rho_h^{n+1}, c_h^{n+1}\right\}$ be solutions to the BE-BCFD scheme \eqref{BE-BCFD:rewrite}--\eqref{BE-BCFD:IBc:rewrite}. Then, there holds
	\begin{equation*}
		\left(\rho_h^{n+1}, 1\right)_{\mathrm{M}}=\left(\rho_h^n, 1\right)_{\mathrm{M}}.
	\end{equation*}
\end{theorem}
\begin{proof}By taking the inner product of \eqref{BE-BCFD:rho:rewrite}  with 1 in the sense of $(\cdot,\cdot)_{\rm M}$, we have
\begin{equation*}
	\begin{aligned}
			\f{\left( \mM_h^{n+1}g_h^{n+1}  , 1\right)_{\mathrm{M}}
            -\left( \mM_h^{n+1}\hat{g}_h^{n} , 1\right)_{\mathrm{M}} }{\tau_{n+1}}
            &=  \left( D_x (\mM_x^{n+1}[d_x g_h^{n+1} ]) + D_y (\mM_x^{n+1}[d_y g_h^{n+1} ]), 1\right)_{\mathrm{M}} \\
		& =-\left(\mM_x^{n+1}[d_x g_h^{n+1} ], d_x 1\right)_x-\left(\mM_y^{n+1}[d_y g_h^{n+1} ], d_y 1\right)_y\\
            &= 0,
	\end{aligned}
\end{equation*}
    where Lemma \ref{lemma:Dd} and the discrete boundary conditions \eqref{BE-BCFD:IBc:rewrite} are applied.
	Thus, we obtain
	\begin{equation*}
		\left(\mM_h^{n+1}g_h^{n+1}, 1\right)_{\mathrm{M}} = \left(\mM_h^{n+1}\hat{g}_h^{n}, 1\right)_{\mathrm{M}},
	\end{equation*}
	which implies the conclusion.
\end{proof}

\begin{theorem}[Positivity-preserving]\label{thm:posi-preserve}	
Let $\left\{\rho_h^{n+1}, c_h^{n+1}\right\}$ be the solutions to the BE-BCFD scheme \eqref{BE-BCFD:rewrite}--\eqref{BE-BCFD:IBc:rewrite}. If $\rho^o > 0$,  then we have $\rho_h^{n+1} >0$; moreover, if $c^o > 0$ and $\rho^o > 0$,  then we further have $c_h^{n+1} >0$.
\end{theorem}
\begin{proof} First, suppose the initial value $\rho^o > 0$ at pointwise level, we discuss the positivity of $\rho_h^{n+1}$ for $n \ge 0$. 
Let $\text{vec}(\cdot): \mathbb{R}^{N_x \times N_y} \to \mathbb{R}^N$  ($N:=N_xN_y$) denote the column-wise vectorization operator. Then, the scheme \eqref{BE-BCFD:rho:rewrite} can be written in a compact matrix form:
\begin{equation}\label{MatrixForm:rho}
	A_\rho\, \text{vec}(g_h^{n+1}) = \text{vec}(\mM_h^{n+1}\circ \hat{g}_h^{n})= \text{vec}(\rho_h^{n}),
\end{equation}
where $\circ$ represents the Hardmard product, and the entries of the $N$-by-$N$ matrix $A_\rho$ is represented as below. 

\paragraph{\indent Case i} If $(x_i,y_j)$ is an interior node of $\Pi_x^*\times \Pi_y^*$, which means $2\leq i \leq N_x-1$, $2\leq j \leq N_y-1$,  then equation \eqref{BE-BCFD:rho:rewrite} reduces to
  \begin{equation}\label{scheme:explicit:interior}
	 \begin{aligned}
	 	&\mM_{h,i,j}^{n+1}g_{h,i,j}^{n+1} \\
	 		&\quad- \tau_{n+1}  \bigg[ \f{\mM_{x,i-1/2,j}^{n+1}}{\Delta x_i \Delta x_{i-1/2}} g_{h,i-1,j}^{n+1}  
                  - \Big( \f{\mM_{x,i-1/2,j}^{n+1}}{\Delta x_i \Delta x_{i-1/2}} + \f{\mM_{x,i+1/2,j}^{n+1}}{\Delta x_i \Delta x_{i+1/2}} \Big) g_{h,i,j}^{n+1} 
                      + \f{\mM_{x,i+1/2,j}^{n+1}}{\Delta x_i \Delta x_{i+1/2}} g_{h,i+1,j}^{n+1} \bigg]\\
	 		&\quad- \tau_{n+1} \bigg[ \f{\mM_{y,i,j-1/2}^{n+1}}{\Delta y_j \Delta y_{j-1/2}} g_{h,i,j-1}^{n+1}  
                  - \Big( \f{\mM_{y,i,j-1/2}^{n+1}}{\Delta y_j \Delta y_{j-1/2}} + \f{\mM_{y,i,j+1/2}^{n+1}}{\Delta y_j \Delta y_{j+1/2}} \Big) g_{h,i,j}^{n+1} + \f{\mM_{y,i,j+1/2}^{n+1}}{\Delta y_j \Delta y_{j+1/2}} g_{h,i,j+1}^{n+1} \bigg]\\
	 	&= \mM_{h,i,j}^{n+1}\hat{g}_{h,i,j}^{n}.
	 	\end{aligned}
	 \end{equation}
In this case, the coefficients of $\{g_{h,i,j}^{n+1}\}$ in \eqref{scheme:explicit:interior} correspond to the $(\kappa:= i + N_x (j-1) )$th row of the matrix $A_\rho$, and
\begin{equation*}
A_\rho(\kappa,\ell)=\left\{
	\begin{aligned}
	&- \tau_{n+1} \f{ \mM_{y,i, j-1/2}^{n+1} }{\Delta y_j \Delta y_{j-1/2}}, \quad \ell=\kappa-N_x,\\
        &- \tau_{n+1} \f{\mM_{x, i-1/2,j}^{n+1}}{\Delta x_i \Delta x_{i-1/2}},\quad \ell=\kappa-1,\\
	& \mM_{h,i,j}^{n+1} + \tau_{n+1} \Big( \f{\mM_{x,i-1/2,j}^{n+1}}{\Delta x_i \Delta x_{i-1/2}} + \f{\mM_{x,i+1/2,j}^{n+1}}{\Delta x_i \Delta x_{i+1/2}} + \f{\mM_{y,i,j-1/2}^{n+1}}{\Delta y_j \Delta y_{j-1/2}} + \f{\mM_{y,i,j+1/2}^{n+1}}{\Delta y_j \Delta y_{j+1/2}}\Big), \quad \ell=\kappa,\\
	&- \tau_{n+1} \f{\mM_{x,i+1/2,j}^{n+1}}{\Delta x_i \Delta x_{i+1/2}}, \quad \ell=\kappa+1,\\
	& - \tau_{n+1} \f{\mM_{y,i,j+1/2}^{n+1}}{\Delta y_j \Delta y_{j+1/2}},\quad \ell=\kappa+N_x.
	\end{aligned}
    \right.
\end{equation*}


\paragraph{\indent Case ii}   If $(x_i,y_j)$ is an edge node of $\Pi_x^*\times \Pi_y^*$ such that the edge parallel to $y$-axis, which means $ i = 1~\text{or} ~N_x$, $2\leq j \leq N_y-1$, then equation \eqref{BE-BCFD:rho:rewrite} with boundary condition \eqref{BE-BCFD:IBc:rewrite} reduces to
    \begin{equation}\label{scheme:explicit:edge-x1}
    	\begin{aligned}
    	&\mM_{h,1,j}^{n+1}g_{h,1,j}^{n+1} - \tau_{n+1} \bigg[ - \f{\mM_{x,3/2,j}^{n+1}}{\Delta x_1 \Delta x_{3/2}} g_{h,1,j}^{n+1} + \f{\mM_{x,3/2,j}^{n+1}}{\Delta x_1 \Delta x_{3/2}} g_{h,2,j}^{n+1} \bigg]\\
    	&\quad- \tau_{n+1} \bigg[ \f{\mM_{y,1,j-1/2}^{n+1}}{\Delta y_j \Delta y_{j-1/2}} g_{h,1,j-1}^{n+1}  
    	- \Big( \f{\mM_{y,1,j-1/2}^{n+1}}{\Delta y_j \Delta y_{j-1/2}} + \f{\mM_{y,1,j+1/2}^{n+1}}{\Delta y_j \Delta y_{j+1/2}} \Big) g_{h,1,j}^{n+1} 
    	+ \f{\mM_{y,1,j+1/2}^{n+1}}{\Delta y_j \Delta y_{j+1/2}} g_{h,1,j+1}^{n+1} \bigg]\\
    &= \mM_{h,i,j}^{n+1}\hat{g}_{h,1,j}^{n}, \quad \text{for}~ i = 1,
    	\end{aligned}
    \end{equation}
and 
\begin{equation}\label{scheme:explicit:edge-xNx}
	\begin{aligned}
		&\mM_{h,N_x,j}^{n+1}g_{h,N_x,j}^{n+1} - \tau_{n+1} \bigg[ \f{\mM_{x,N_x-1/2,j}^{n+1}}{\Delta x_{N_x} \Delta x_{N_x-1/2} } g_{h,N_x-1,j}^{n+1}  -  \f{\mM_{x,N_x-1/2,j}^{n+1}}{\Delta x_{N_x} \Delta x_{N_x-1/2}}   g_{h,N_x,j}^{n+1} \bigg]\\
		&\quad- \tau_{n+1} \bigg[ \f{\mM_{y,N_x,j-1/2}^{n+1}}{\Delta y_j \Delta y_{j-1/2}} g_{h,N_x,j-1}^{n+1}  
		- \Big( \f{\mM_{y,N_x,j-1/2}^{n+1}}{\Delta y_j \Delta y_{j-1/2}} + \f{\mM_{y,N_x,j+1/2}^{n+1}}{\Delta y_j \Delta y_{j+1/2}} \Big) g_{h,N_x,j}^{n+1} + \f{\mM_{y,N_x,j+1/2}^{n+1}}{\Delta y_j \Delta y_{j+1/2}} g_{h,N_x,j+1}^{n+1} \bigg]\\
		&= \mM_{h,N_x,j}^{n+1}\hat{g}_{h,N_x,j}^{n}, \quad \text{for}~ i = N_x.
	\end{aligned}
\end{equation}
In this case, the coefficients in \eqref{scheme:explicit:edge-x1} correspond to the $(\kappa = 1 + N_x (j-1) )$th row of the matrix $A_\rho$, and
\begin{equation*}
A_\rho(\kappa,\ell)=\left\{
	\begin{aligned}
	&- \tau_{n+1} \f{ \mM_{y,1, j-1/2}^{n+1} }{\Delta y_j \Delta y_{j-1/2}}, \quad \ell=\kappa-N_x,\\
	& \mM_{h,1,j}^{n+1} + \tau_{n+1} \Big(  \f{\mM_{x,3/2,j}^{n+1}}{\Delta x_1 \Delta x_{3/2}} + \f{\mM_{y,1,j-1/2}^{n+1}}{\Delta y_j \Delta y_{j-1/2}} + \f{\mM_{y,1,j+1/2}^{n+1}}{\Delta y_j \Delta y_{j+1/2}}\Big), \quad \ell=\kappa,\\
	&- \tau_{n+1} \f{\mM_{x,3/2,j}^{n+1}}{\Delta x_1 \Delta x_{3/2}}, \quad \ell=\kappa+1,\\
	& - \tau_{n+1} \f{\mM_{y,1,j+1/2}^{n+1}}{\Delta y_j \Delta y_{j+1/2}},\quad \ell=\kappa+N_x,
	\end{aligned}
    \right.
\end{equation*}
and the coefficients in \eqref{scheme:explicit:edge-xNx} correspond to the $(\kappa = N_x j )$th row of the matrix $A_\rho$, and
\begin{equation*}
A_\rho(\kappa,\ell)=\left\{
	\begin{aligned}
	&- \tau_{n+1} \f{ \mM_{y,N_x, j-1/2}^{n+1} }{\Delta y_j \Delta y_{j-1/2}}, \quad \ell=\kappa-N_x,\\
        &- \tau_{n+1} \f{\mM_{x, N_x-1/2,j}^{n+1}}{\Delta x_{N_x} \Delta x_{N_x-1/2}},\quad \ell=\kappa-1,\\
	& \mM_{h,N_x,j}^{n+1} + \tau_{n+1} \Big( \f{\mM_{x,N_x-1/2,j}^{n+1}}{\Delta x_{N_x} \Delta x_{N_x-1/2}} + \f{\mM_{y,N_x,j-1/2}^{n+1}}{\Delta y_j \Delta y_{j-1/2}} + \f{\mM_{y,N_x,j+1/2}^{n+1}}{\Delta y_j \Delta y_{j+1/2}}\Big), \quad \ell=\kappa,\\
	& - \tau_{n+1} \f{\mM_{y,N_x,j+1/2}^{n+1}}{\Delta y_j \Delta y_{j+1/2}},\quad \ell=\kappa+N_x.
	\end{aligned}
    \right.
\end{equation*}

\paragraph{\indent Case iii} If $(x_i,y_j)$ is an edge node of $\Pi_x^*\times \Pi_y^*$ such that the edge parallel to $x$-axis, which means $2\leq i \leq N_x-1$, $ j = 1~\text{or} ~N_y$, then equation \eqref{BE-BCFD:rho:rewrite} with boundary condition \eqref{BE-BCFD:IBc:rewrite} reduces to
\begin{equation}\label{scheme:explicit:edge-y1}
	\begin{aligned}
		&\mM_{h,i,1}^{n+1}g_{h,i,1}^{n+1} \\
		&\quad- \tau_{n+1} \Big[ \f{\mM_{x,i-1/2,1}^{n+1}}{\Delta x_i \Delta x_{i-1/2}} g_{h,i-1,1}^{n+1}  - \Big( \f{\mM_{x,i-1/2,1}^{n+1}}{\Delta x_i \Delta x_{i-1/2}} + \f{\mM_{x,i+1/2,1}^{n+1}}{\Delta x_i \Delta x_{i+1/2}} \Big) g_{h,i,1}^{n+1} 
		+ \f{\mM_{x,i+1/2,j}^{n+1}}{\Delta x_i \Delta x_{i+1/2}} g_{h,i+1,1}^{n+1} \Big]\\
		&\quad- \tau_{n+1} \Big[ - \f{\mM_{y,i,3/2}^{n+1}}{\Delta y_1 \Delta y_{3/2}}  g_{h,i,1}^{n+1} + \f{\mM_{y,i,3/2}^{n+1}}{\Delta y_1 \Delta y_{3/2}} g_{h,i,2}^{n+1} \Big] = \mM_{h,i,1}^{n+1}\hat{g}_{h,i,1}^{n}, \quad \text{for}~ j = 1,
	\end{aligned}
\end{equation}
and
\begin{equation}\label{scheme:explicit:edge-yNy}
	\begin{aligned}
		&\mM_{h,i,N_y}^{n+1}g_{h,i,N_y}^{n+1} \\
		&\quad- \tau_{n+1} \Big[ \f{\mM_{x,i-1/2,N_y}^{n+1}}{\Delta x_i \Delta x_{i-1/2}} g_{h,i-1,N_y}^{n+1}  - \Big( \f{\mM_{x,i-1/2,N_y}^{n+1}}{\Delta x_i \Delta x_{i-1/2}} + \f{\mM_{x,i+1/2,N_y}^{n+1}}{\Delta x_i \Delta x_{i+1/2}} \Big) g_{h,i,N_y}^{n+1} + \f{\mM_{x,i+1/2,N_y}^{n+1}}{\Delta x_i \Delta x_{i+1/2}} g_{h,i+1,N_y}^{n+1} \Big]\\
		&\quad- \tau_{n+1} \Big[ \f{\mM_{y,i,N_y-1/2}^{n+1}}{\Delta y_{N_y} \Delta y_{N_y-1/2}} g_{h,i,N_y-1}^{n+1}  -  \f{\mM_{y,i,N_y-1/2}^{n+1}}{\Delta y_{N_y} \Delta y_{N_y-1/2}}  g_{h,i,N_y}^{n+1} \Big] = \mM_{h,i,N_y}^{n+1}\hat{g}_{h,i,N_y}^{n}, \quad \text{for}~ j = N_y.
	\end{aligned}
\end{equation}
In this case, the coefficients in \eqref{scheme:explicit:edge-y1} correspond to the $(\kappa = i )$th row of the matrix $A_\rho$, and
\begin{equation*}
A_\rho(\kappa,\ell)=\left\{
	\begin{aligned}
        &- \tau_{n+1} \f{\mM_{x, i-1/2,1}^{n+1}}{\Delta x_i \Delta x_{i-1/2}},\quad \ell=\kappa-1,\\
	& \mM_{h,i,1}^{n+1} + \tau_{n+1} \Big( \f{\mM_{x,i-1/2,1}^{n+1}}{\Delta x_i \Delta x_{i-1/2}} + \f{\mM_{x,i+1/2,1}^{n+1}}{\Delta x_i \Delta x_{i+1/2}} + \f{\mM_{y,i,3/2}^{n+1}}{\Delta y_1 \Delta y_{3/2}}\Big), \quad \ell=\kappa,\\
	&- \tau_{n+1} \f{\mM_{x,i+1/2,1}^{n+1}}{\Delta x_i \Delta x_{i+1/2}}, \quad \ell=\kappa+1,\\
	& - \tau_{n+1} \f{\mM_{y,i,3/2}^{n+1}}{\Delta y_1 \Delta y_{3/2}},\quad \ell=\kappa+N_x.
	\end{aligned}
    \right.
\end{equation*}
and the coefficients in \eqref{scheme:explicit:edge-yNy} correspond to the $(\kappa = i + N_x (N_y-1) )$th row of the matrix $A_\rho$, and
\begin{equation*}
A_\rho(\kappa,\ell)=\left\{
	\begin{aligned}
	&- \tau_{n+1} \f{ \mM_{y,i, N_y-1/2}^{n+1} }{\Delta y_{N_y} \Delta y_{N_y-1/2}}, \quad \ell=\kappa-N_x,\\
        &- \tau_{n+1} \f{\mM_{x, i-1/2,N_y}^{n+1}}{\Delta x_i \Delta x_{i-1/2}},\quad \ell=\kappa-1,\\
	& \mM_{h,i,N_y}^{n+1} + \tau_{n+1} \Big( \f{\mM_{x,i-1/2,N_y}^{n+1}}{\Delta x_i \Delta x_{i-1/2}} + \f{\mM_{x,i+1/2,N_y}^{n+1}}{\Delta x_i \Delta x_{i+1/2}} + \f{\mM_{y,i,N_y-1/2}^{n+1}}{\Delta y_{N_y} \Delta y_{N_y-1/2}} \Big), \quad \ell=\kappa,\\
	&- \tau_{n+1} \f{\mM_{x,i+1/2,N_y}^{n+1}}{\Delta x_i \Delta x_{i+1/2}}, \quad \ell=\kappa+1.
	\end{aligned}
    \right.
\end{equation*}

  \paragraph{\indent Case iv} If $(x_i,y_j)$ is a corner node of $\Pi_x^*\times \Pi_y^*$, which means $ i = 1~\text{or} ~N_x$, $ j = 1~\text{or} ~N_y$, then equation \eqref{BE-BCFD:rho:rewrite} with boundary condition \eqref{BE-BCFD:IBc:rewrite} reduces to
\begin{equation}\label{scheme:explicit:corner11}
	\begin{aligned}
		&\mM_{h,1,1}^{n+1}g_{h,1,1}^{n+1} 
	    - \tau_{n+1} \Big[ - \f{\mM_{x,3/2,1}^{n+1}}{\Delta x_1 \Delta x_{3/2}} g_{h,1,1}^{n+1} + \f{\mM_{x,3/2,j}^{n+1}}{\Delta x_1 \Delta x_{3/2}} g_{h,2,1}^{n+1} \Big]\\
		&\qquad- \tau_{n+1} \Big[ - \f{\mM_{y,1,3/2}^{n+1}}{\Delta y_1 \Delta y_{3/2}}  g_{h,1,1}^{n+1} + \f{\mM_{y,1,3/2}^{n+1}}{\Delta y_1 \Delta y_{3/2}} g_{h,1,2}^{n+1} \Big]
		= \mM_{h,1,1}^{n+1}\hat{g}_{h,1,1}^{n}, 
	\end{aligned}
\end{equation}
for~ $i = j= 1$; and
\begin{equation}\label{scheme:explicit:corner1Ny}
	\begin{aligned}
		&\mM_{h,1,N_y}^{n+1}g_{h,1,N_y}^{n+1} 
		- \tau_{n+1} \Big[ - \f{\mM_{x,3/2,N_y}^{n+1}}{\Delta x_1 \Delta x_{3/2}} g_{h,1,N_y}^{n+1} + \f{\mM_{x,3/2,j}^{n+1}}{\Delta x_1 \Delta x_{3/2}} g_{h,2,N_y}^{n+1} \Big]\\
		&\qquad- \tau_{n+1} \Big[ \f{\mM_{y,1,N_y-1/2}^{n+1}}{\Delta y_{N_y} \Delta y_{N_y-1/2}} g_{h,1,N_y-1}^{n+1}  -  \f{\mM_{y,1,N_y-1/2}^{n+1}}{\Delta y_{N_y} \Delta y_{N_y-1/2}}  g_{h,1,N_y}^{n+1} \Big]
		= \mM_{h,1,N_y}^{n+1}\hat{g}_{h,1,N_y}^{n}, 
	\end{aligned}
\end{equation}
for~ $i = 1,~ j=N_y$; and
\begin{equation}\label{scheme:explicit:cornerNx1}
	\begin{aligned}
		&\mM_{h,N_x,1}^{n+1}g_{h,N_x,1}^{n+1} 
		- \tau_{n+1} \Big[ \f{\mM_{x,N_x-1/2,1}^{n+1}}{\Delta x_{N_x} \Delta x_{N_x-1/2} } g_{h,N_x-1,1}^{n+1}  -  \f{\mM_{x,N_x-1/2,1}^{n+1}}{\Delta x_{N_x} \Delta x_{N_x-1/2}}   g_{h,N_x,1}^{n+1} \Big]\\
		&\qquad- \tau_{n+1} \Big[ - \f{\mM_{y,1,3/2}^{n+1}}{\Delta y_1 \Delta y_{3/2}}  g_{h,N_x,1}^{n+1} + \f{\mM_{y,N_x,3/2}^{n+1}}{\Delta y_1 \Delta y_{3/2}} g_{h,N_x,2}^{n+1} \Big]
		= \mM_{h,N_x,1}^{n+1}\hat{g}_{h,N_x,1}^{n}, 
	\end{aligned}
\end{equation}
for $i = N_x,~ j=1$; and
\begin{equation}\label{scheme:explicit:cornerNxNy}
	\begin{aligned}
		&\mM_{h,N_x,N_y}^{n+1}g_{h,N_x,N_y}^{n+1} 
		- \tau_{n+1} \Big[ \f{\mM_{x,N_x-1/2,N_y}^{n+1}}{\Delta x_{N_x} \Delta x_{N_x-1/2} } g_{h,N_x-1,N_y}^{n+1}  -  \f{\mM_{x,N_x-1/2,N_y}^{n+1}}{\Delta x_{N_x} \Delta x_{N_x-1/2}}   g_{h,N_x,N_y}^{n+1} \Big]\\
		&\qquad- \tau_{n+1} \Big[ \f{\mM_{y,N_x,N_y-1/2}^{n+1}}{\Delta y_{N_y} \Delta y_{N_y-1/2}} g_{h,N_x,N_y-1}^{n+1}  -  \f{\mM_{y,N_x,N_y-1/2}^{n+1}}{\Delta y_{N_y} \Delta y_{N_y-1/2}}  g_{h,N_x,N_y}^{n+1} \Big]
		= \mM_{h,N_x,N_y}^{n+1}\hat{g}_{h,N_x,N_y}^{n}, 
	\end{aligned}
\end{equation}
for $i = N_x,~ j=N_y$.
In this case, the coefficients for the above four equations \eqref{scheme:explicit:corner11}--\eqref{scheme:explicit:cornerNxNy} correspond to the $(\kappa = 1,\ N_x,\ 1+N_x(N_y-1)\ \text{and}\ N_xN_y )$th row of the matrix $A_\rho$, respectively, and
\begin{equation*}
A_\rho(1,\ell)=\left\{
	\begin{aligned}
	& \mM_{h,1,1}^{n+1} + \tau_{n+1} \Big(  \f{\mM_{x,3/2,1}^{n+1}}{\Delta x_1 \Delta x_{3/2}} + \f{\mM_{y,1,3/2}^{n+1}}{\Delta y_1 \Delta y_{3/2}}\Big), \quad \ell=1,\\
	&- \tau_{n+1} \f{\mM_{x,3/2,1}^{n+1}}{\Delta x_1 \Delta x_{3/2}}, \quad \ell=2,\\
	& - \tau_{n+1} \f{\mM_{y,1,3/2}^{n+1}}{\Delta y_1 \Delta y_{3/2}},\quad \ell=1+N_x,
	\end{aligned}
    \right.
\end{equation*}
\begin{equation*}
A_\rho(N_x,\ell)=\left\{
	\begin{aligned}
        &- \tau_{n+1} \f{\mM_{x, N_x-1/2,j}^{n+1}}{\Delta x_{N_x} \Delta x_{N_x-1/2}},\quad \ell=N_x-1,\\
	& \mM_{h,N_x,1}^{n+1} + \tau_{n+1} \Big( \f{\mM_{x,N_x-1/2,1}^{n+1}}{\Delta x_{N_x} \Delta x_{N_x-1/2}}  + \f{\mM_{y,N_x,3/2}^{n+1}}{\Delta y_1 \Delta y_{3/2}}\Big), \quad \ell=N_x,\\
	& - \tau_{n+1} \f{\mM_{y,N_x,3/2}^{n+1}}{\Delta y_1 \Delta y_{3/2}},\quad \ell=2 N_x,
	\end{aligned}
    \right.
\end{equation*}
\begin{equation*}
A_\rho(1+N_x(N_y-1),\ell)=\left\{
	\begin{aligned}
	&- \tau_{n+1} \f{ \mM_{y,1, N_y-1/2}^{n+1} }{\Delta y_{N_y} \Delta y_{N_y-1/2}}, \quad \ell=1+N_x(N_y-2),\\
	& \mM_{h,1,N_y}^{n+1} + \tau_{n+1} \Big( \f{\mM_{x,3/2,N_y}^{n+1}}{\Delta x_1 \Delta x_{3/2}} + \f{\mM_{y,1,N_y-1/2}^{n+1}}{\Delta y_{N_y} \Delta y_{N_y-1/2}} \Big), \quad \ell=1+N_x(N_y-1),\\
	&- \tau_{n+1} \f{\mM_{x,3/2,N_y}^{n+1}}{\Delta x_1 \Delta x_{3/2}}, \quad \ell=2+N_x(N_y-1).
	\end{aligned}
    \right.
\end{equation*}
\begin{equation*}
A_\rho(N_xN_y,\ell)=\left\{
	\begin{aligned}
	&- \tau_{n+1} \f{ \mM_{y,N_x, N_y-1/2}^{n+1} }{\Delta y_{N_y} \Delta y_{N_y-1/2}}, \quad \ell=N_x(N_y-1),\\
        &- \tau_{n+1} \f{\mM_{x, N_x-1/2,N_y}^{n+1}}{\Delta x_{N_x} \Delta x_{N_x-1/2}},\quad \ell=N_xN_y-1,\\
	& \mM_{h,N_x,N_y}^{n+1} + \tau \Big( \f{\mM_{x,N_x-1/2,N_y}^{n+1}}{\Delta x_{N_x} \Delta x_{N_x-1/2}} + \f{\mM_{y,N_x,N_y-1/2}^{n+1}}{\Delta y_{N_y} \Delta y_{N_y-1/2}} \Big), \quad \ell=N_xN_y.
	\end{aligned}
    \right.
\end{equation*}

Let $\bm{1}$ be an $N$-dimensional column vector with all entries being one. Then, from above discussions, it can be easily verified that 
\begin{equation*}
	A_\rho \bm{1} = \text{vec}( \mM_h^{n+1}) >0,
\end{equation*}
as $\mM_{h,i,j}^{n+1} = {\rm e}^{c_{h,i,j}^{n+1}} > 0$ for all $i, j$.  Thus, Lemma \ref{lem:M-matrix} implies that $A_\rho$ is a non-singular M-matrix, which by Lemma \ref{lem:M-matrix:inverse} further implies $A_\rho^{-1} \geq 0$. Therefore, if $\rho_h^0=\rho^o > 0$ on $\Pi_x^*\times \Pi_y^*$, from \eqref{MatrixForm:rho} and by deduction we have 
$\text{vec}( g_h^{n+1}) >0$, and thus $\rho_h^{n+1} = \mM_h^{n+1}g_h^{n+1}>0$.

Next, we conclude from \eqref{BE-BCFD:c:rewrite} that
\begin{equation}\label{scheme:c:rewrite}
 (\varepsilon + \alpha\tau_{n+1})  c_{h}^{n+1} -\tau_{n+1} D_x (d_x c_h^{n+1})
                 -\tau_{n+1} D_y (d_y c_h^{n+1}) =\varepsilon c_{h}^{n}+ \tau_{n+1} \rho_{h}^{n}.
\end{equation}
Similarly as above, but more simple, we can show that the corresponding coefficient matrix of \eqref{scheme:c:rewrite} is also a non-singular M-matrix. Therefore, if $c_h^0=c^o > 0$ and $\rho_h^0=\rho^o > 0$ on $\Pi_x^*\times \Pi_y^*$, by deduction we directly have $c_h^{n+1} >0$.
\end{proof}

\begin{remark}\label{rmk:positivity:ellip}
In the parabolic--elliptic case $\varepsilon=0$, the discrete concentration equation in the BE-BCFD scheme reduces to a steady elliptic problem of the form
\[
- D_x(d_x c_h^{n+1}) - D_y(d_y c_h^{n+1}) + \alpha c_h^{n+1} = \rho_h^n.
\]
It is easy to verify that the associated coefficient matrix is a nonsingular M-matrix, and hence the positivity of $\rho_h^n$ directly implies $c_h^{n+1}>0$.
\end{remark}

\begin{remark} In fact,	$A_\rho$ can be written in the following form
	\begin{equation}\label{matrix:splitting}
		A_\rho = M_d + S
	\end{equation}
where $M_d := {\rm diag}(\text{vec}( \mM_h^{n+1}))$ is a diagonal matrix and $S \bm{1} = \bm{0}$.
\end{remark}

Next, define the discrete counterpart of the continuous energy \eqref{energy:continue:rewrite} as
\begin{equation}\label{discrete-energy}
  E^n := \Big(\rho_h^n \log ( \f{\rho_h^n}{\mM_h^n}) - \rho_h^n , 1 \Big)_{\mathrm{M}} + \f12 \|\bm{d} c_h^n\|_{\mathrm{TM}}^2  + \f12\|c_h^n\|_{\mathrm{M}}^2.
\end{equation}
We show that the fully discrete BE-BCFD scheme \eqref{BE-BCFD:rewrite}--\eqref{BE-BCFD:IBc:rewrite} unconditionally decays this energy \eqref{discrete-energy}.

\begin{theorem}[Discrete energy dissipation]\label{thm:energy-dissipasion}
	Let $\left\{\rho_h^{n+1}, c_h^{n+1}\right\}$ be the solutions to the BE-BCFD scheme \eqref{BE-BCFD:rewrite}--\eqref{BE-BCFD:IBc:rewrite}. Then, there holds
	$$
		E^{n+1}\leq E^n.
	$$
\end{theorem}
\begin{proof} A direct calculation shows that
\begin{equation}\label{discrete-energy:diff}
		\begin{aligned}
		E^{n+1} - E^n &= \Big(\rho_h^{n+1} \log ( \f{\rho_h^{n+1}}{\mM_h^{n+1}}) -\rho_h^n \log ( \f{\rho_h^n}{\mM_h^n}), 1 \Big)_{\mathrm{M}}\\
		&\qquad+ \f12 \left[\left( \|\bm{d} c_h^{n+1}\|_{\mathrm{TM}}^2 +\|c_h^{n+1}\|_{\mathrm{M}}^2  \right)  -  \left( \|\bm{d} c_h^n\|_{\mathrm{TM}}^2+\| c_h^n\|_{\mathrm{TM}}^2\right) \right]
		 =: \sum_{i=1}^2 I_i,
		\end{aligned}
\end{equation}
where the mass conservation law as shown in Theorem \ref{thm:be-mass-cov} is applied.
    
Next, we estimate the right-hand side of \eqref{discrete-energy:diff} term by term. Firstly, due to the definition in \eqref{BE-BCFD:e3}, the $I_1$-term in \eqref{discrete-energy:diff} can be rewritten as
 \begin{equation}\label{I1}
 	\begin{aligned}
 		I_1 &= \Big(\rho_h^{n+1} \log ( \f{\rho_h^{n+1}}{\mM_h^{n+1}}) , 1 \Big)_{\mathrm{M}} - \Big(\rho_h^{n} \log ( \f{\rho_h^n}{\mM_h^{n+1}}), 1 \Big)_{\mathrm{M}} + \left(\rho_h^n, c_h^n - c_h^{n+1} \right)_{\mathrm{M}}\\
 		&=\left(\mM_h^{n+1}g_h^{n+1} \log g_h^{n+1}, 1 \right)_{\mathrm{M}} - \left(\mM_h^{n+1}\hat{g}_h^{n} \log \hat{g}_h^{n}, 1 \right)_{\mathrm{M}} + \left(\rho_h^n, c_h^n - c_h^{n+1} \right)_{\mathrm{M}}.
 	\end{aligned}
 \end{equation}
Let $W_d := {\rm diag}( \text{vec}(\bm{h_x}^\top \otimes \bm{h_y}) )$ be an $N$-by-$N$ diagonal matrix. Let $B=(b_{i, j})$ with entry $b_{i,j} := g_{h,i,j}^{n+1} \log ( g_{h,i,j}^{n+1} )$ be an $N_x$-by-$N_y$ matrix. Thus, we have
\begin{equation}\label{I11:equvai}
	\begin{aligned}
	\left(\mM_h^{n+1}g_h^{n+1} \log g_h^{n+1}, 1 \right)_{\mathrm{M}} 
	&= \sum_{i=1}^{N_x} \sum_{j=1}^{N_y} \Delta x_i \Delta y_j \mM_{h,i,j}^{n+1} b_{i,j} 
	= \bm{1}^\top W_d M_d\, \text{vec}(B).
	\end{aligned}
\end{equation}
Note that by \eqref{matrix:splitting}, equation \eqref{MatrixForm:rho} is equivalent to
\begin{equation*} 
   \tilde{A}_\rho~\text{vec}(g_h^{n+1}) = \text{vec}(\hat{g}_h^{n}), \quad \tilde{A}_\rho := I + M_d^{-1} S.
\end{equation*}
It is easy to check that $\tilde{A}_\rho \bm{1} = \bm{1}$, as the discrete diffusion operator preserves
constants, i.e., $S\bm{1}=\bm{0}$.
Therefore, $\tilde{A}_\rho$ is also a non-singular M-matrix. Consequently, Lemma \ref{lem:M-matrix:inverse} implies that $\tilde{A}_\rho^{-1}\geq 0$. Moreover, it is easy to verify that $\tilde{A}_\rho^{-1} \bm{1} = \bm{1}$, which further implies that 
$
    \text{vec}(g_h^{n+1})=\tilde A_\rho^{-1} \text{vec}(\hat{g}_h^{n})
$ defines a row convex combination of $\text{vec}(\hat{g}_h^{n})$. Furthermore, as the function $x\log x$ is convex for $0< x <\infty$, so by Jensen’s inequality and the above discussion, we immediately have
\begin{equation}\label{convex}
	\text{vec}(B) \leq \tilde{A}_\rho^{-1} \text{vec}(\hat{B}),
\end{equation}
where $\hat{B}=(\hat{b}_{i, j})$ is a matrix of size $N_x\times N_y$ with entry $\hat{b}_{i,j} := \hat{g}_{h,i,j}^{n} \log ( \hat{g}_{h,i,j}^{n} )$. 

In addition, note that the fact $\bm{1}^\top W_d S = \bm{0}^\top$, we see
\begin{equation*}
	\bm{1}^\top W_d M_d \tilde{A}_\rho = \bm{1}^\top W_d M_d (I + M_d^{-1} S  ) = \bm{1}^\top W_d M_d  \Longrightarrow \bm{1}^\top W_d M_d \tilde{A}_\rho^{-1} = \bm{1}^\top W_d M_d,
\end{equation*}
and therefore, combining \eqref{I11:equvai}--\eqref{convex}, we obtain
\begin{equation}\label{I11:result}
	\begin{aligned}
	 \left(\mM_h^{n+1}g_h^{n+1} \log g_h^{n+1}, 1 \right)_{\mathrm{M}} 
		&\leq  \bm{1}^\top W_d M_d \tilde{A}_\rho^{-1} \text{vec}(\hat{B}) \\
		& = \bm{1}^\top W_d M_d \text{vec}(\hat{B}) 
		= \left(\mM_h^{n+1}\hat{g}_h^{n} \log \hat{g}_h^{n}, 1 \right)_{\mathrm{M}}.
	\end{aligned}
\end{equation}
Thus, inserting \eqref{I11:result} into \eqref{I1}, the first term can be estimated as
\begin{equation}\label{I1:result}
	I_1 \leq \left(\rho_h^n, c_h^n - c_h^{n+1} \right)_{\mathrm{M}}.
\end{equation}

Secondly, we shall pay attention to the estimate of $I_2$, which is related to the concentration $c_h$. Taking the discrete inner product of \eqref{BE-BCFD:c:rewrite} with $c_h^{n+1}-c_h^n$, with the help of Lemma \ref{lemma:Dd}, we obtain
\begin{equation}\label{energy:c}
	\begin{aligned}
	\left(d_t c_{h}^{n+1}, c_h^{n+1}-c_h^n \right)_{\mathrm{M}}
	&= \left(D_x (d_x c_h^{n+1}) + D_y (d_y c_h^{n+1}), c_h^{n+1}-c_h^n \right)_{\mathrm{M}} \\
	&\qquad-  \left(c_{h}^{n+1}, c_h^{n+1}-c_h^n \right)_{\mathrm{M}} + \left(\rho_{h}^{n}, c_h^{n+1}-c_h^n \right)_{\mathrm{M}} \\
	& = -\left(d_x c_h^{n+1}, d_x (c_h^{n+1}-c_h^n) \right)_x - \left(d_y c_h^{n+1}, d_y (c_h^{n+1}-c_h^n) \right)_y\\
	&\qquad-  \left(c_{h}^{n+1}, c_h^{n+1}-c_h^n \right)_{\mathrm{M}} + \left(\rho_{h}^{n}, c_h^{n+1}-c_h^n \right)_{\mathrm{M}}.
	\end{aligned}
\end{equation}
As  $a(a-b)\ge \f12 (a^2-b^2)$, the first three terms on the right-hand side of \eqref{energy:c} can be estimated respectively as
\begin{equation}\label{energy:c:2}
\begin{aligned}
	-\left(d_x c_h^{n+1}, d_x (c_h^{n+1}-c_h^n) \right)_x
	 &\leq - \frac{1}{2}\left(\|d_x c_h^{n+1}\|_x^2-\|d_x c_h^n\|_x^2\right), \\
	- \left(d_y c_h^{n+1}, d_y (c_h^{n+1}-c_h^n) \right)_y  
	&\leq - \frac{1}{2}\left(\|d_y c_h^{n+1}\|_y^2-\|d_y c_h^n\|_y^2\right),\\
	-  \left(c_{h}^{n+1}, c_h^{n+1}-c_h^n \right)_{\mathrm{M}}
	&\leq - \frac{1}{2}\left(\| c_h^{n+1}\|_{\mathrm{M}}^2-\| c_h^n\|_{\mathrm{M}}^2\right).
\end{aligned}
\end{equation}
Therefore, inserting \eqref{energy:c:2} into \eqref{energy:c}, and rearranging the resulting equation, we have
\begin{equation}\label{I2}
	\begin{aligned}
	I_2&=\f12 \left( \|\bm{d} c_h^{n+1}\|_{\mathrm{TM}}^2 - \|\bm{d} c_h^n\|_{\mathrm{M}}^2 \right) 
       + \f12\left( \|c_h^{n+1}\|_{\mathrm{M}}^2 -\| c_h^n\|_{\mathrm{TM}}^2\right) \\
	& \leq -\left(d_t c_{h}^{n+1}, c_h^{n+1}-c_h^n \right)_{\mathrm{M}} + \left(\rho_{h}^{n}, c_h^{n+1}-c_h^n \right)_{\mathrm{M}}\\
    & =-\tau_{n+1}\|d_t c_{h}^{n+1}\|_{\mathrm{M}}^2 + \left(\rho_{h}^{n}, c_h^{n+1}-c_h^n \right)_{\mathrm{M}}.
	\end{aligned}
\end{equation}

Finally, combining \eqref{I1:result} and \eqref{I2} with \eqref{discrete-energy:diff}, we immediately get
\begin{equation*}
	E^{n+1} - E^n 	\leq -\tau_{n+1} \|d_t c_{h}^{n+1}\|_{\mathrm{M}}^2 \le 0,
\end{equation*}
which proves the conclusion.
\end{proof}
\subsection{Analysis of the PC-BCFD scheme}
This subsection addresses the structure-preserving properties of the second-order PC-BCFD scheme \eqref{PC-BCFD:rewrite}–\eqref{PC-BCFD:IBc:rewrite}, where only discrete mass conservation and conditional positivity preservation are proved.

\begin{theorem}[Discrete mass conservation]\label{thm:pc-mass-cov}
Let $\big\{\rho_h^{n+1}, c_h^{n+1}, \rho_h^{n+1/2}, c_h^{n+1/2}\big\}$ be the solutions to the PC-BCFD scheme \eqref{PC-BCFD:rewrite}--\eqref{PC-BCFD:IBc:rewrite}. Then, there holds
	\begin{equation*}
	\left(\rho_h^{n+1}, 1\right)_{\mathrm{M}} = (\rho_h^{n+1/2}, 1)_{\mathrm{M}} 
    = \left(\rho_h^n, 1\right)_{\mathrm{M}}.
	\end{equation*}
\end{theorem}
\begin{proof}  Similar to Theorem \ref{thm:be-mass-cov}, by taking the inner product on both sides of equations \eqref{PC-BCFD:rho_half:rewrite} and \eqref{PC-BCFD:rho:rewrite} with 1, respectively, we get
\begin{equation*}
	\big(\mM_h^{n+1/2}g_h^{n+1/2}, 1 \big)_{\mathrm{M}} = \big(\mM_h^{n+1/2}\check{g}_h^{n}, 1\big)_{\mathrm{M}}\quad 
	\text{and}\quad \left(\mM_h^{n+1}g_h^{n+1}, 1\right)_{\mathrm{M}} = \left(\mM_h^{n+1}\hat{g}_h^{n}, 1\right)_{\mathrm{M}},
\end{equation*}
which imply that
\begin{equation*}
	\big(\rho_h^{n+1/2}, 1 \big)_{\mathrm{M}} = \big(\rho_h^{n}, 1 \big)_{\mathrm{M}}\quad 
		\text{and}\quad \left(\rho_h^{n+1}, 1\right)_{\mathrm{M}} = \left(\rho_h^{n}, 1\right)_{\mathrm{M}}.
\end{equation*}
Therefore, the results in Theorem \ref{thm:pc-mass-cov} are obtained.
\end{proof}

\begin{theorem}[Positivity-preserving]\label{thm:pc-posi-preserve}
	Let $\big\{\rho_h^{n+1}, c_h^{n+1}, \rho_h^{n+1/2}, c_h^{n+1/2}\big\}$ be the solutions to the PC-BCFD scheme \eqref{PC-BCFD:rewrite}--\eqref{PC-BCFD:IBc:rewrite}. Then,
	\begin{enumerate}[(i)]
		\item if $\rho^o > 0$,  then $\rho_h^{n+1/2} >0$; moreover, if $c^o > 0$ and $\rho^o > 0$,  then $c_h^{n+1/2} >0$;
		\item if $\rho^o > 0$ and $\tau \leq \f{h^2}{ \sigma^6}\f{\min_{i,j}\{\mM_{h,i,j}^{n}\} }{ \max_{l,j}\{\mM_{x,l-1/2,j}^{n+1/2}\} + \max_{i,k}\{\mM_{y,i,k-1/2}^{n+1/2}\} }$,  then $\rho_h^{n+1} >0$; moreover, if $c^o > 0$, $\rho^o > 0$ and $\tau \leq \f{\varepsilon h^2}{2 \sigma^6}$, then $c_h^{n+1} >0$.
	\end{enumerate}
\end{theorem}
\begin{proof}	The results in (i) can be proved in a very similar way as Theorem \ref{thm:posi-preserve}. In the following, we only need to prove the results in (ii).

Similar as in Theorem \ref{thm:posi-preserve}, we can rewrite \eqref{PC-BCFD:rho:rewrite} in a compact matrix form:
\begin{equation}\label{MatrixForm:rho'}
	B_\rho\, \text{vec}(g_h^{n+1}) = \text{vec}(f)(\hat{g}_{h}^{n},g_h^{n}),
\end{equation}
where the right-hand side of \eqref{MatrixForm:rho'} is a vector with respect to $\hat{g}_{h}^{n}$ and $g_h^{n}$. The entries of
$B_\rho$ and $\text{vec}(f)$ can be written explicitly. For example, for an interior node $(x_i,y_j)$ of $\Pi_x^*\times \Pi_y^*$, i.e., $2\leq i \leq N_x-1$, $2\leq j \leq N_y-1$, the linear algebraic equation corresponding to \eqref{PC-BCFD:rho:rewrite} reads as follows:
	\begin{equation}\label{pc-scheme:explicit:interior}
		\begin{aligned}
			&\mM_{h,i,j}^{n+1}g_{h,i,j}^{n+1} \\
			&\quad- \f{\tau_{n+1}}{2} \Big[ \f{\mM_{x,i-1/2,j}^{n+1/2}}{\Delta x_i \Delta x_{i-1/2}} g_{h,i-1,j}^{n+1}  - ( \f{\mM_{x,i-1/2,j}^{n+1/2}}{\Delta x_i \Delta x_{i-1/2}} + \f{\mM_{x,i+1/2,j}^{n+1/2}}{\Delta x_i \Delta x_{i+1/2}} ) g_{h,i,j}^{n+1} + \f{\mM_{x,i+1/2,j}^{n+1/2}}{\Delta x_i \Delta x_{i+1/2}} g_{h,i+1,j}^{n+1} \Big]\\
			&\quad- \f{\tau_{n+1}}{2} \Big[ \f{\mM_{y,i,j-1/2}^{n+1/2}}{\Delta y_j \Delta y_{j-1/2}} g_{h,i,j-1}^{n+1}  - ( \f{\mM_{y,i,j-1/2}^{n+1/2}}{\Delta y_j \Delta y_{j-1/2}} + \f{\mM_{y,i,j+1/2}^{n+1/2}}{\Delta y_j \Delta y_{j+1/2}} ) g_{h,i,j}^{n+1} + \f{\mM_{y,i,j+1/2}^{n+1/2}}{\Delta y_j \Delta y_{j+1/2}} g_{h,i,j+1}^{n+1} \Big]\\
			&= \mM_{h,i,j}^{n+1}\hat{g}_{h,i,j}^{n}\\
			&\quad+ \f{\tau_{n+1}}{2} \Big[ \f{\mM_{x,i-1/2,j}^{n+1/2}}{\Delta x_i \Delta x_{i-1/2}} g_{h,i-1,j}^{n}  - ( \f{\mM_{x,i-1/2,j}^{n+1/2}}{\Delta x_i \Delta x_{i-1/2}} + \f{\mM_{x,i+1/2,j}^{n+1/2}}{\Delta x_i \Delta x_{i+1/2}} ) g_{h,i,j}^{n} + \f{\mM_{x,i+1/2,j}^{n+1/2}}{\Delta x_i \Delta x_{i+1/2}} g_{h,i+1,j}^{n} \Big]\\
			&\quad+ \f{\tau_{n+1}}{2} \Big[ \f{\mM_{y,i,j-1/2}^{n+1/2}}{\Delta y_j \Delta y_{j-1/2}} g_{h,i,j-1}^{n}  - ( \f{\mM_{y,i,j-1/2}^{n+1/2}}{\Delta y_j \Delta y_{j-1/2}} + \f{\mM_{y,i,j+1/2}^{n+1/2}}{\Delta y_j \Delta y_{j+1/2}} ) g_{h,i,j}^{n} + \f{\mM_{y,i,j+1/2}^{n+1/2}}{\Delta y_j \Delta y_{j+1/2}} g_{h,i,j+1}^{n} \Big].
		\end{aligned}
	\end{equation}
Besides, for the corner and edge nodes, the scheme can be rewritten in a similar way as in Theorem \ref{thm:posi-preserve}.
It is then easy to verify that the coefficient matrix $B_\rho$ is still a non-singular M-matrix. Therefore, if right-hand side vector $\text{vec}(f)$ is also non-negative, the positivity of $\rho_h^{n+1}$ can be guaranteed. Assuming that
\begin{equation}\label{pc-positive-rhs}
	\mM_{h,i,j}^{n+1}\hat{g}_{h,i,j}^{n} - \f{\tau_{n+1}}{2}( \f{\mM_{x,i-1/2,j}^{n+1/2}}{\Delta x_i \Delta x_{i-1/2}} + \f{\mM_{x,i+1/2,j}^{n+1/2}}{\Delta x_i \Delta x_{i+1/2}} +
	\f{\mM_{y,i,j-1/2}^{n+1/2}}{\Delta y_j \Delta y_{j-1/2}} + \f{\mM_{y,i,j+1/2}^{n+1/2}}{\Delta y_j \Delta y_{j+1/2}} )g_{h,i,j}^{n} \geq 0,
\end{equation}
for an arbitrary $i, j$, which indicates that if
\begin{equation}\label{pc-posi-condition}
	\tau_{n+1} \leq \f{h^2}{ \sigma^6}\f{\min_{i,j}\{\mM_{h,i,j}^{n}\} }{ \max_{l,j}\{\mM_{x,l-1/2,j}^{n+1/2}\} + \max_{i,k}\{\mM_{y,i,k-1/2}^{n+1/2}\} }
\end{equation}
holds, the right-hand side of \eqref{pc-scheme:explicit:interior} is positive for any $2\leq i \leq N_x-1$, $2\leq j \leq N_y-1$. Thus, under the time-step condition \eqref{pc-posi-condition}, the numerical solution $g_h^{n+1}$ (and consequently $\rho_h^{n+1}$) obtained from \eqref{PC-BCFD:rho:rewrite} preserves positivity.

Similarly, under the time-step condition 
\begin{equation}\label{pc-posi-condition2}
	\tau_{n+1} \leq \f{\varepsilon h^2}{2 \sigma^6},
\end{equation}
the numerical solution $c_h^{n+1}$ obtained from \eqref{PC-BCFD:c:rewrite} also preserves positivity.
\end{proof}

\begin{remark}\label{rm:positivity}As discussed in Remark \ref{rmk:positivity:ellip}, for the parabolic--elliptic case $\varepsilon=0$, the prediction equation for the concentration in the PC-BCFD scheme still yields $c_h^{n+1/2}>0$ by the same M-matrix argument. However, the correction step is written for the average value $\overline c_h^{\,n+1/2}=(c_h^{n+1}+c_h^n)/2$, so the present analysis only gives the positivity of $\overline c_h^{\,n+1/2}$, but does not directly imply that of $c_h^{n+1}$. Furthermore, we note that the conditions \eqref{pc-posi-condition} and \eqref{pc-posi-condition2} seem very limited, as for example \eqref{pc-posi-condition} derives from the restrictive assumption \eqref{pc-positive-rhs} used to maintain the positivity of the right-hand side of \eqref{MatrixForm:rho'}. In practice, however, this only requires that the right-hand side of \eqref{pc-scheme:explicit:interior} (but not that of \eqref{pc-positive-rhs}) is positive. Thus, such conditions \eqref{pc-posi-condition}--\eqref{pc-posi-condition2} are only sufficient, and the numerical tests in Section \ref{sec:num} indeed show that they are not necessary.
\end{remark}

\begin{remark}
Since \eqref{PC-BCFD:c_half:rewrite}--\eqref{PC-BCFD:rho_half:rewrite} have essentially the same mathematical structure as \eqref{BE-BCFD:c:rewrite}--\eqref{BE-BCFD:rho:rewrite}, the predicted solutions $\rho_h^{n+1/2}$ and $c_h^{n+1/2}$ satisfy the same discrete energy-dissipation law, and hence $E^{n+1/2}\le E^n$. In contrast, the correction step \eqref{PC-BCFD:c:rewrite}--\eqref{PC-BCFD:rho:rewrite} involves both the last time value $\hat g_h^n$ and the average values $\overline c_h^{\,n+1/2}$ and $\overline g_h^{\,n+1/2}$; therefore, it no longer retains the same discrete gradient flow structure as the first-order BE-BCFD scheme. Consequently, the energy-dissipation argument used for the first-order scheme cannot be applied directly to the correction step.
Hence, no discrete energy-dissipation result is established here for the fully discrete second-order prediction-correction scheme. Nevertheless, numerical experiments suggest that the proposed second-order method remains energy dissipative in practice. 

We also note that the present method is a fully discrete linear implicit scheme formulated in terms of the original variables. To the best of our knowledge, comparable second-order linear schemes for the Keller--Segel systems, together with rigorous positivity-preserving and energy-dissipation analysis, remain scarce in the literature. In contrast, a recent second-order nonlinear scheme that preserves both positivity and the original energy dissipation is presented in  \cite{DWZ'25}.
\end{remark}

 \section{Numerical results}\label{sec:num}
In this section, we present several numerical experiments using the proposed BE-BCFD scheme \eqref{BE-BCFD:rewrite}–\eqref{BE-BCFD:IBc:rewrite} and  PC-BCFD scheme \eqref{PC-BCFD:rewrite}–\eqref{PC-BCFD:IBc:rewrite}. Subsection \ref{subsec:accuracy} demonstrates the accuracy of both schemes on both uniform and non-uniform spatial grids. Subsections \ref{example:pe}–\ref{example:pp} are used to simulate various blow-up scenarios with different initial values, and meanwhile, to verify preservation of physical structures: mass conservation, positivity-preserving and energy dissipation, that proved in Theorems \ref{thm:be-mass-cov}--\ref{thm:posi-preserve} and \ref{thm:energy-dissipasion}--\ref{thm:pc-posi-preserve}.

Two different types of non-uniform spatial grids are utilized in this paper, which are listed as follows:
\begin{enumerate}[(i)]
	\item \emph{Random disturbance grids}: We first define a uniform partition $x_{{\rm fix},i+1/2}=a^x+ ih_{\rm fix}$, $i=0,\ldots, N_x$, with equal grid size $ h_{\rm fix}=(b^x-a^x)/N_x$. Then, through a slight random adjustment of the grid sizes, we define the non-uniform grids:
	\begin{equation}\label{grid:rand}
		x_{i+1/2} = x_{{\rm fix},i+1/2} + \beta *h_{\rm fix}*(-1+2*rand(i)),  \quad i=1,2, \ldots, N_x-1,
	\end{equation}
	where the mesh parameter $\beta$ is used to regulate random disturbance within a specific range. The $y$-direction non-uniform grids $\{y_{j+1/2}\}$ $(j=0, \ldots, N_y)$ are defined in a similar manner. A reference grid is shown in Fig. \ref{randonm_grid}.
	\item \emph{Middle refinement grids}: In real simulation, the following specified non-uniform graded grids are also considered:
	\begin{equation}\label{grid:mid}
		\begin{cases}
			x_{N_{x}/2+i+1/2}=\f{a^x + b^x}{2} + \f{b^x - a^x}{2}\left[\frac{i}{\left(N_x / 2+1\right)}\right]^\gamma, & i=0,1, \ldots, N_x / 2+1, \\ 
			x_{N_{x}/2-i+1/2}=\f{a^x + b^x}{2} -\f{b^x - a^x}{2}\left[ \frac{i}{\left(N_x / 2+1\right)}\right]^\gamma, & i=1,2, \ldots, N_x / 2+1,
		\end{cases}
	\end{equation}
	where $\gamma>1$ specifies how much the central region is thickened (larger $\gamma$ means greater thickening). The $y$-direction non-uniform grids $\{y_{j+1/2}\}$ $(j=0, \ldots, N_y)$ are similarly defined. Here, for simplicity, $N_x$ and $N_y$ are chosen as even integers. A reference grid is shown in Fig. \ref{middle_grid}. 
\end{enumerate}
 \begin{figure}[!t]
	\centering
	\subfigure[$\beta = 0.5$, $N_x = N_y = 9$ in \eqref{grid:rand}]{\label{randonm_grid}\includegraphics[width=0.45\linewidth]{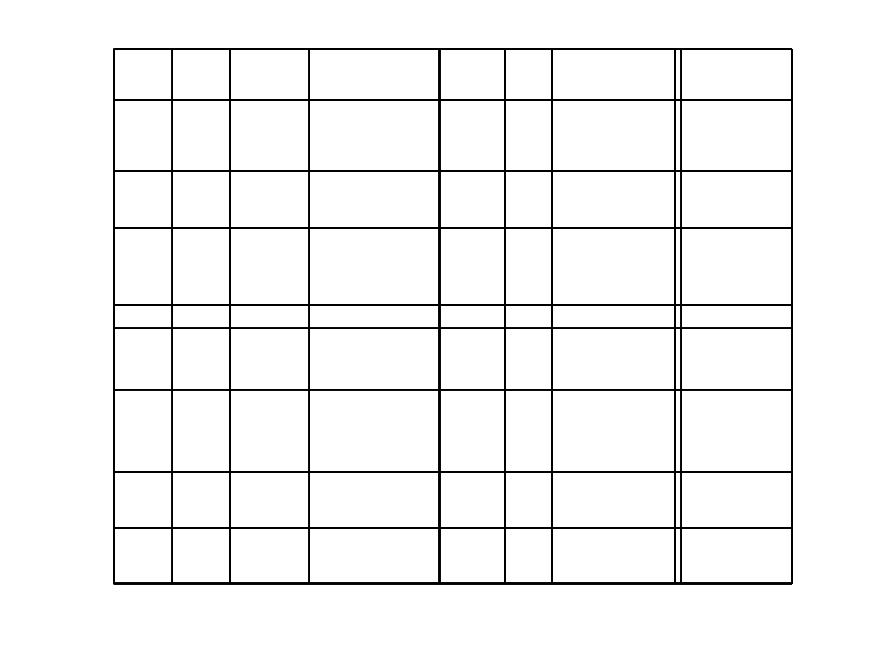}}
	\subfigure[$\gamma=2$, $N_x = N_y = 40$ in \eqref{grid:mid}]{\label{middle_grid}\includegraphics[width=0.45\linewidth]{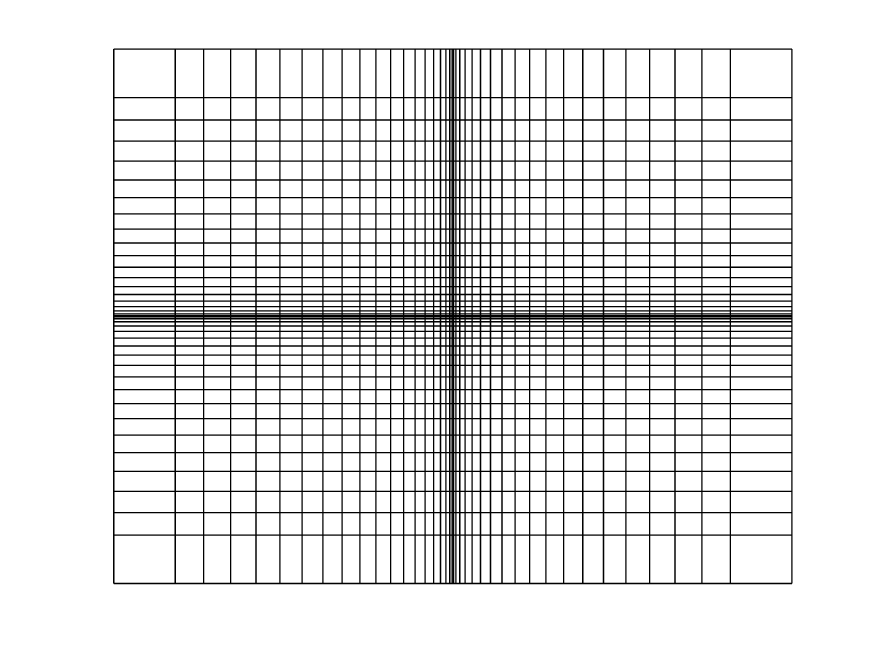}}
	\caption{Schematic non-uniform spatial grids.}
	\label{fig:plotgrid}
\end{figure} 

 In all tests below, we always set $N_x = N_y \equiv M$ for a square domain $\Omega$.
\subsection{Accuracy tests}\label{subsec:accuracy}
In this subsection, we consider the following parabolic-parabolic Keller--Segel system with source terms
\begin{equation}\label{test:model}
    \left\{
    \begin{aligned}
     & \p_t \rho  = \Delta \rho-  \nabla \cdot(\rho \nabla c) + f_\rho(x,y),  &&\quad \text{in}~~  \Omega \times(0, T], \\
     & \p_t c  = \Delta c- c+  \rho  + f_c(x,y),   &&\quad \text{in}~~  \Omega \times(0, T], 
    \end{aligned}
    \right.
\end{equation}
enclosed with homogeneous Neumann boundary conditions.

\begin{example}\label{exa:accuracy1}
	 For the first example, let $\Omega=(0, \pi)^2$ and consider initial conditions 
\[
 \rho^o( x, y)=3 \cos x \cos y+3,\quad	c^o(x, y)=\cos x \cos y+3,
\]
with source terms $f_\rho(x, y)=-3 \cos (2 x) \cos ^2 y-3 \cos ^2 x \cos (2 y)$ and $f_c(x,y) = 0$ in \eqref{test:model}. In this scenario, the exact solutions of \eqref{test:model} are in steady state, i.e, $\rho( x, y, t) \equiv \rho^o( x, y),\	c( x, y, t )\equiv c^o(x, y)$ for any time $t$. 
\end{example}

This example is conducted to test the spatial accuracy on general non-uniform spatial grids \eqref{grid:rand} for the BE-BCFD scheme \eqref{BE-BCFD:rewrite}--\eqref{BE-BCFD:IBc:rewrite}. As the solution is actually time-independent, the time stepsize can be set as $\tau = h_{\rm fix}$, and the corresponding discrete $L^2$-norm errors of $\rho^n-\rho_h^n$ and $c^n-c_h^n$ at $T = 1$ are listed in Table \ref{tab:accuracy1}. We see that for different $\beta$, either uniform ($\beta=0$) or non-uniform ($\beta \neq 0$) grids, the second-order spatial accuracy can both be observed.
\begin{table}[!htbp]
	\centering \caption{ $L^2$ errors for the BE-BCFD scheme for Example \ref{exa:accuracy1}.} \label{tab:accuracy1}
	\setlength{\tabcolsep}{3mm}
	\begin{tabular}{c| c c c c c}
		\toprule
		Mesh parameter & $M$ &  $\|\rho^n-\rho_h^n\|_{\rm M}$&  Order&  $\|c^n-c_h^n\|_{\rm M}$& Order  \\
		\midrule
		\multirow{5}*{\makecell{ $\beta = 0$}}
		&20 & 1.71e-02 & --- & 3.43e-03 &  ---\\
		&40 & 4.48e-03 & 1.93 & 8.69e-04 & 1.98 \\
		&80 & 1.15e-03 & 1.96 & 2.20e-04 & 1.98 \\
		&160 & 2.92e-04 & 1.98 & 5.53e-05 & 1.99 \\
		\midrule
		\multirow{5}*{\makecell{ $\beta = 0.2$}}
		&20 & 1.88e-02 & --- & 4.40e-03 & --- \\
		&40 & 5.44e-03 & 1.79 & 1.12e-03 & 1.98 \\
		&80 & 1.33e-03 & 2.03 & 2.69e-04 & 2.06 \\
		&160 & 3.32e-04 & 2.00 & 6.85e-05 &1.97 \\
		\midrule
		\multirow{5}*{\makecell{ $\beta = 0.5$}}
		&20 & 2.77e-02 & --- & 8.62e-03 & --- \\
		&40 & 9.55e-03 & 1.54 & 2.24e-03 & 1.95 \\
		&80 & 2.17e-03 & 2.14 & 5.08e-04 & 2.14 \\
		&160 & 5.12e-04 & 2.08 & 1.21e-04 &2.06 \\
		\bottomrule
	\end{tabular}
\end{table}

\begin{example}\label{exa:accuracy2} In this example, the computational domain is set to $\Omega=(0,1)^2$. The source terms are chosen such that the exact solutions are 
	\begin{align*}
		\rho(x, y, t)= c(x, y, t)= t (x^2-x )^2 (y^2-y )^2.
	\end{align*}
\end{example}

We numerically test the spatial and temporal accuracy of the PC-BCFD scheme \eqref{PC-BCFD:rewrite}--\eqref{PC-BCFD:IBc:rewrite} by setting the grid sizes $\tau=h_{\rm fix}$, in which the discrete $L^2$-norm errors are measured at $T=1$ and shown in Table \ref{tab:accuracy2}. We can observe that the proposed scheme achieves second-order convergence in both time and space, demonstrated on both uniform and non-uniform spatial grids.
\begin{table}[!tbp]
	\centering 
	\caption{ $L^2$ errors for the PC-BCFD scheme for Example \ref{exa:accuracy2}.} \label{tab:accuracy2}
	\setlength{\tabcolsep}{3mm}
	\begin{tabular}{c| c c c c c}
		\toprule
		Mesh parameter & $M$ &  $\|\rho^n-\rho_h^n\|_{\rm M}$&  Order&  $\|c^n-c_h^n\|_{\rm M}$& Order  \\
		\midrule
		\multirow{5}*{\makecell{ $\beta = 0$}}
		&20 & 8.36e-05 & --- & 8.36e-05 & --- \\
		&40 & 2.09e-05 & 2.00 & 2.09e-05 & 2.00 \\
		&80 & 5.23e-06 & 2.00 & 5.22e-06 & 2.00 \\
		&160 & 1.31e-06 & 2.00 & 1.31e-06 & 2.00 \\
		\midrule
		\multirow{5}*{\makecell{ $\beta = 0.2$}}
		&20 & 8.97e-05 & --- & 8.97e-05 & --- \\
		&40 & 2.29e-05 & 1.97 & 2.29e-05 & 1.97 \\
		&80 & 5.89e-06 & 2.03 & 5.89e-06 & 2.03 \\
		&160 & 1.40e-06 & 1.99 & 1.40e-06 & 1.99 \\
		\midrule
		\multirow{5}*{\makecell{ $\beta = 0.5$}}
		&20 & 1.08e-04 & --- & 1.08e-04 & --- \\
		&40 & 2.86e-05 & 1.91 & 2.86e-05 & 1.91 \\
		&80 & 6.69e-06 & 2.10 & 6.68e-06 & 2.10 \\
		&160 & 1.69e-06 & 1.99 & 1.69e-06 & 1.99 \\
		\bottomrule
	\end{tabular}
\end{table}

 \subsection{Parabolic-elliptic case ($\varepsilon=0$)}\label{example:pe}
In this subsection, the BE-BCFD scheme \eqref{BE-BCFD:rewrite}--\eqref{BE-BCFD:IBc:rewrite} is employed to simulate various dynamic evolution processes for the parabolic-elliptic Keller–Segel system ($\varepsilon=0$). The mass conservation, positivity preservation, and energy dissipation properties of the scheme are verified. Moreover, according to the existence theory for the 2D parabolic-elliptic Keller--Segel system, solutions exist globally if the initial total cell mass is strictly less than the critical threshold $8\pi$; otherwise, the cell density may blow up in a finite time. Such blow-up phenomena are also simulated under specific initial conditions.

\begin{example}\label{exam:less8pi:pe} (Global existence)  In this simulation, we set the initial cell density  
\begin{equation*}
  \rho^o( x, y)=\frac{60}{1+40\left(x^2+y^2\right)}, 
     \quad \text{in} \quad \Omega=(-2,2)^2,
\end{equation*}
 such that the initial total mass of cells is below the critical threshold, i.e., $M_\rho (0) \approx 24.97 < 8\pi$. 
\end{example}
In this test, we set $M = 40$ and $\tau = 1/M$. The evolutions of (i) the maximum and minimum values of both the cell density $\rho_h$ and the chemoattractant concentration $c_h$, (ii) the discrete energy, and (iii) the total mass of $\rho_h$ are illustrated in Figs. \ref{fig:smooth-m40-uni}--\ref{fig:smooth-m40-beta05}, under both uniform ($\beta=0$) and non-uniform ($\beta=0.5$) spatial grids described by \eqref{grid:rand}. Key observations include:
\begin{enumerate}[(i)]
    \item For the chosen positive initial data $\rho^o$, the numerical solutions $\rho_h$ and $c_h$ remain positive up to the final tested time $T=15$. Furthermore, the discrete energy dissipation and mass conservation properties are clearly observed. These numerical observations are fully consistent with the theoretical conclusions established in Theorems \ref{thm:be-mass-cov}, \ref{thm:posi-preserve}, and \ref{thm:energy-dissipasion}.
    
    \item Even on random non-uniform spatial grids, the inherent physical properties—such as positivity preservation, energy dissipation, and mass conservation—remain clearly observable. Furthermore, the solutions are largely consistent with those obtained on uniform spatial grids.
\end{enumerate}

\begin{figure}[!tbp]
	\centering
	\includegraphics[width=0.32\linewidth]{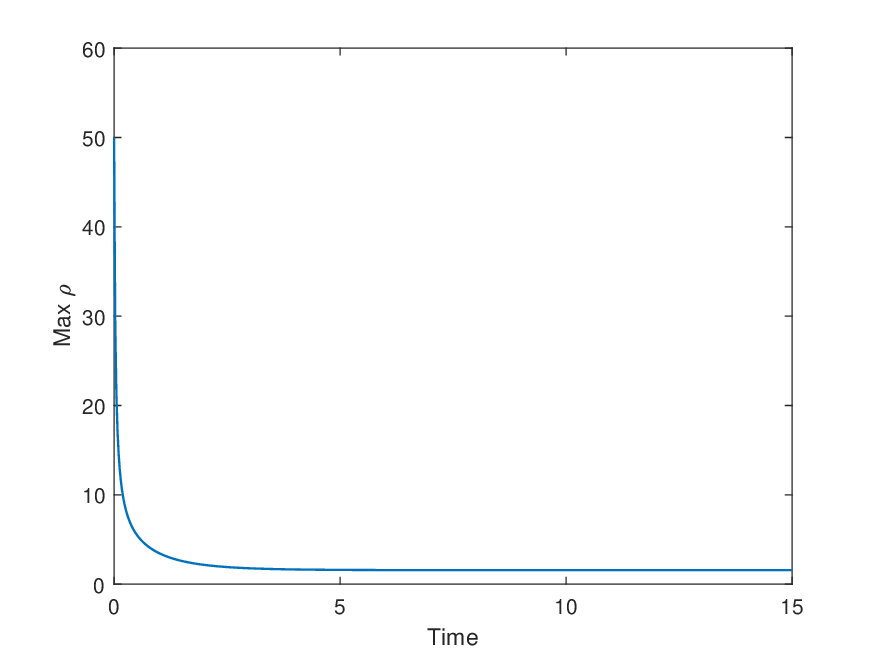}
	\includegraphics[width=0.32\linewidth]{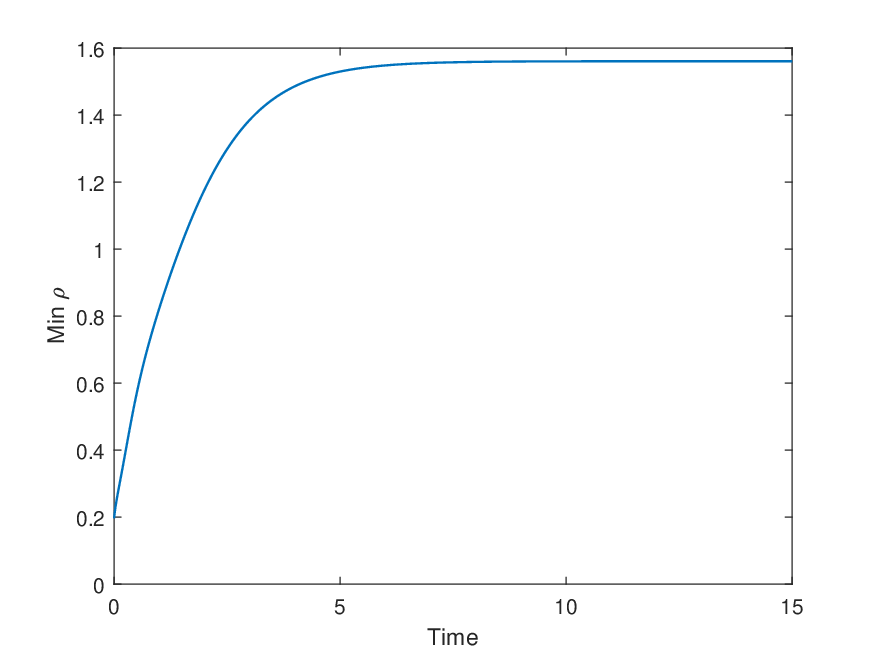}
	\includegraphics[width=0.32\linewidth]{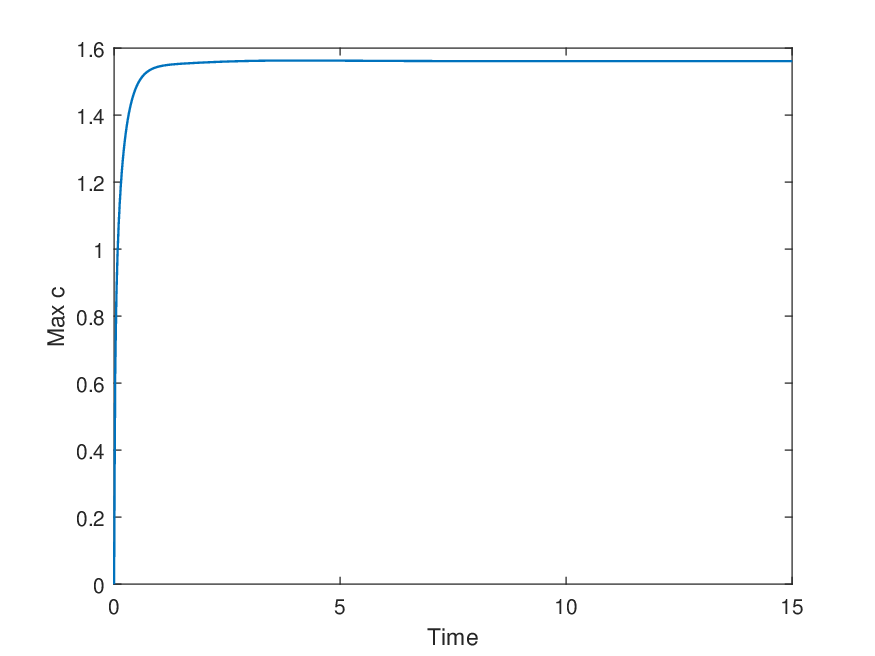}
	\includegraphics[width=0.32\linewidth]{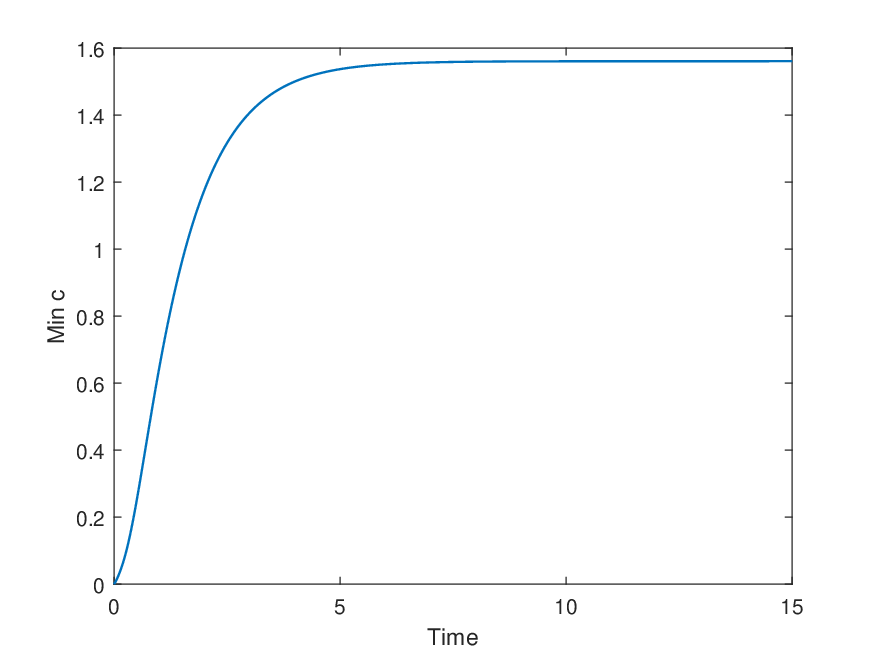}
	\includegraphics[width=0.32\linewidth]{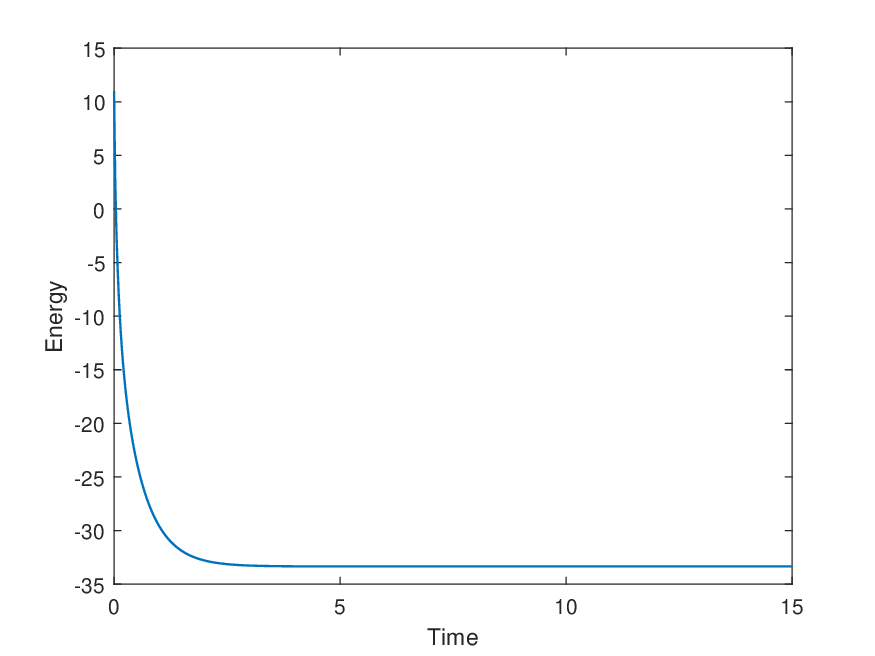}
	\includegraphics[width=0.32\linewidth]{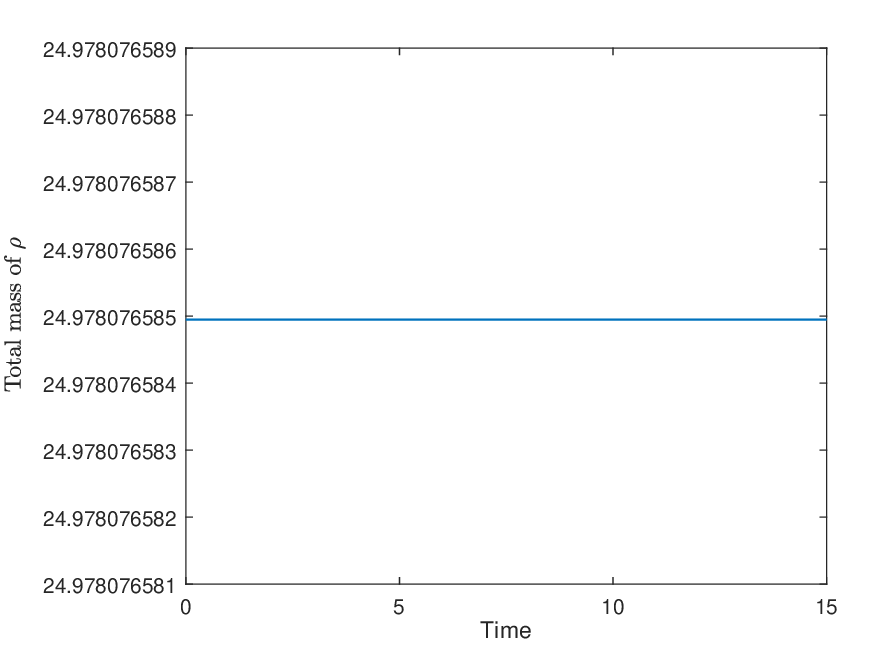}
	\caption{Evolutions of cell density, chemoattractant concentration, discrete energy and total cell mass for the BE-BCFD scheme on uniform  grids for Example \ref{exam:less8pi:pe}. }
	\label{fig:smooth-m40-uni}
\end{figure}
\begin{figure}[!tbp]
\vspace{-10pt}
	\centering
	\includegraphics[width=0.32\linewidth]{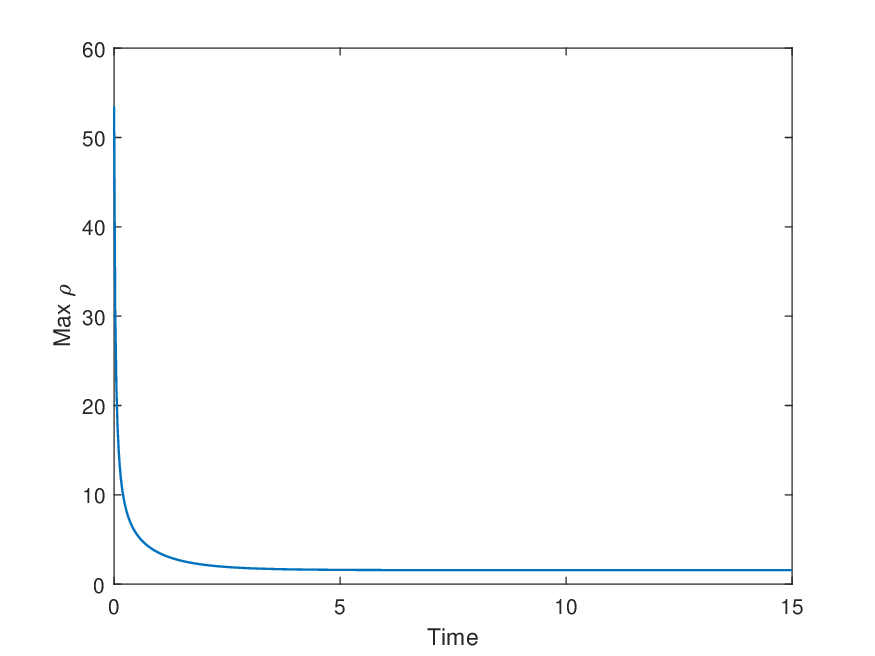}
	\includegraphics[width=0.32\linewidth]{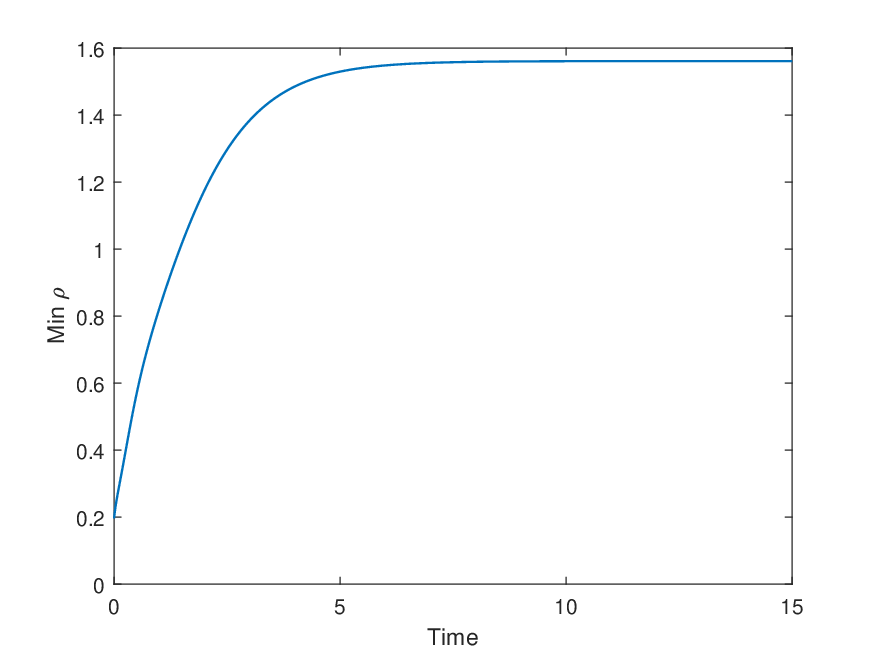}
	\includegraphics[width=0.32\linewidth]{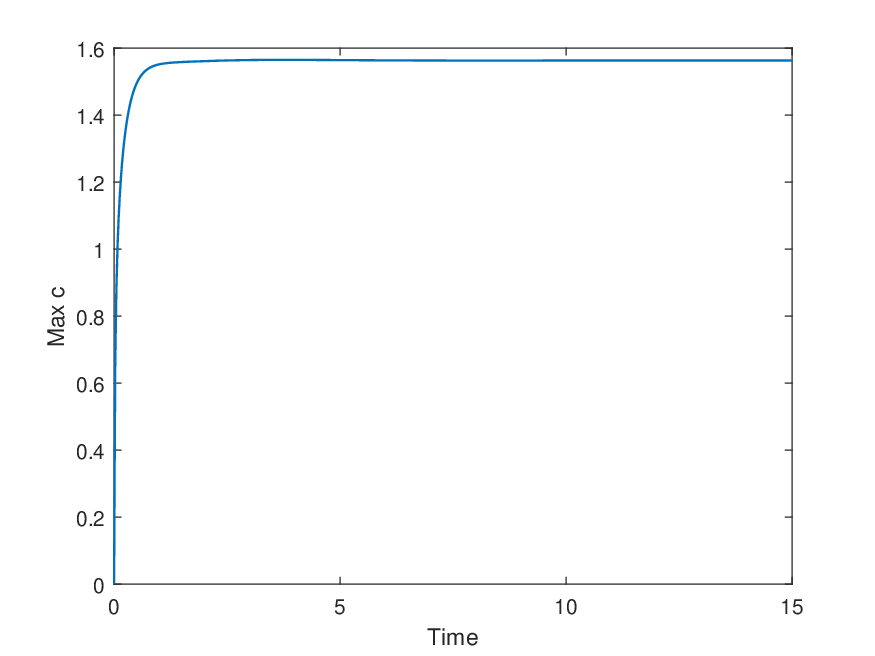}
	\includegraphics[width=0.32\linewidth]{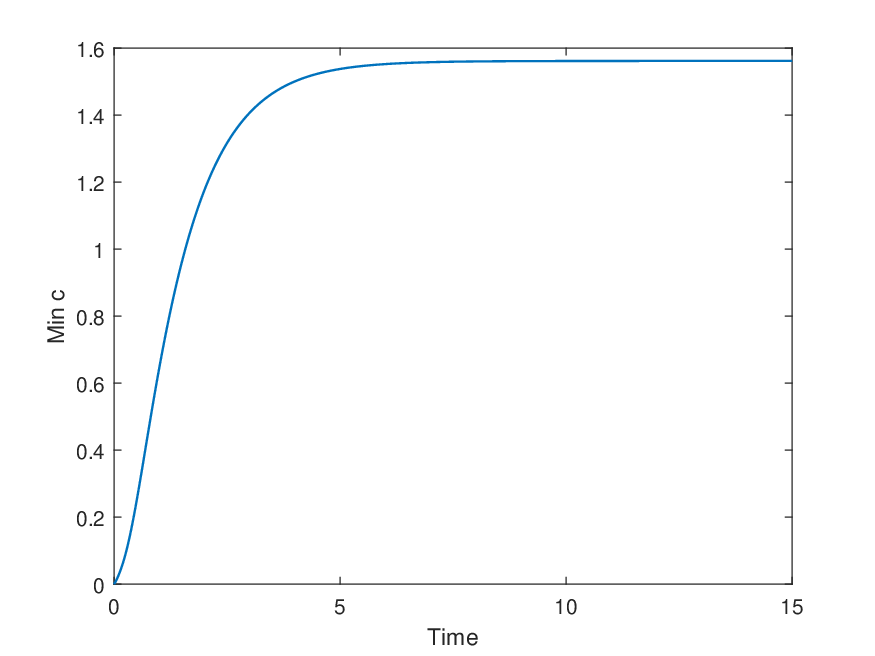}
	\includegraphics[width=0.32\linewidth]{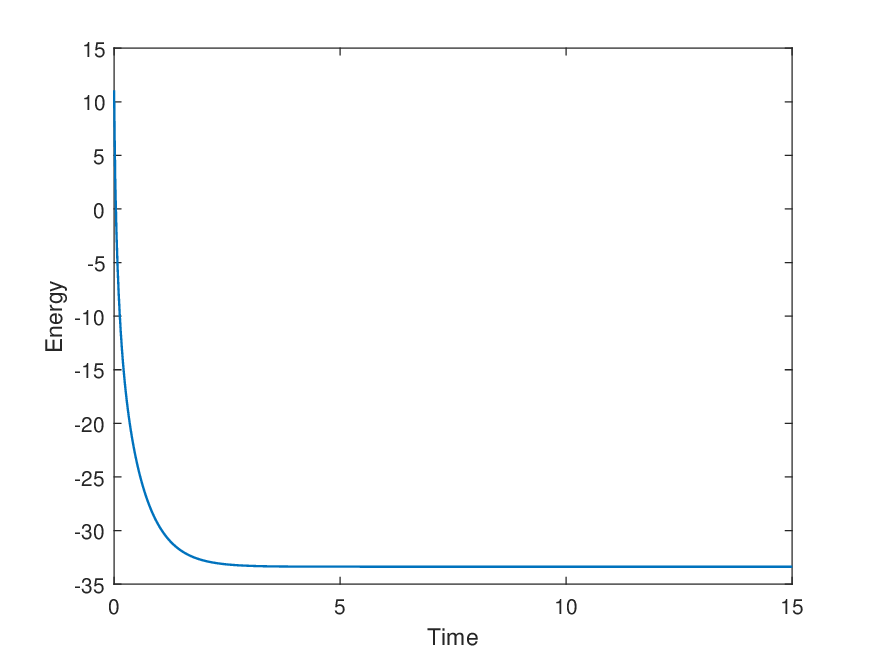}
	\includegraphics[width=0.32\linewidth]{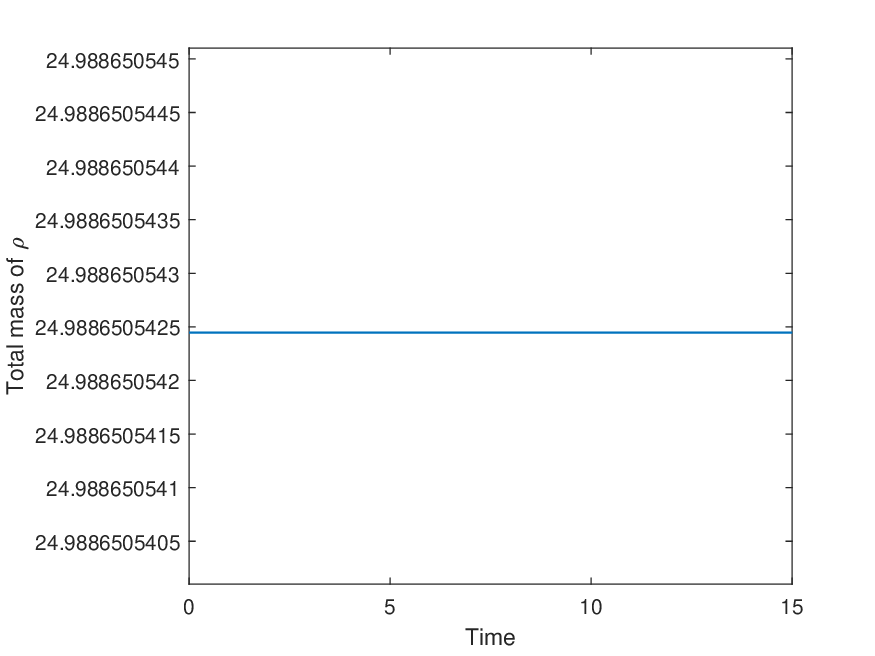}
	\caption{Evolutions of cell density, chemoattractant concentration, discrete energy and total cell mass for the BE-BCFD scheme on non-uniform grids with $\beta=0.5$ for Example \ref{exam:less8pi:pe}.}
	\label{fig:smooth-m40-beta05}
\end{figure}

\begin{figure}[!t]
	\centering
	\includegraphics[width=0.32\linewidth]{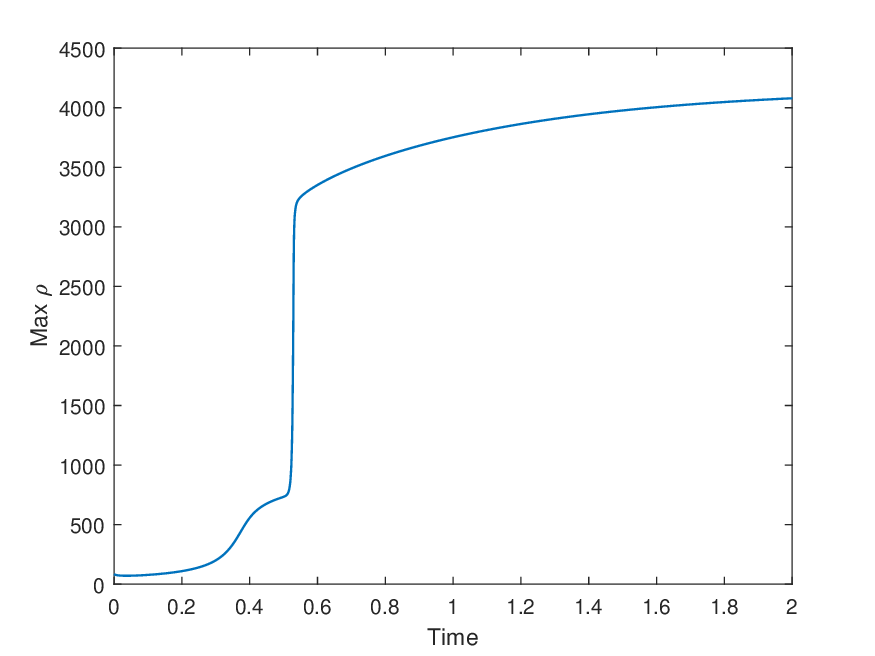}
	\includegraphics[width=0.32\linewidth]{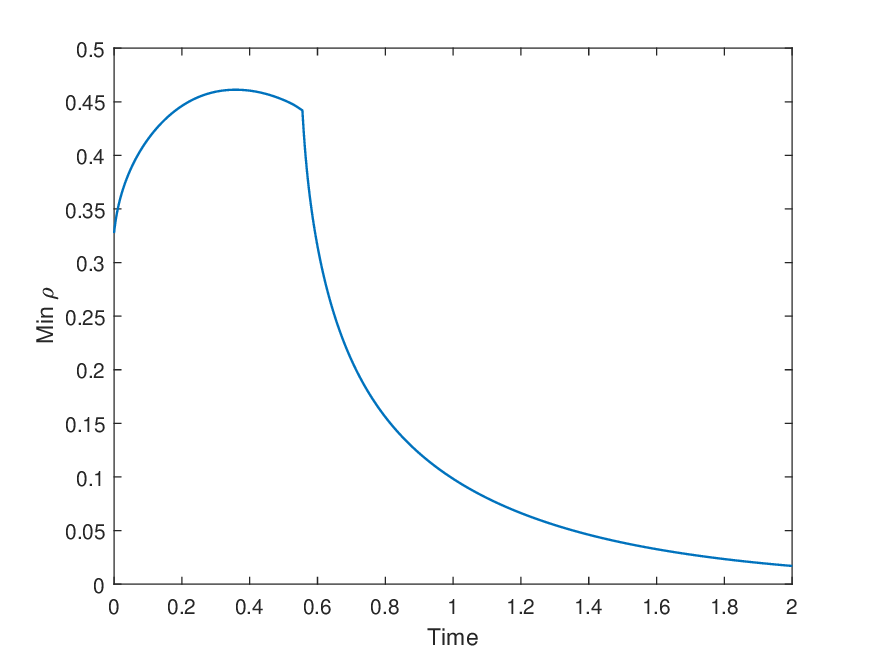}
	\includegraphics[width=0.32\linewidth]{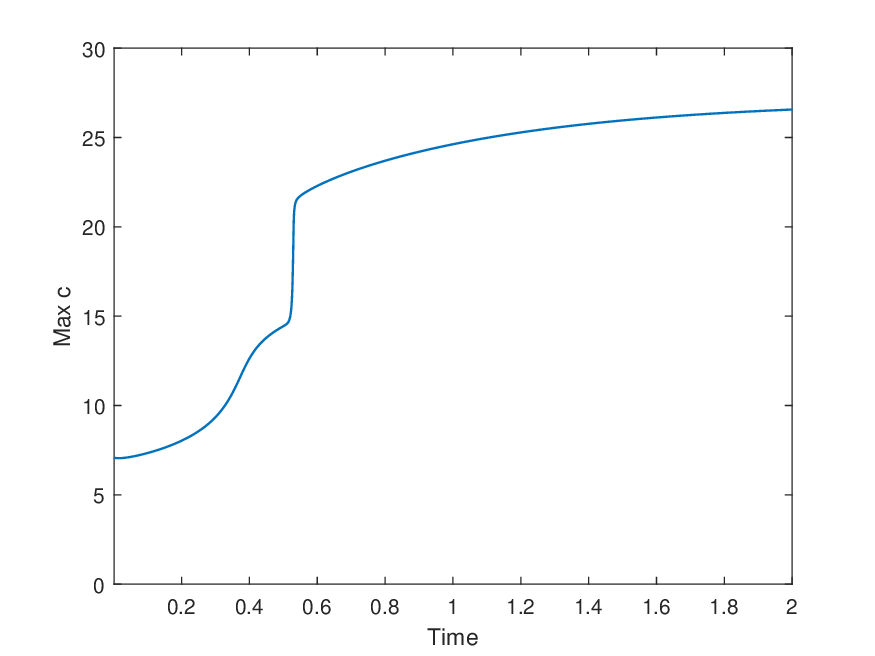}
	\includegraphics[width=0.32\linewidth]{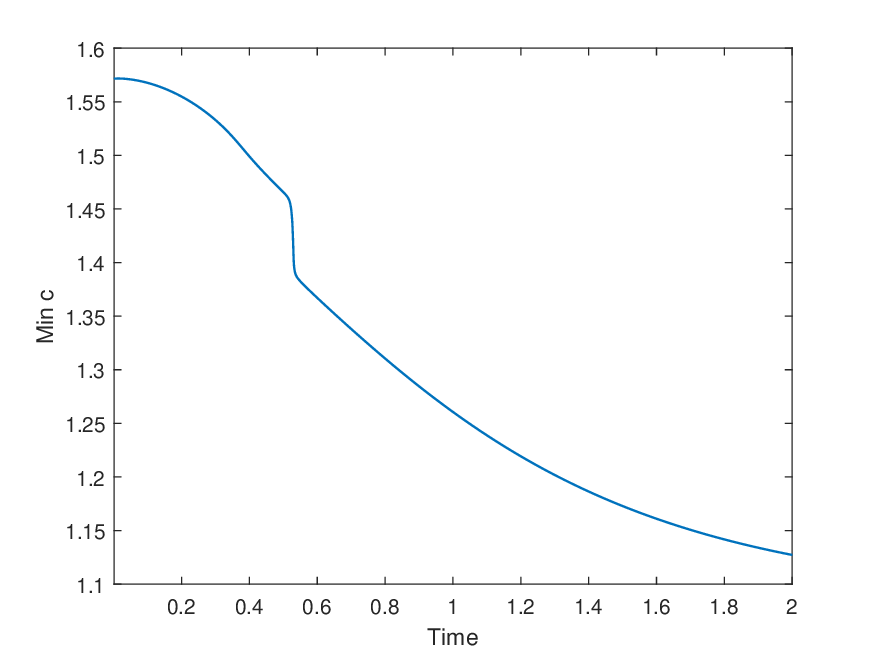}
	\includegraphics[width=0.32\linewidth]{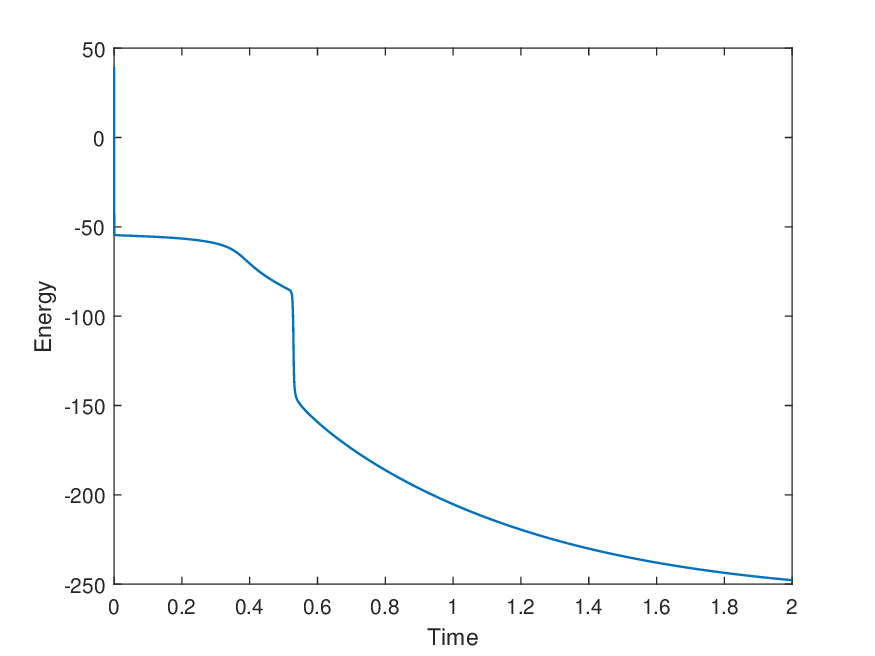}
	\includegraphics[width=0.32\linewidth]{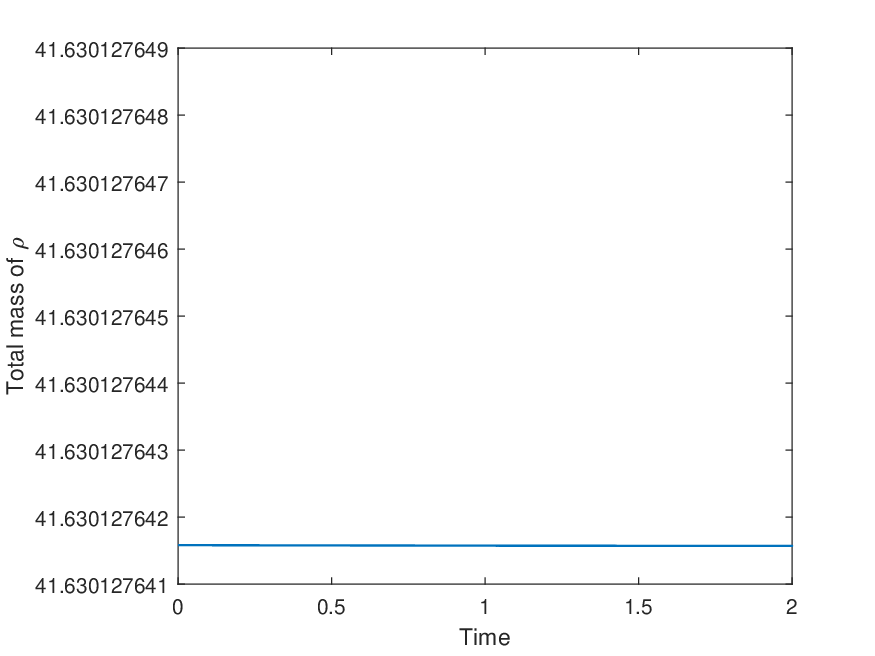}
	\caption{Evolutions of cell density, chemoattractant concentration, discrete energy and total cell mass for the BE-BCFD scheme on uniform grids with $M=40$ for Example \ref{exam:grate8pi:pe}. }
	\label{fig:PE-rho100M40-uni}
\end{figure}
\begin{figure}[!t]
	\vspace{-10pt}
	\centering
	\includegraphics[width=0.32\linewidth]{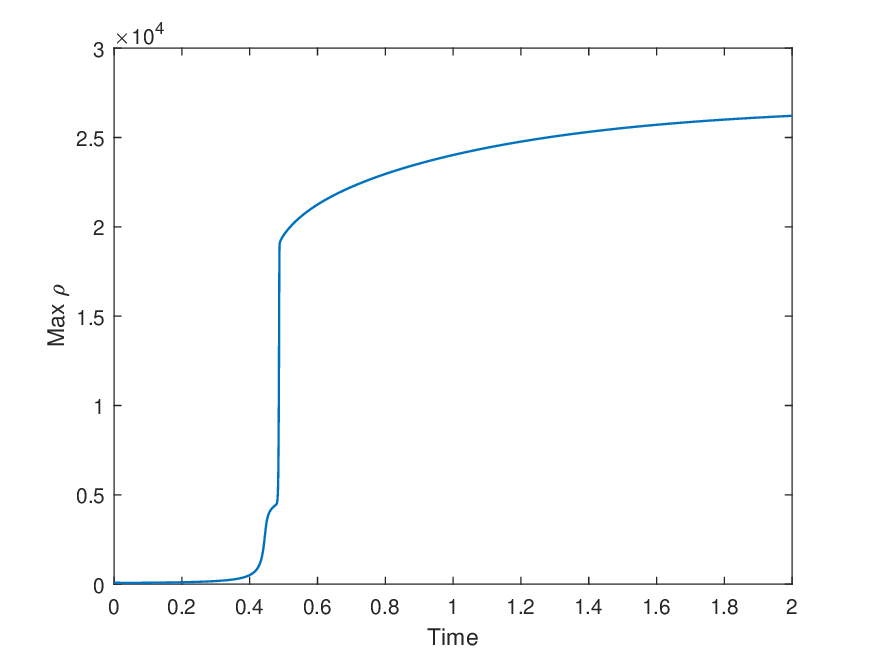}
	\includegraphics[width=0.32\linewidth]{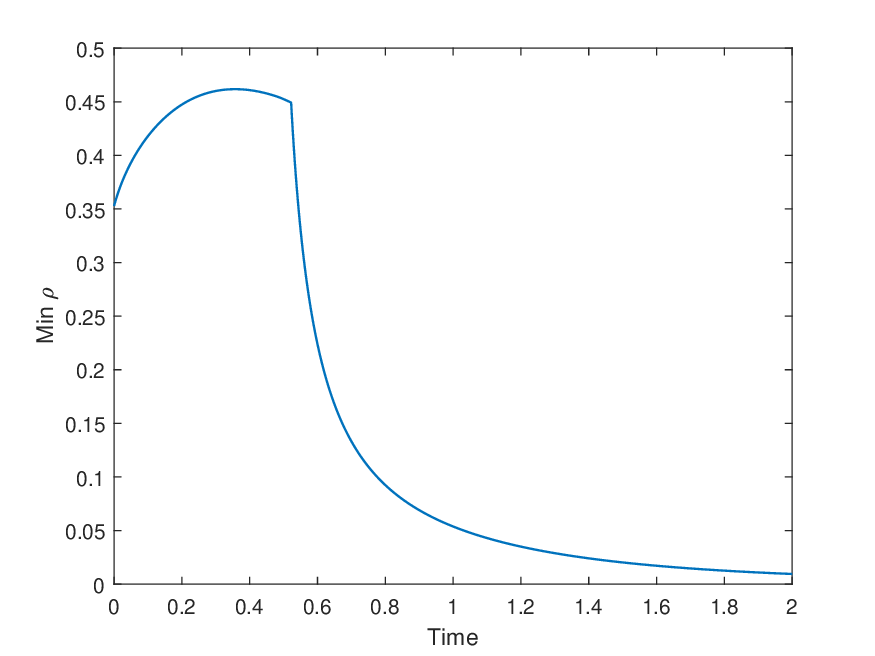}
	\includegraphics[width=0.32\linewidth]{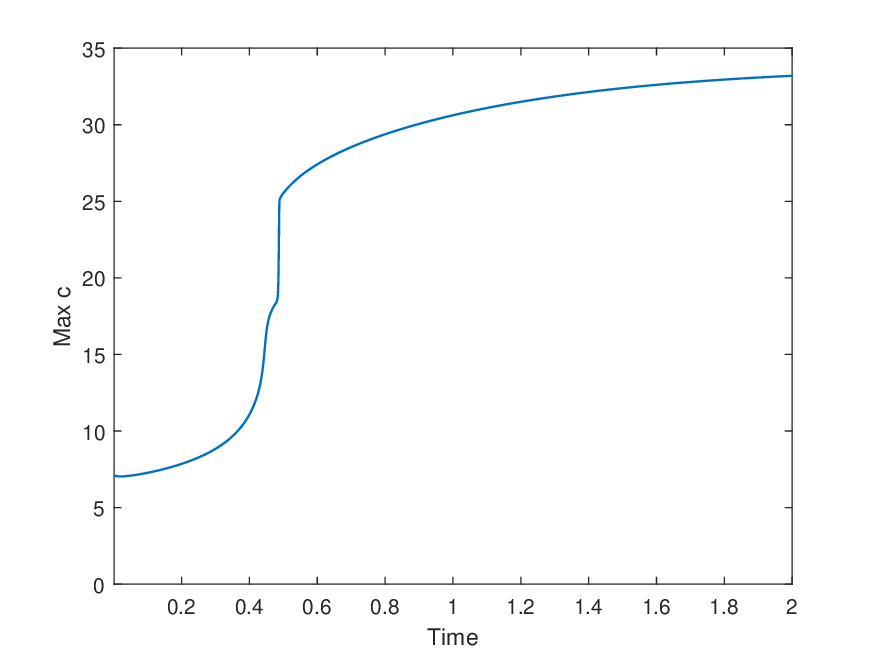}
	\includegraphics[width=0.32\linewidth]{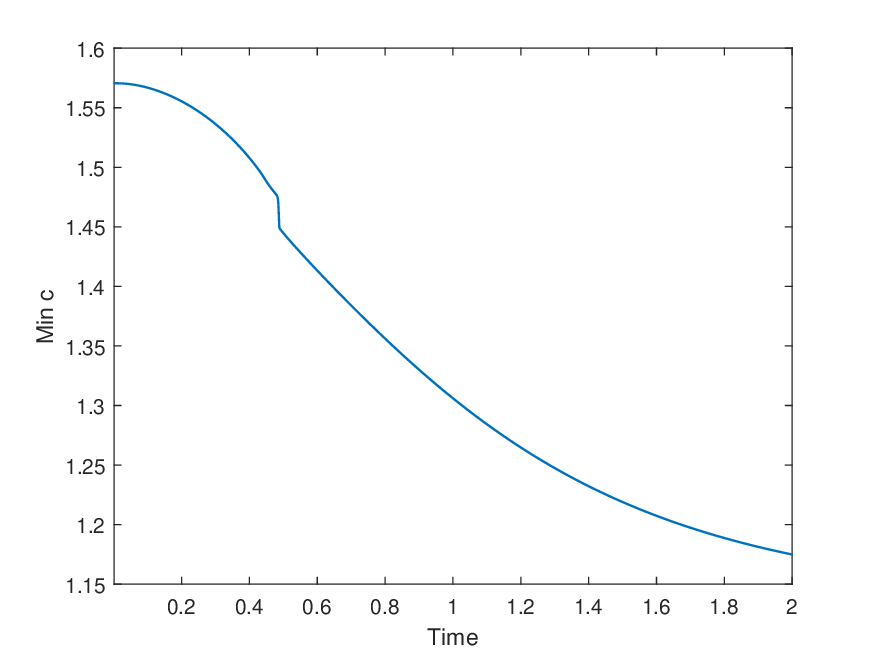}
	\includegraphics[width=0.32\linewidth]{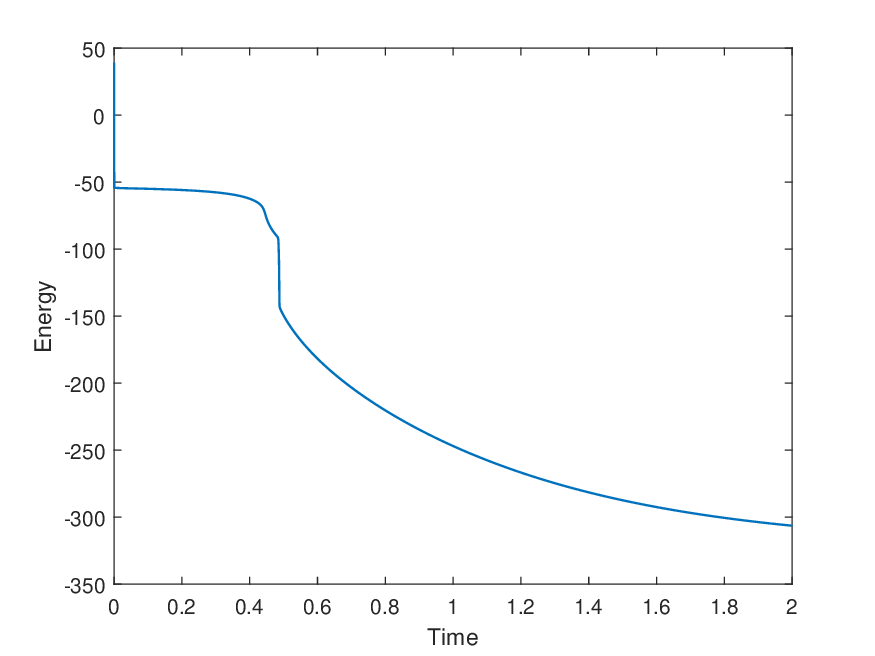}
	\includegraphics[width=0.32\linewidth]{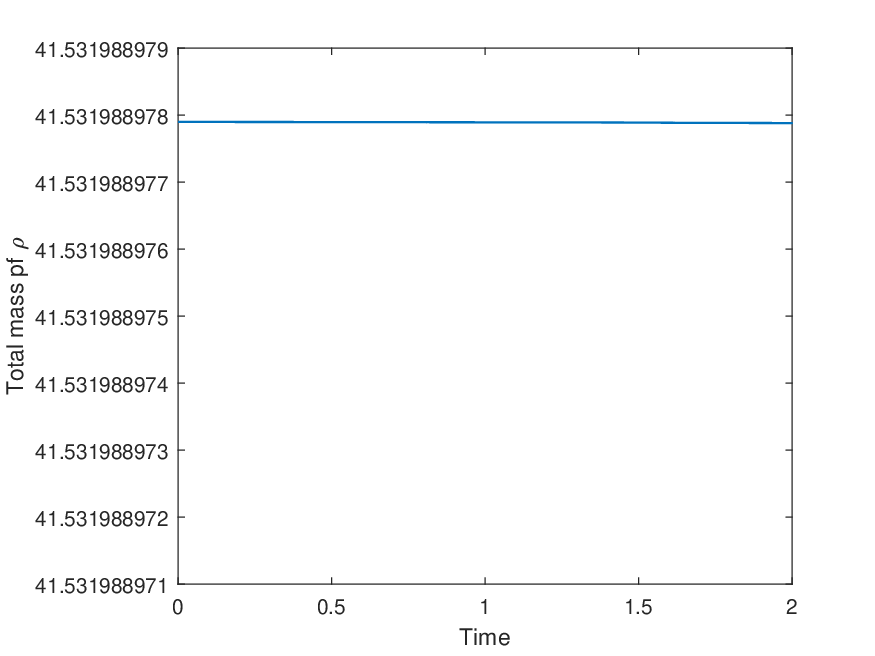}
	\caption{Evolutions of cell density, chemoattractant concentration, discrete energy and total cell mass for the BE-BCFD scheme on non-uniform grids with $M=40$ and $\gamma=1.29$ for Example \ref{exam:grate8pi:pe}. }
	\label{fig:PE-rho100M40-mid13}
\end{figure}
\begin{figure}[!t]
	\centering
	\includegraphics[width=0.32\linewidth]{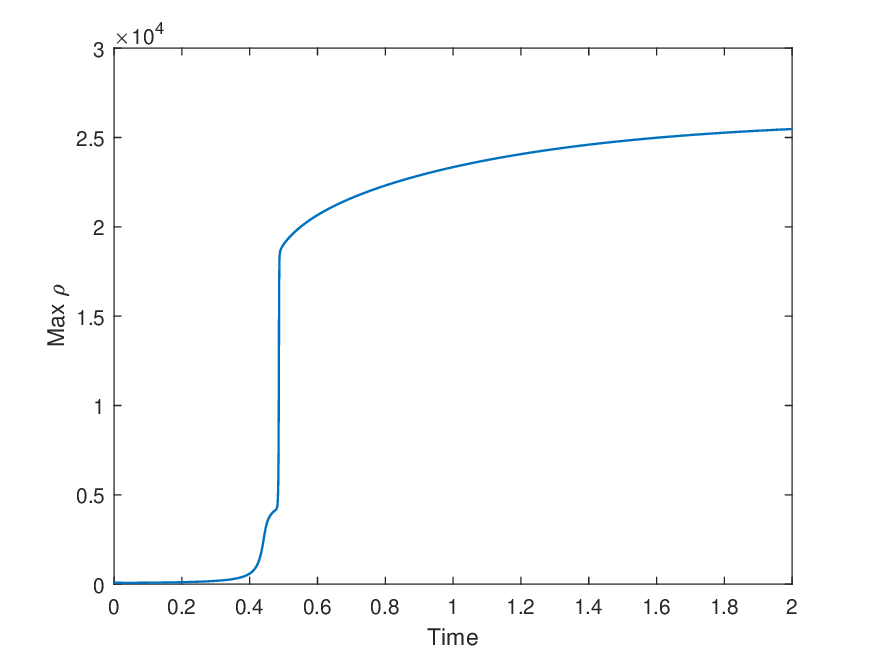}
	\includegraphics[width=0.32\linewidth]{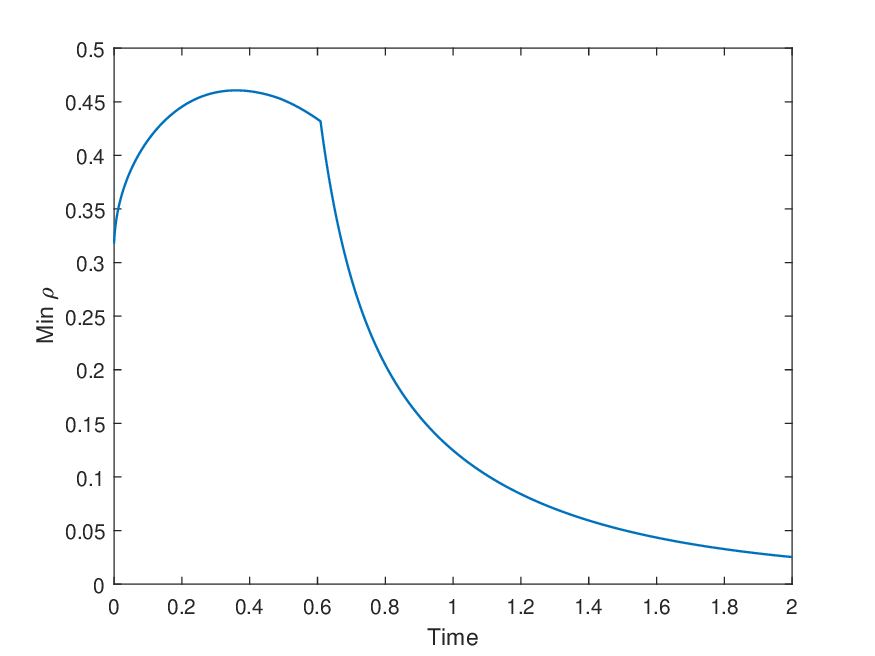}
	\includegraphics[width=0.32\linewidth]{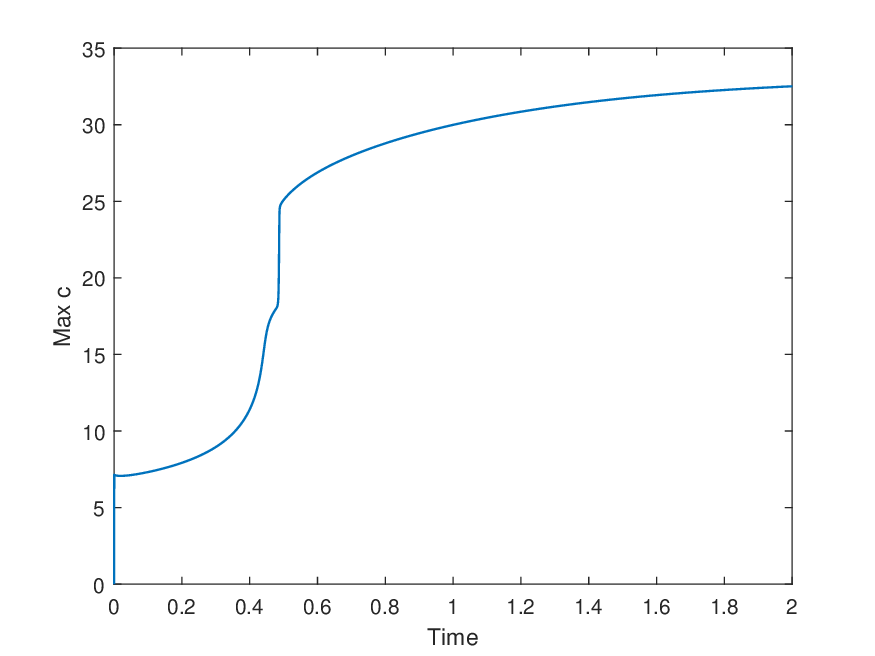}
	\includegraphics[width=0.32\linewidth]{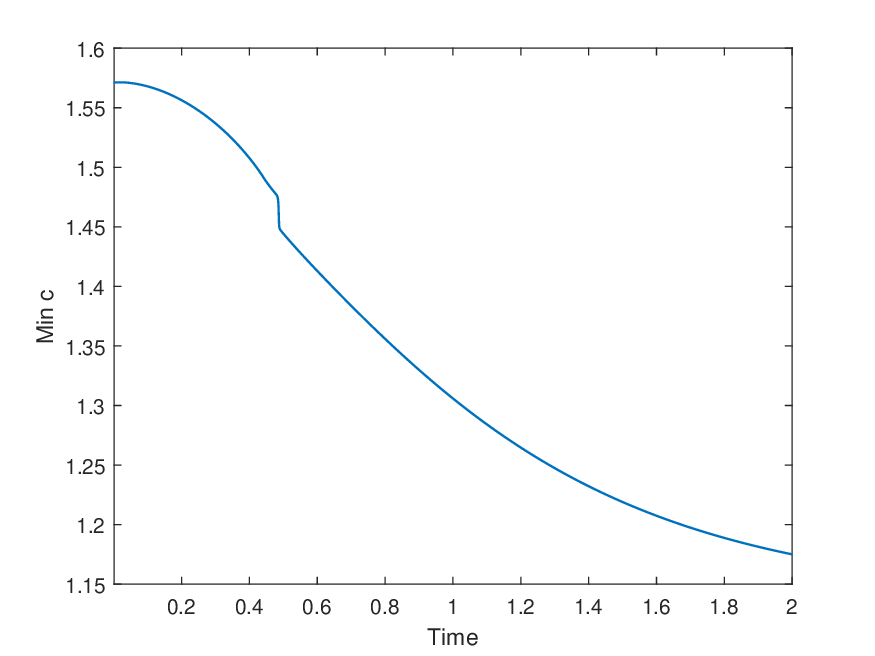}
	\includegraphics[width=0.32\linewidth]{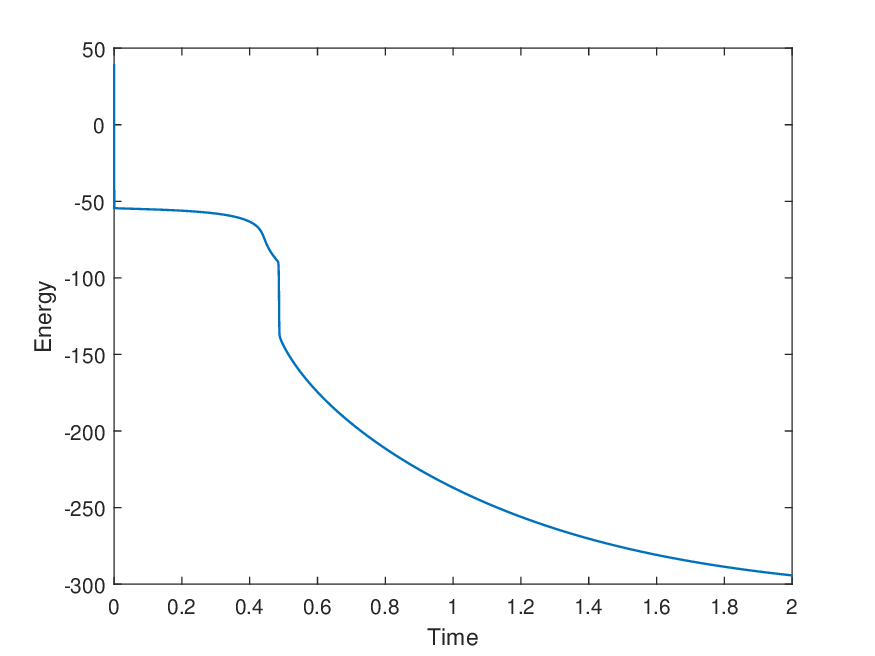}
	\includegraphics[width=0.32\linewidth]{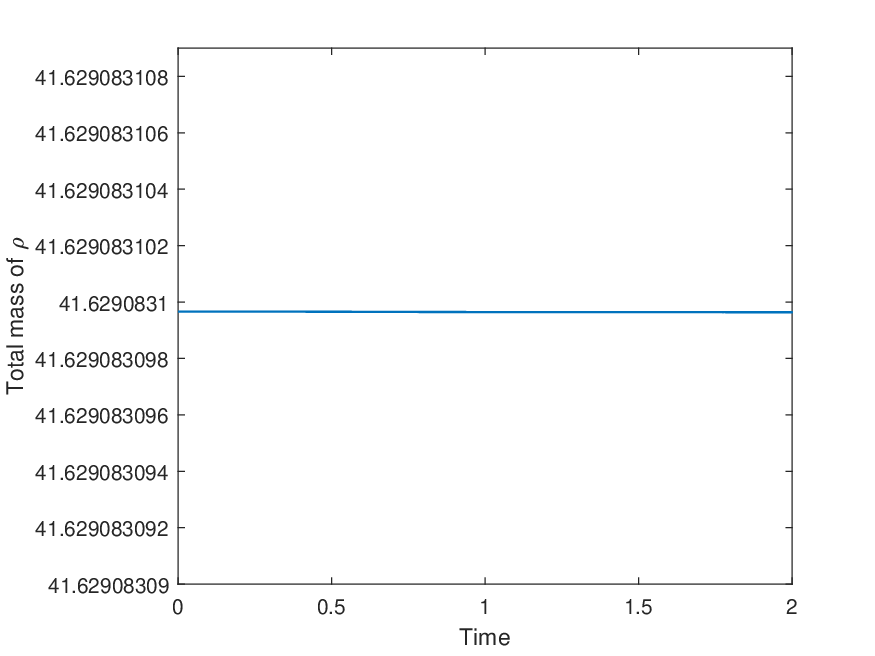}
	\caption{Evolutions of cell density, chemoattractant concentration, discrete energy and total cell mass for the BE-BCFD scheme on uniform grids with $M=100$ for Example \ref{exam:grate8pi:pe}. }
	\label{fig:PE-rho100M100-uni}
\end{figure}

\begin{example}\label{exam:grate8pi:pe}(Blow-up)  In order to explore the chemotactic blow-up for the parabolic-elliptic Keller–Segel system, we increase the initial total mass of cell density to $M_\rho (0) \approx 41.62>8 \pi$, by changing the initial cell density of Example \ref{exam:less8pi:pe} to
 \begin{equation*}
  \rho^o( x, y)=\frac{100}{1+40\left(x^2+y^2\right)}, \quad \text{in} \quad \Omega=(-2,2)^2.
\end{equation*}
\end{example}

As the cell density is expected to blow up within a short finite time, we set $\tau=5\times 10^{-4}$ small enough. Evolutions of the cell density $\rho_h$, the chemoattractant concentration $c_h$, the discrete energy, and the total mass of $\rho_h$ are depicted in Figs. \ref{fig:PE-rho100M40-uni}--\ref{fig:PE-rho100M100-uni} for three different spatial grid configurations: uniform grids with $M = 40$, non-uniform grids with $M = 40$ and $\gamma = 1.29$, and uniform grids with $M = 100$. Additionally, Fig. \ref{fig:PE-rho100M40-uni-solution} presents contour plots of $\rho_h$ on various spatial grids at time instants $t = 0.1,\ 0.8,\ 2$ for the BE-BCFD scheme. Our observations indicate that:
\begin{enumerate}[(i)]
    \item Although blow‑up occurs in this scenario, the BE‑BCFD scheme still effectively reproduces the physical properties of positivity preservation, energy dissipation, and mass conservation established in the theoretical analysis, across all three distinct types of spatial grids. 
    
    \item  As can be seen, both the coarser non-uniform grid ($M=40$, $\gamma=1.29$) and the finer uniform grid ($M=100$) produce nearly the same blow-up phenomenon, characterized by significant gradient variations and a maximum value of cell density exceeding $2.5\times10^4$. In contrast, on the coarser uniform grid ($M=40$), although notable gradient variations are also observed, the maximum value of $\rho_h$ only reaches about $4000$. Moreover, the blow-up time is accurately captured by the coarser non-uniform grid and the finer uniform grid. This demonstrates that the non-uniform grid BCFD method achieves results comparable to those obtained on the fine uniform grid and performs significantly better than the coarser uniform grid approach.
\end{enumerate}
In summary, these findings highlight the efficiency and accuracy of the non-uniform grid BCFD method in simulating blow-up phenomena.

\begin{figure}[!bt]
  \vspace{-12pt}
  \centering
    \includegraphics[width=0.32\linewidth]{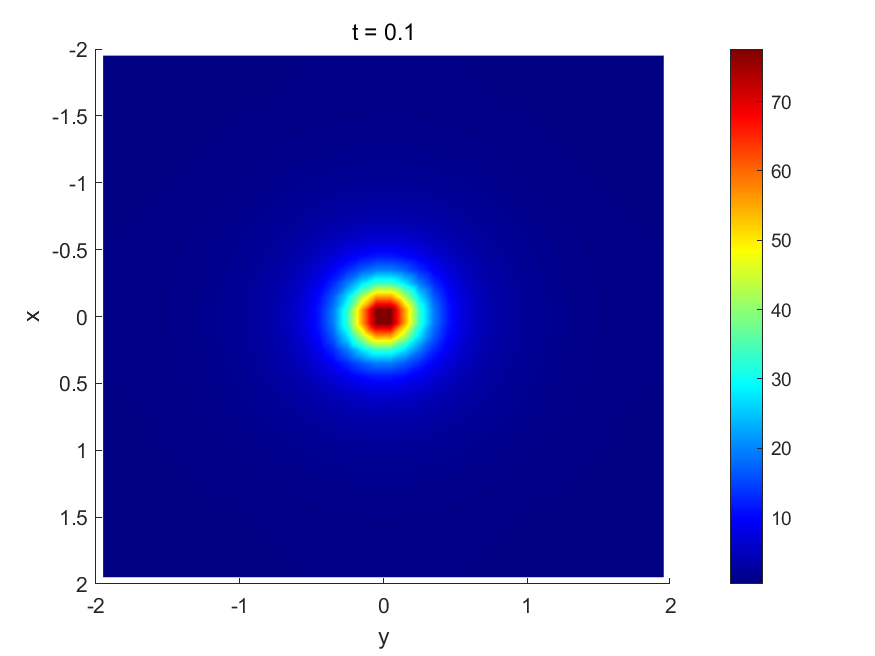}   
    \includegraphics[width=0.32\linewidth]{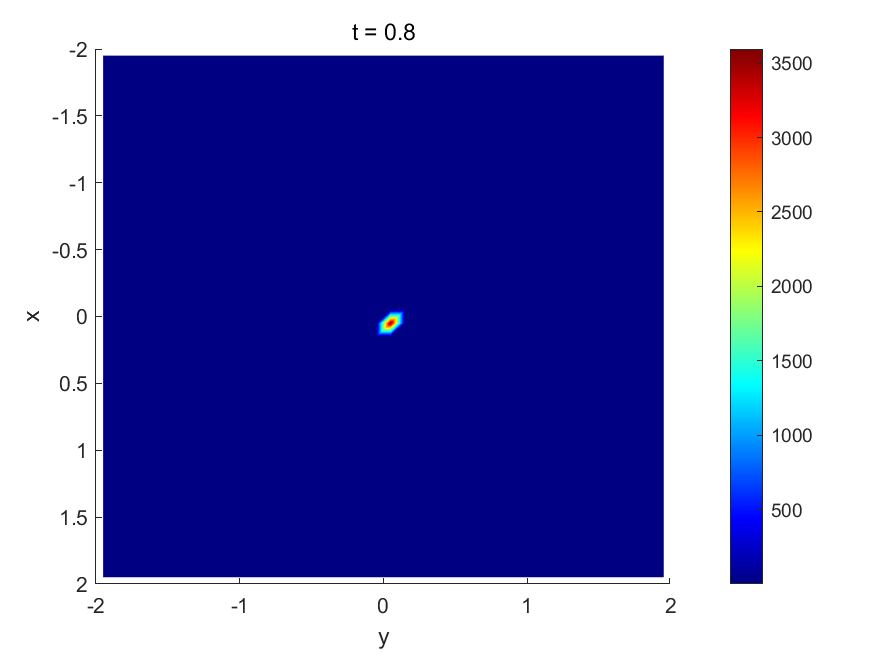}
    \includegraphics[width=0.32\linewidth]{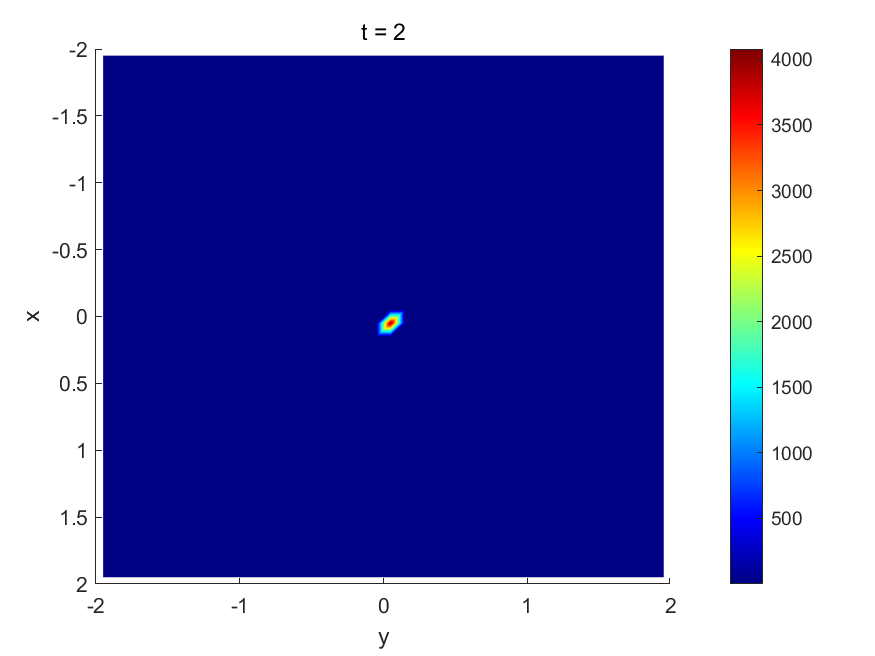}
    
	\includegraphics[width=0.32\linewidth]{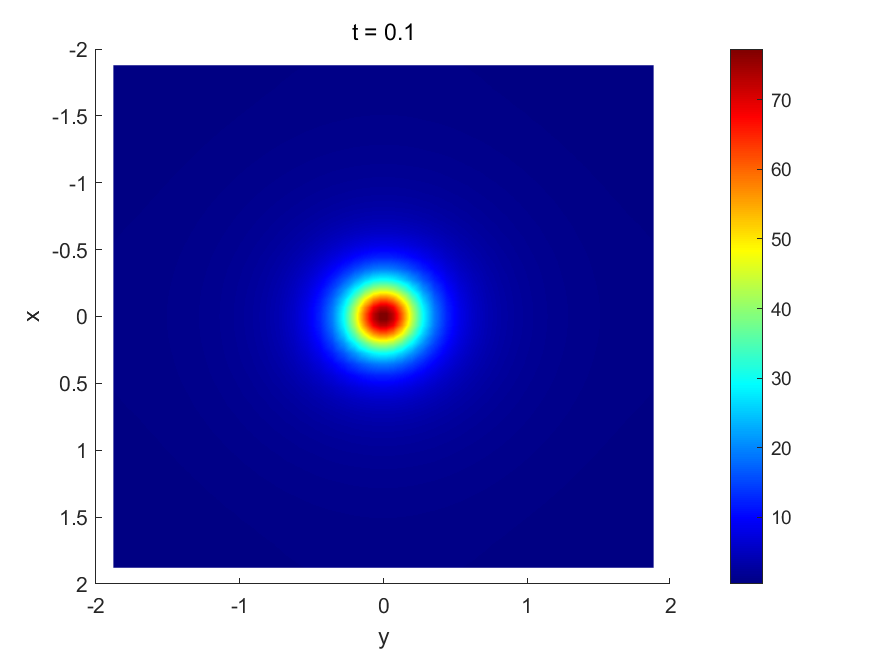}
     \includegraphics[width=0.32\linewidth]{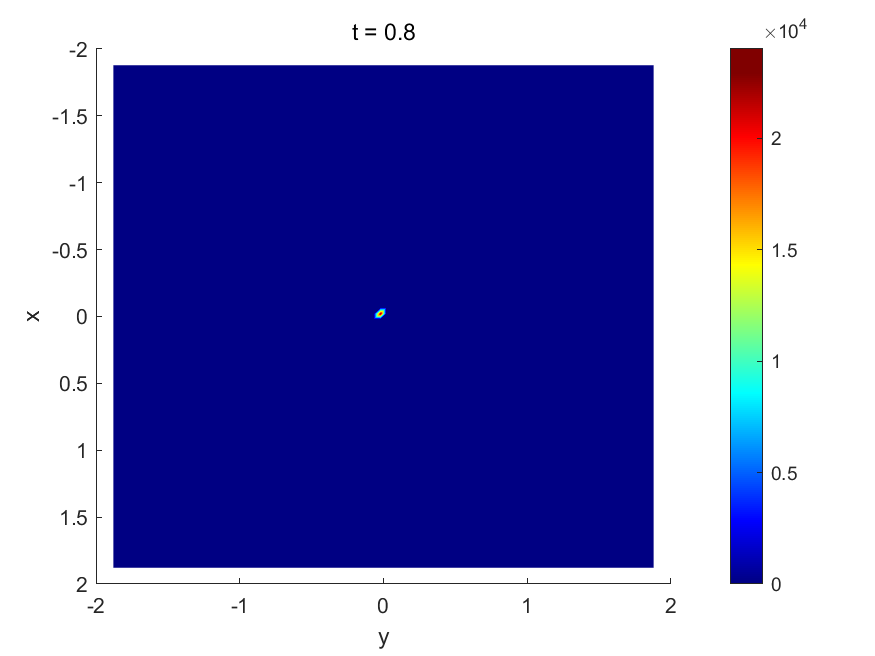}
     \includegraphics[width=0.32\linewidth]{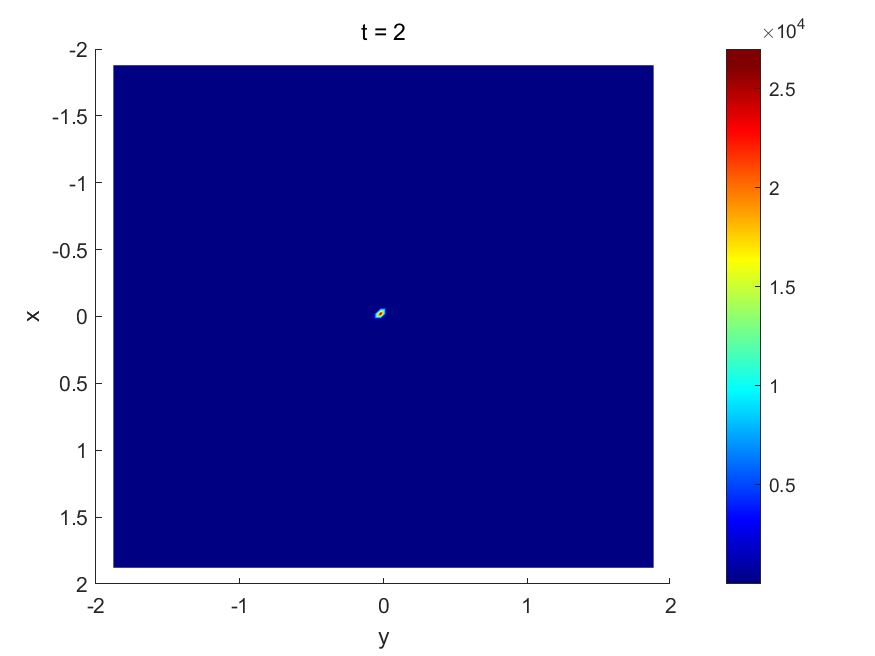}
     
     \includegraphics[width=0.32\linewidth]{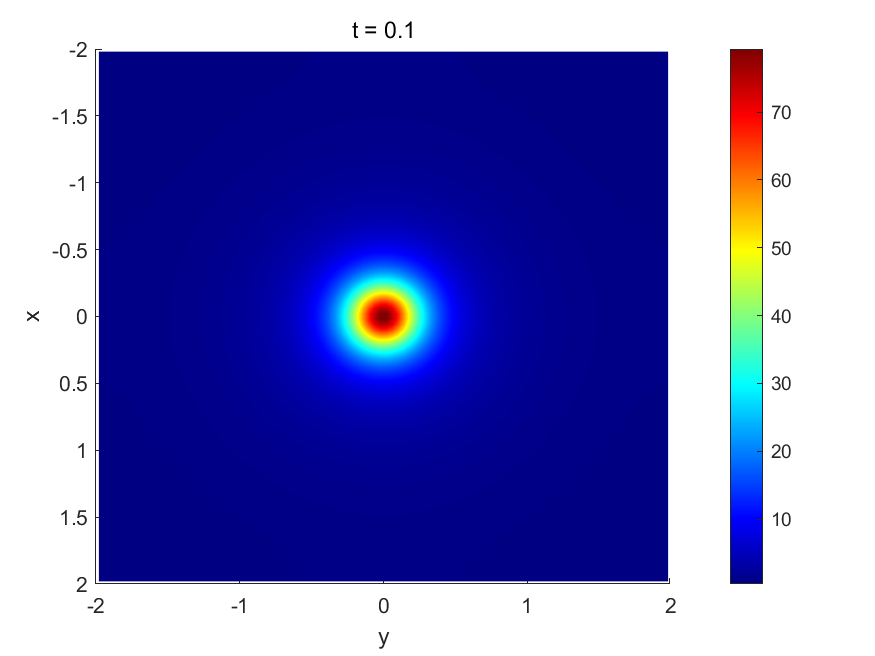}
     \includegraphics[width=0.32\linewidth]{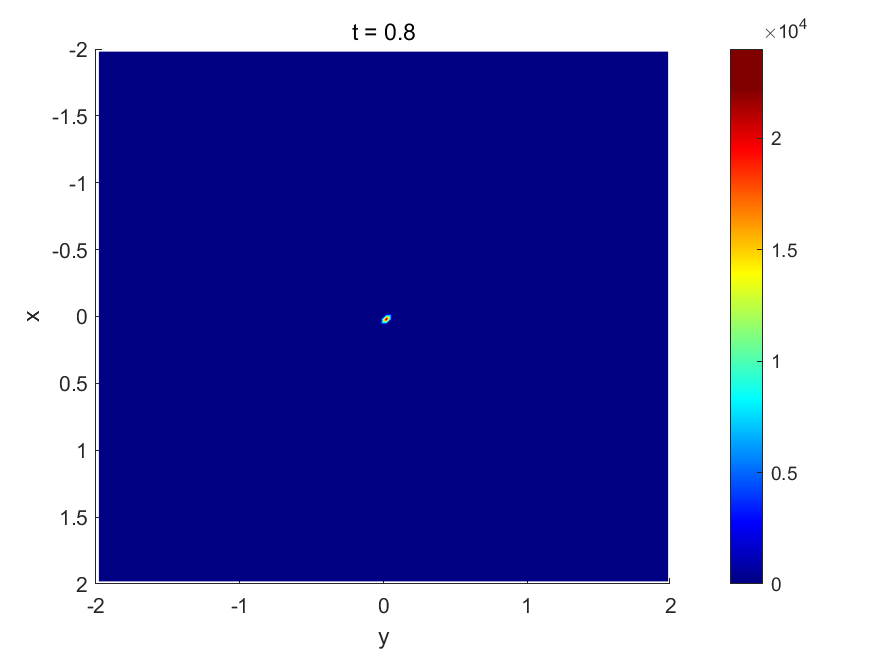}
     \includegraphics[width=0.32\linewidth]{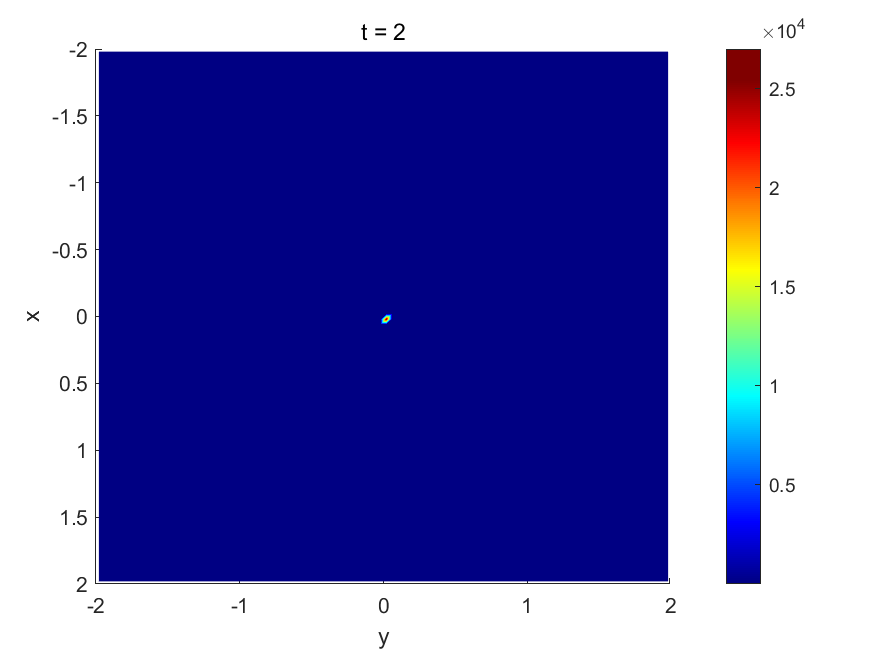}
	\caption{Contour plots of $\rho_h$ on uniform grids with $M=40$ (top), non-uniform grids with $M=40$ and $\gamma=1.29$ (middle), and uniform grids with $M=100$ (bottom) at $t = 0.1,\ 0.8,\ 2$ for the BE-BCFD scheme for Example \ref{exam:grate8pi:pe}.}
	\label{fig:PE-rho100M40-uni-solution}
\end{figure}

\subsection{Parabolic-parabolic case ($\varepsilon =1$)}\label{example:pp}
In this subsection, we utilize both the BE-BCFD scheme \eqref{BE-BCFD:rewrite}--\eqref{BE-BCFD:IBc:rewrite} and the PC-BCFD scheme \eqref{PC-BCFD:rewrite}--\eqref{PC-BCFD:IBc:rewrite} to simulate various dynamic processes for the parabolic-parabolic Keller-Segel system ($\varepsilon=1$). Throughout these simulations, we also validate key physical properties such as mass conservation, positivity preservation, and energy dissipation for both schemes. Moreover, as in the parabolic-elliptic Keller--Segel system, the same critical threshold $M_{\rho}(0)=8\pi$ governs the global existence versus blow-up for the parabolic-parabolic Keller–Segel system.

\begin{example}\label{exam:less8pi:pp} (Global existence) For this example, we set the initial cell density and chemoattractant concentration as follows:
\begin{equation*}
		\rho^o(x, y)=11 \exp \left(-(x^2+y^2)\right),\quad
		c^o(x, y)=5 \exp \left(-(x^2+y^2)/2\right) ,
\end{equation*}
where the initial total mass of cells $M_\rho (0) \approx 24.54 < 8\pi$ in $\Omega=(-1,1)^2$. 
\end{example}

We take the parameters as $M = 40$ and $\tau = 1/M$. Figs. \ref{fig:PP-smooth-M40-uni}--\ref{fig:PP-smooth-m40-beta04} show the evolutions of the cell density $\rho_h$, the chemoattractant concentration $c_h$, the discrete energy, and the total mass of $\rho_h$ on both uniform spatial grids and non-uniform spatial grids with $\beta=0.4$. We notice that the BE-BCFD scheme \eqref{BE-BCFD:rewrite}--\eqref{BE-BCFD:IBc:rewrite} successfully preserves the essential properties of positivity preservation, energy dissipation, and mass conservation when applied to the parabolic-parabolic Keller–Segel system. Additionally, the random perturbations of the uniform spatial grids do not adversely affect the simulation results.
\begin{figure}[!t]
	\centering
	\includegraphics[width=0.32\linewidth]{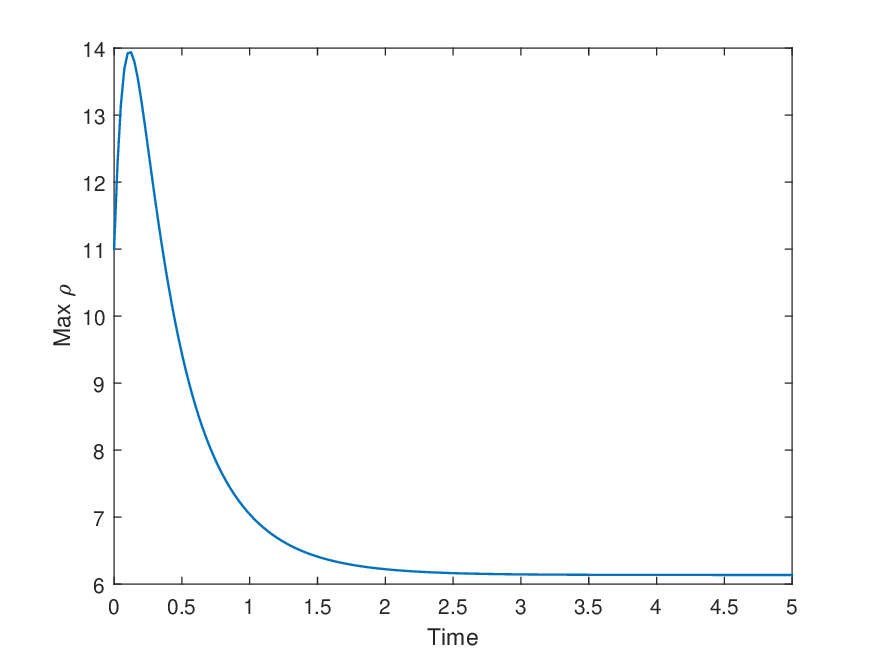}
	\includegraphics[width=0.32\linewidth]{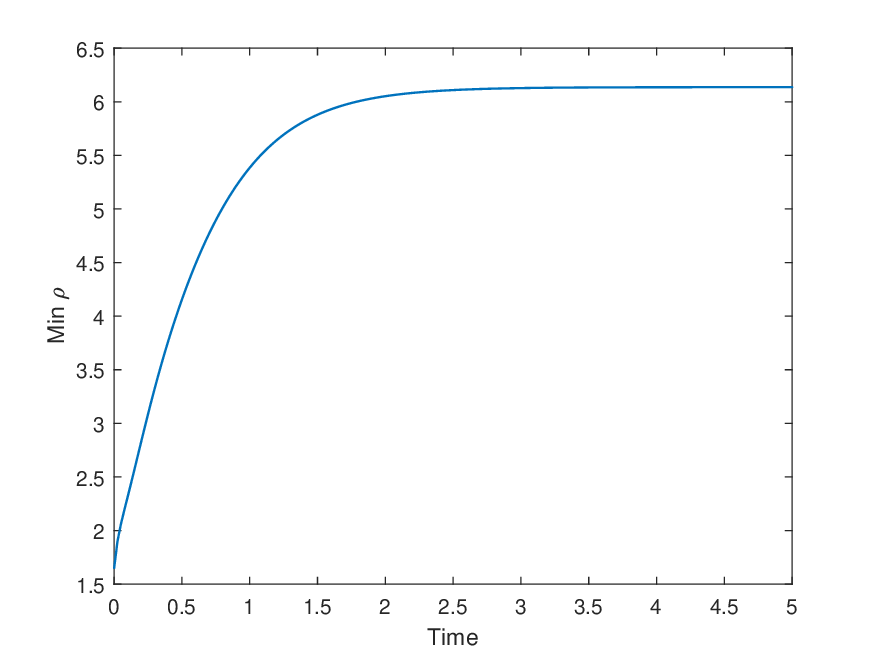}
	\includegraphics[width=0.32\linewidth]{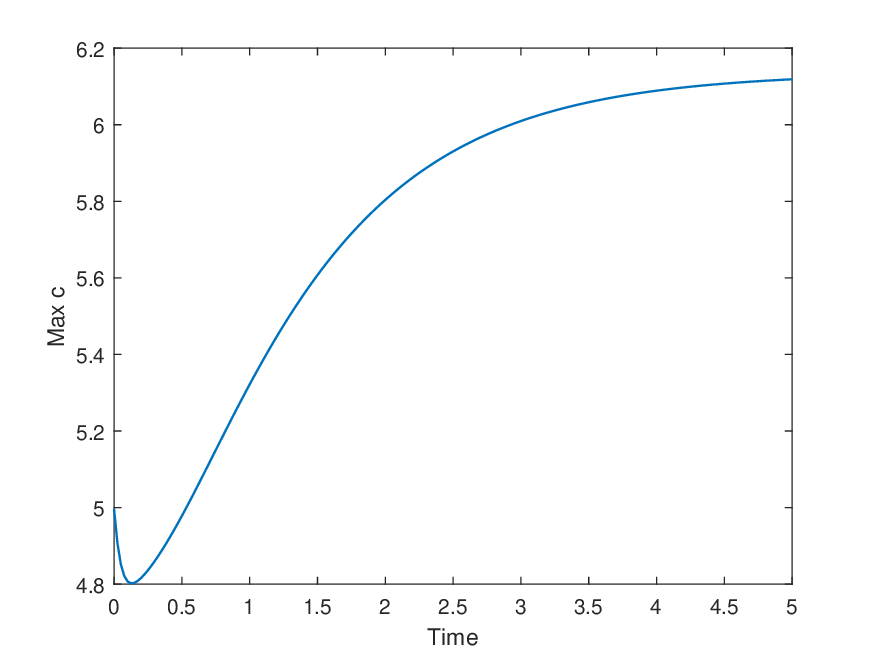}
	\includegraphics[width=0.32\linewidth]{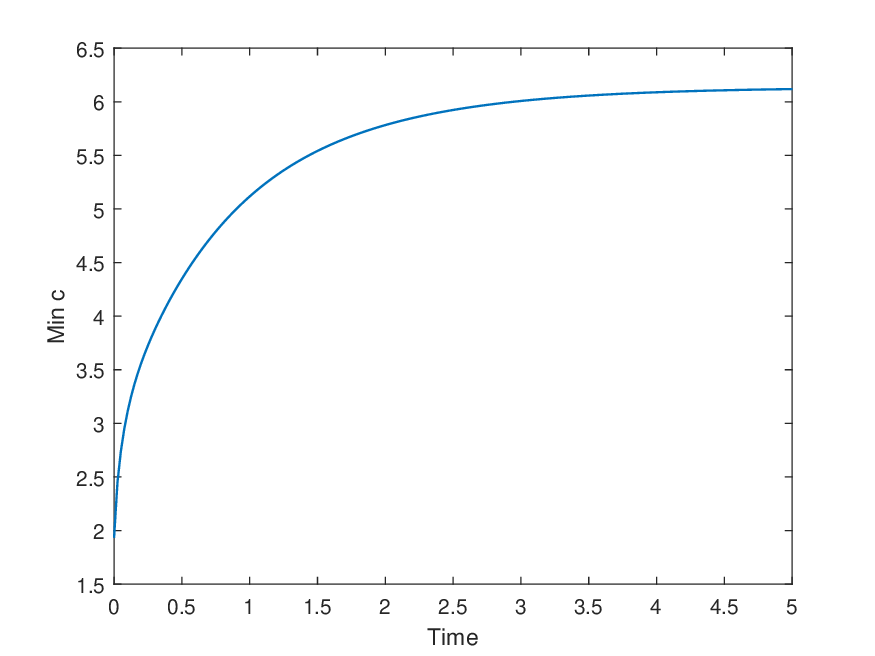}
	\includegraphics[width=0.32\linewidth]{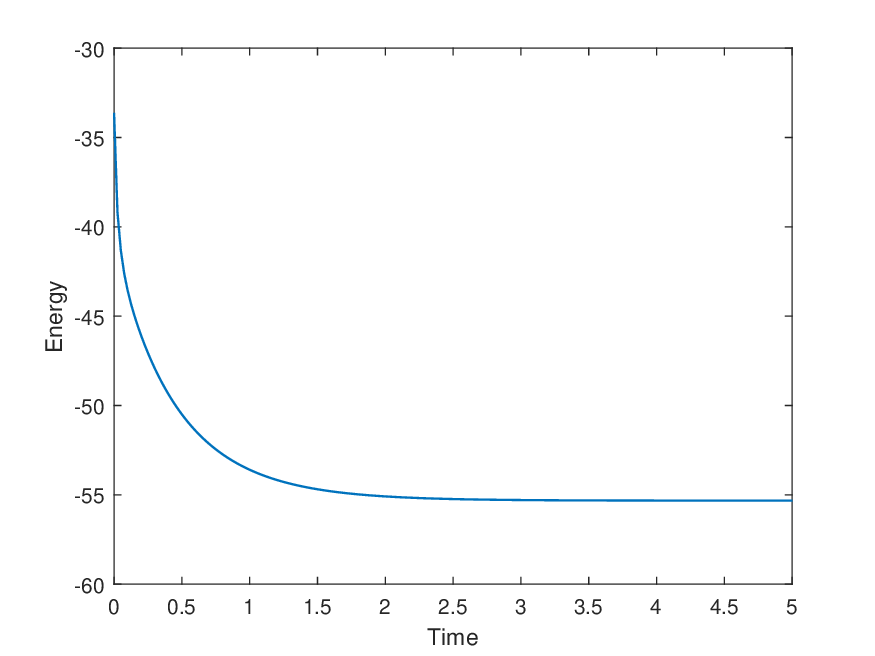}
	\includegraphics[width=0.32\linewidth]{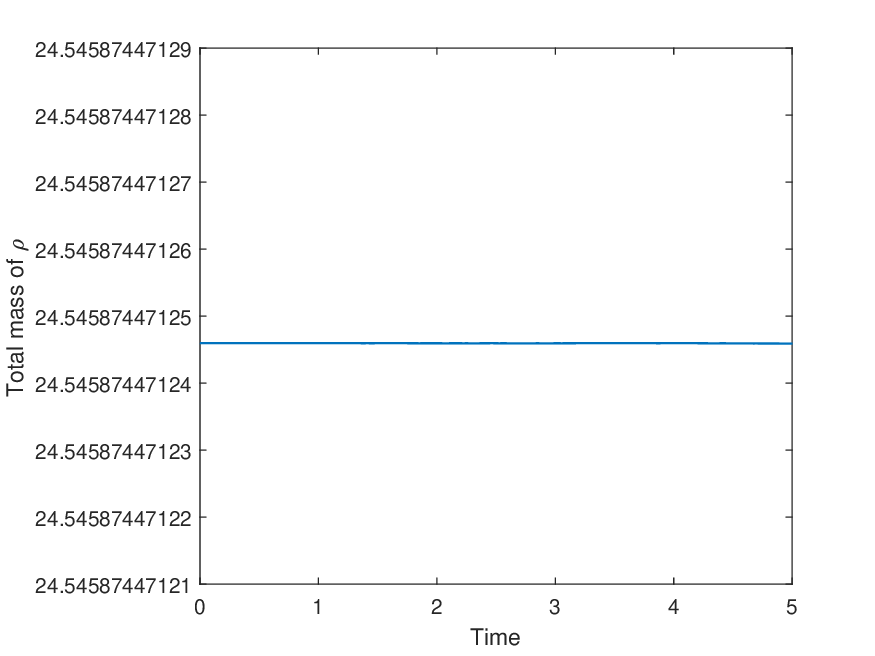}
	\caption{Evolutions of cell density, chemoattractant concentration, discrete energy and total cell mass for the BE-BCFD scheme on uniform grids with $M=40$ for Example \ref{exam:less8pi:pp}. }
	\label{fig:PP-smooth-M40-uni}
\end{figure}
\begin{figure}[!t]
	\centering
	\includegraphics[width=0.32\linewidth]{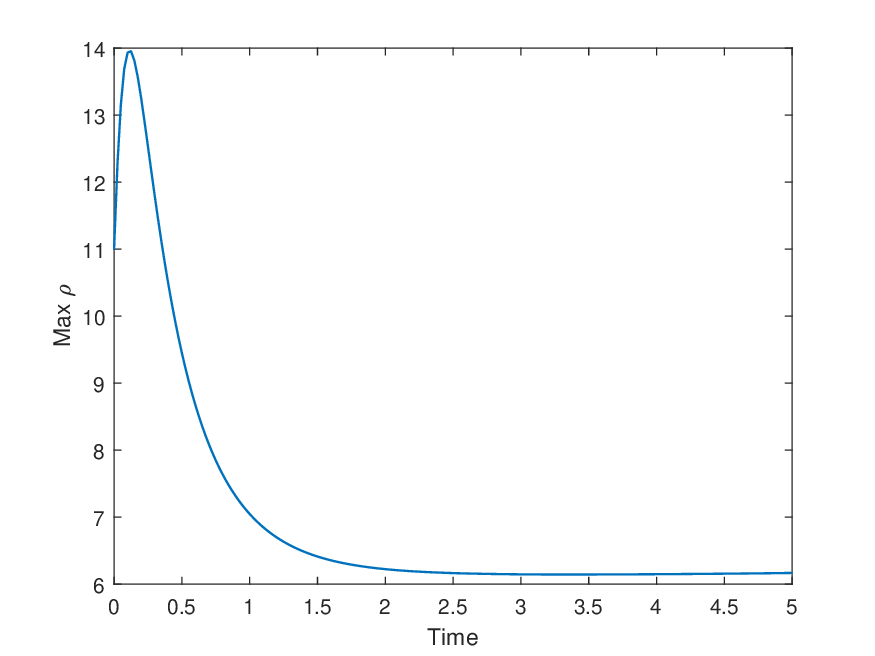}
	\includegraphics[width=0.32\linewidth]{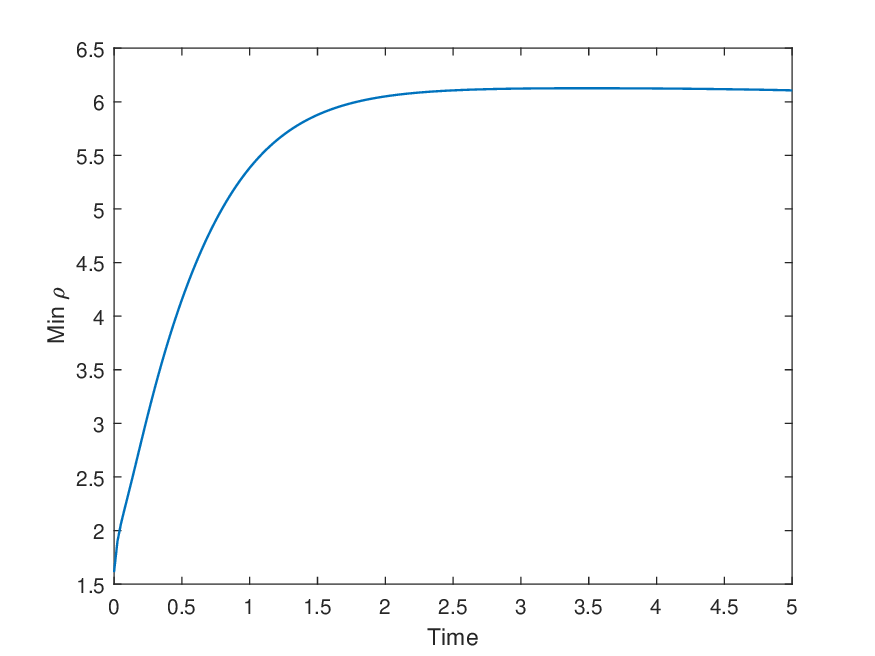}
	\includegraphics[width=0.32\linewidth]{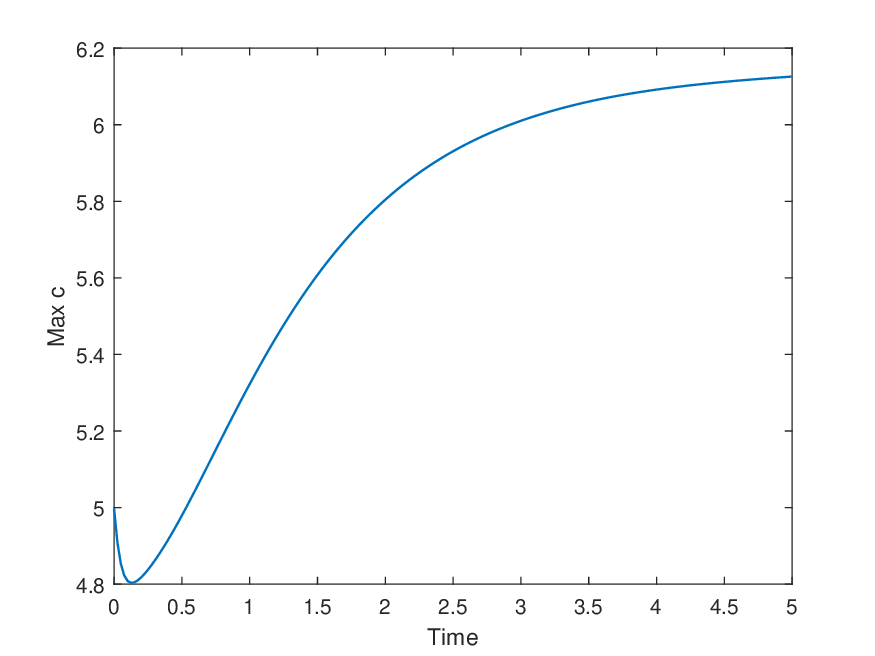}
	\includegraphics[width=0.32\linewidth]{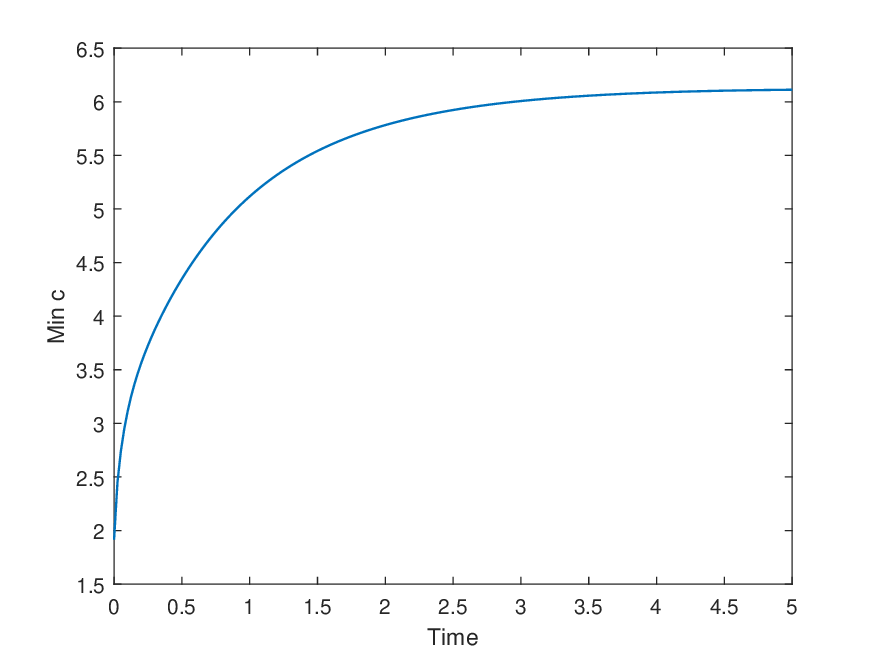}
	\includegraphics[width=0.32\linewidth]{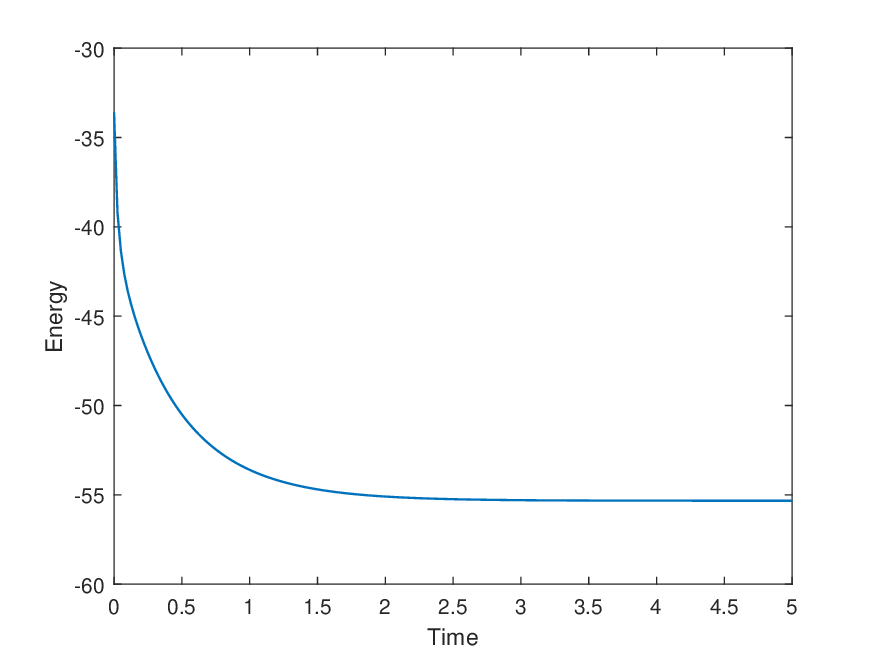}
	\includegraphics[width=0.32\linewidth]{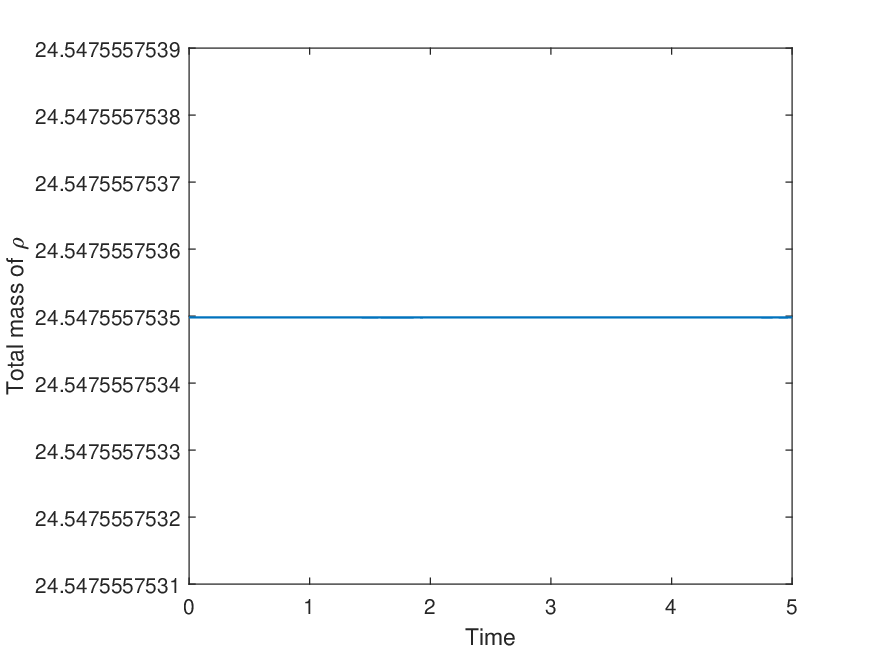}
	\caption{Evolutions ofcell density, chemoattractant concentration, discrete energy and total cell mass for the BE-BCFD scheme on non-uniform grids with $M=40$ and $\beta=0.4$ for Example \ref{exam:less8pi:pp}. }
	\label{fig:PP-smooth-m40-beta04}
\end{figure}

\begin{example}\label{exam:grate8pi:pp}	(Blow-up) 
In order to further explore the chemotactic blow-up for the parabolic-parabolic Keller-Segel system, we increase the initial total cell mass to $M_\rho (0) \approx 31.42>8\pi$, by changing the initial cell density and chemoattractant concentration as
	\begin{equation*}
		\rho^o(x, y)=130 \exp \left(-13 (x^2+y^2 )\right),\quad
	    c^o(x, y)=13 \exp \left(-2 (x^2+y^2 )\right),
	\end{equation*}
	in the domain $\Omega=(-1,1)^2$. 
\end{example}

\begin{figure}[!t]
	\centering
	\includegraphics[width=0.32\linewidth]{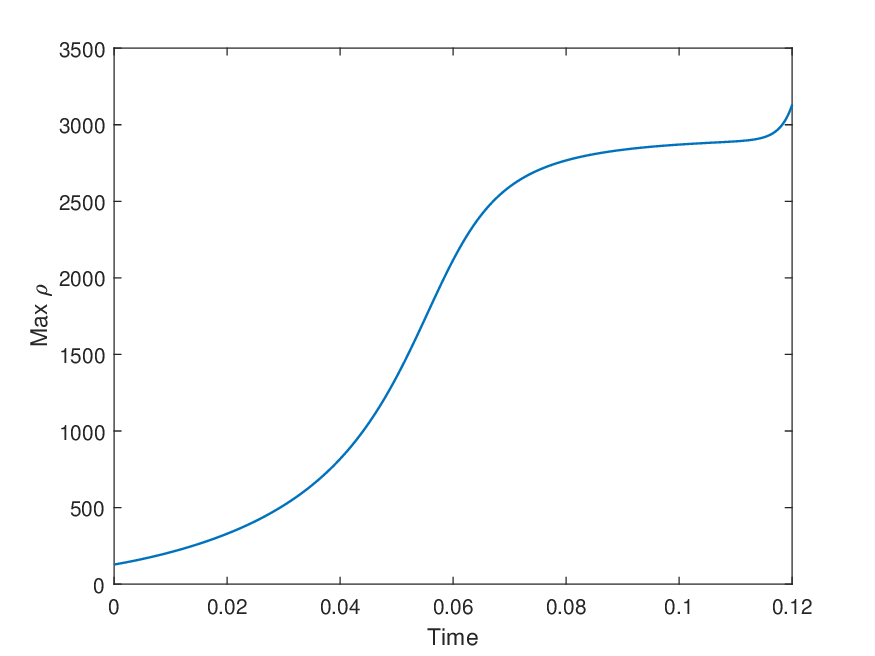}
	\includegraphics[width=0.32\linewidth]{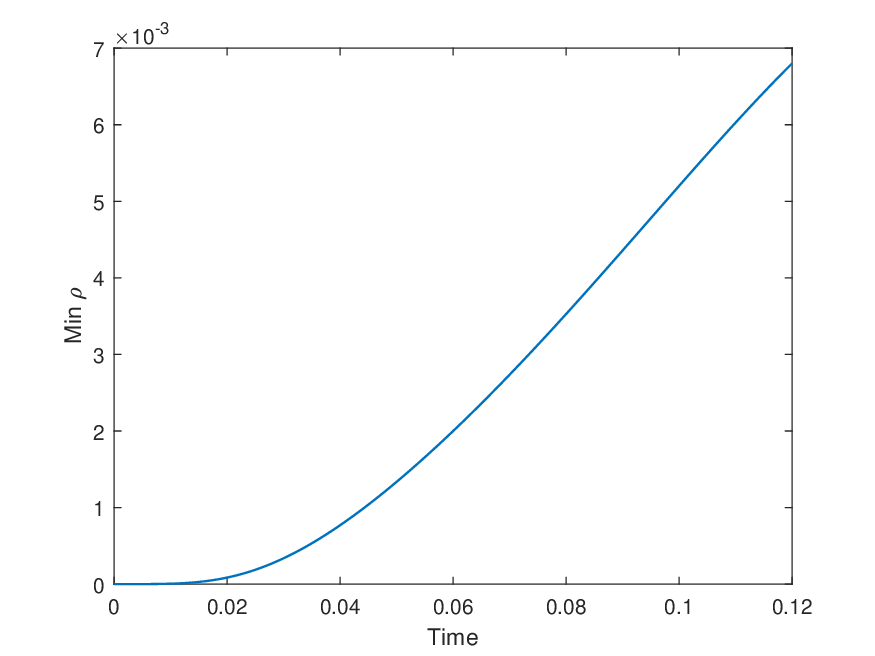}
	\includegraphics[width=0.32\linewidth]{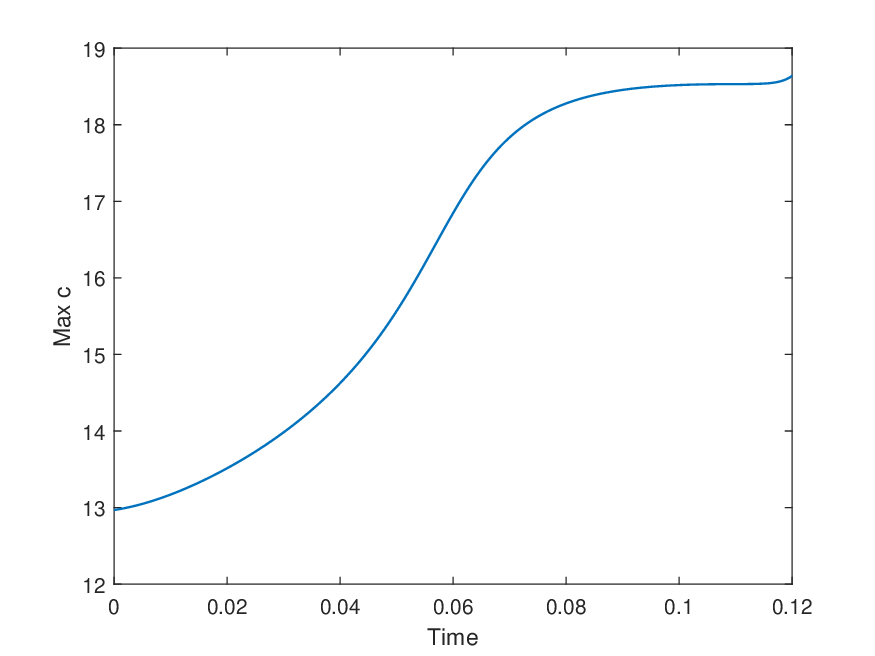}
	\includegraphics[width=0.32\linewidth]{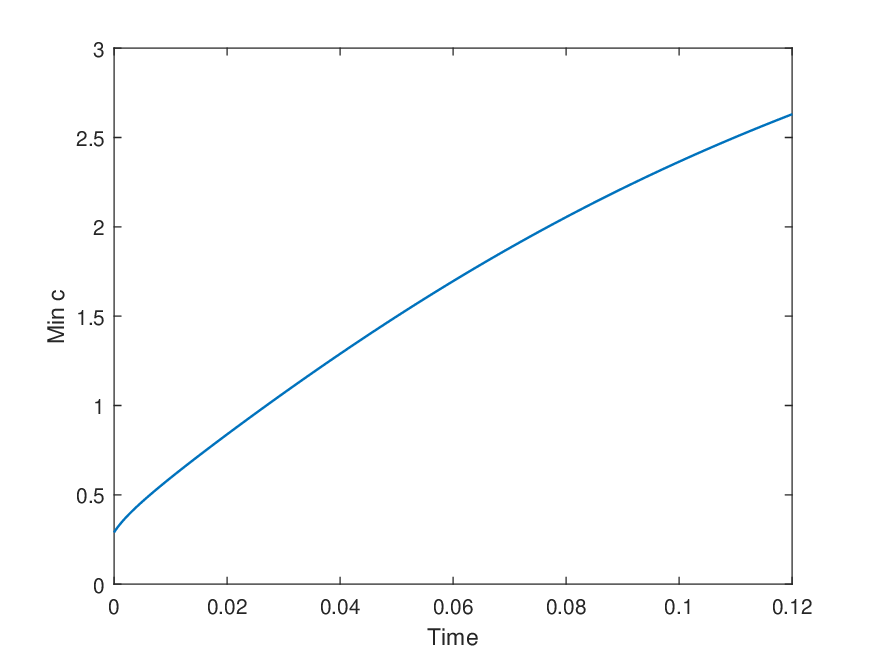}
	\includegraphics[width=0.32\linewidth]{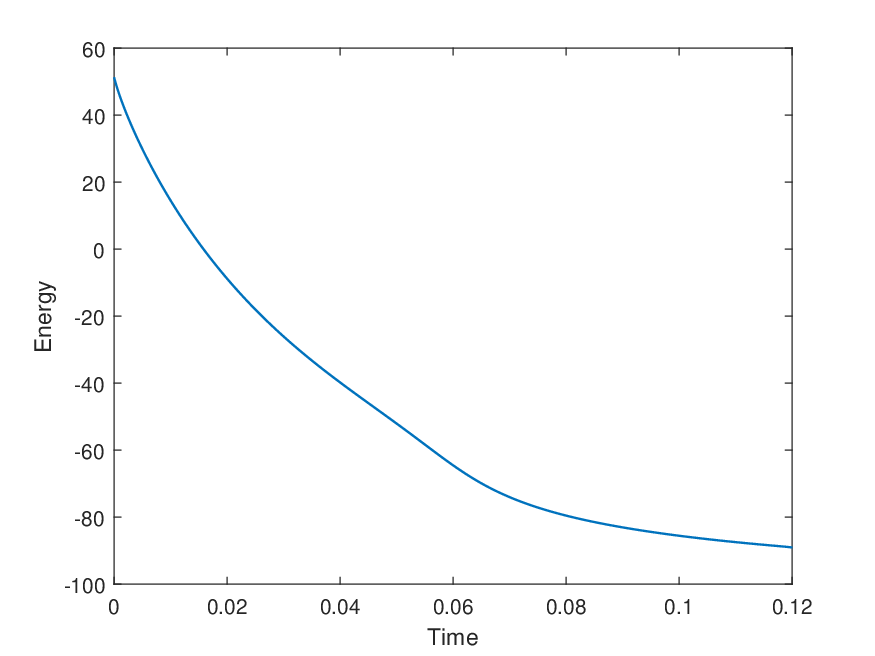}
	\includegraphics[width=0.32\linewidth]{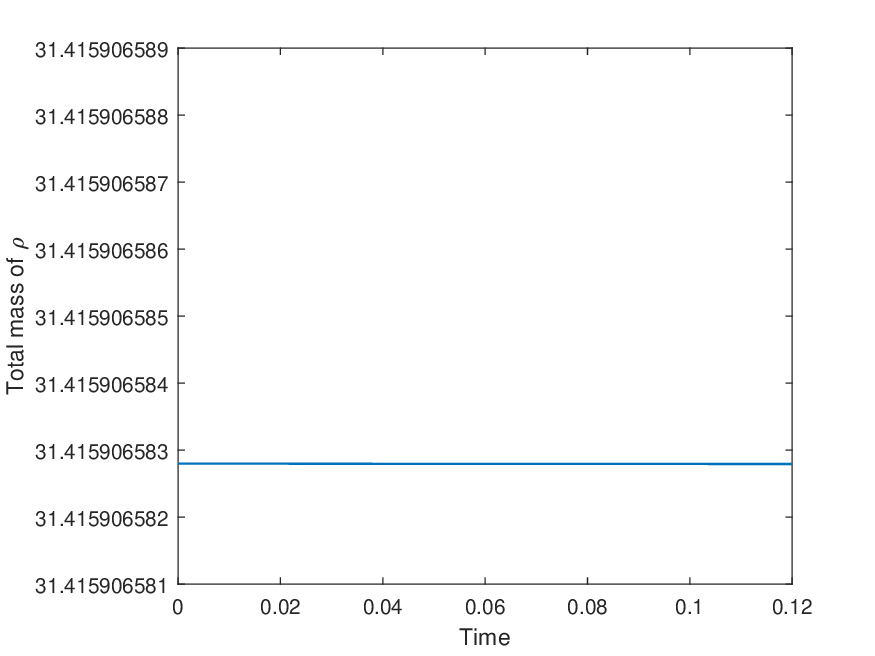}
	\caption{Evolutions of cell density, chemoattractant concentration, discrete energy and total cell mass for the PC-BCFD scheme on uniform grids with $M=40$ for Example \ref{exam:grate8pi:pp}. }
	\label{fig:PP-rhoBU-22M40-uni}
\end{figure}
\begin{figure}[!ht]
	\centering
	\includegraphics[width=0.32\linewidth]{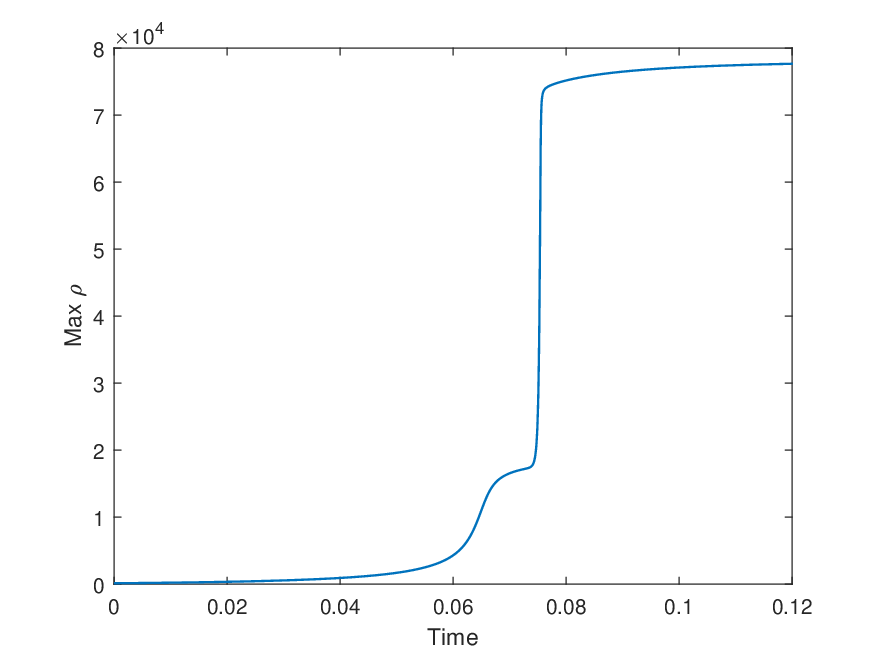}
	\includegraphics[width=0.32\linewidth]{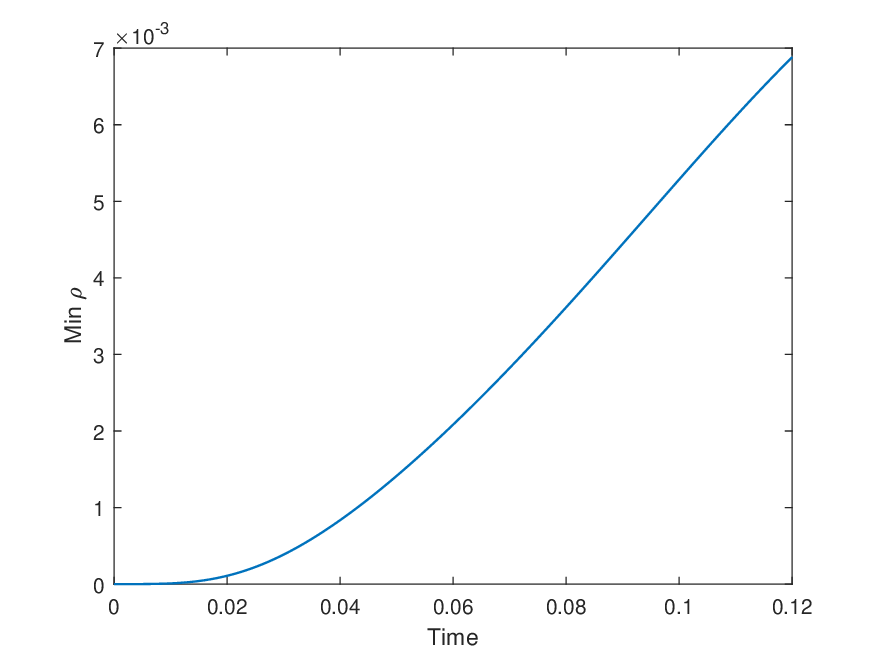}
	\includegraphics[width=0.32\linewidth]{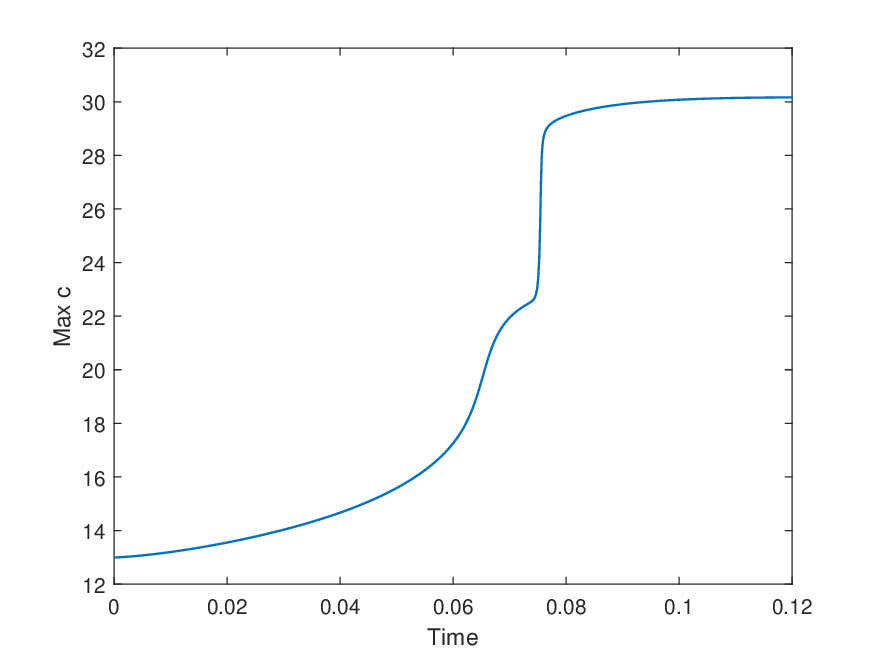}
	\includegraphics[width=0.32\linewidth]{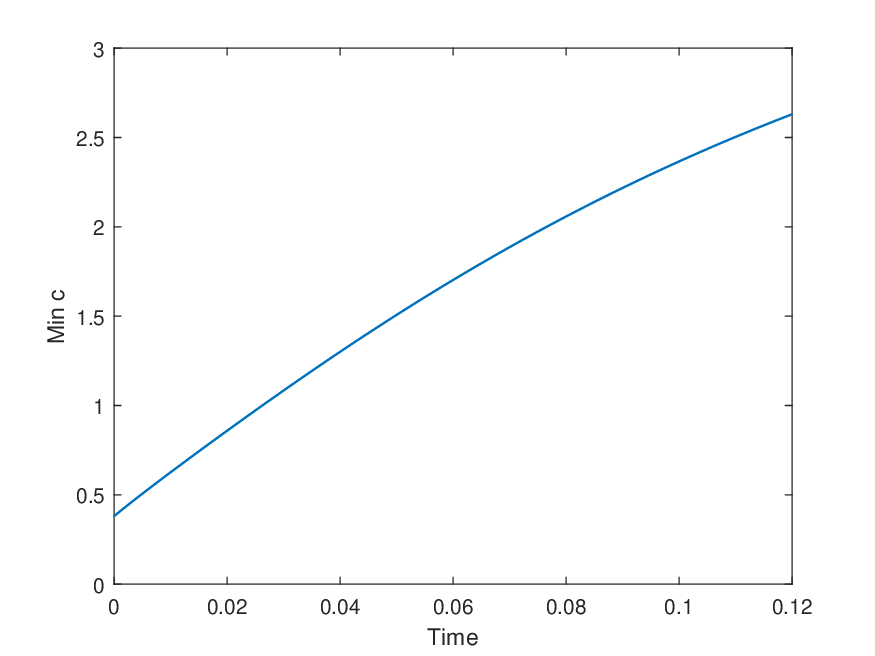}
	\includegraphics[width=0.32\linewidth]{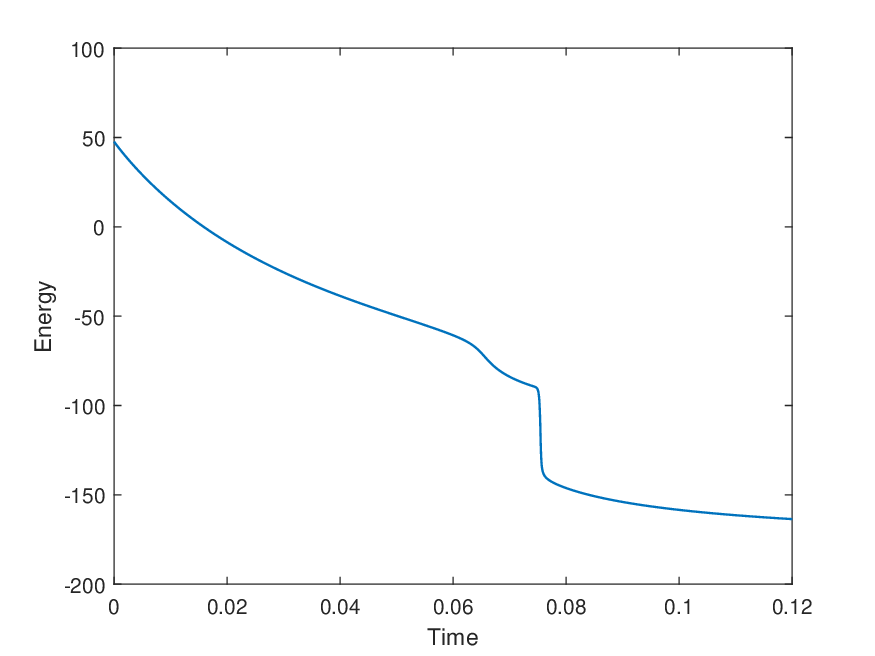}
	\includegraphics[width=0.32\linewidth]{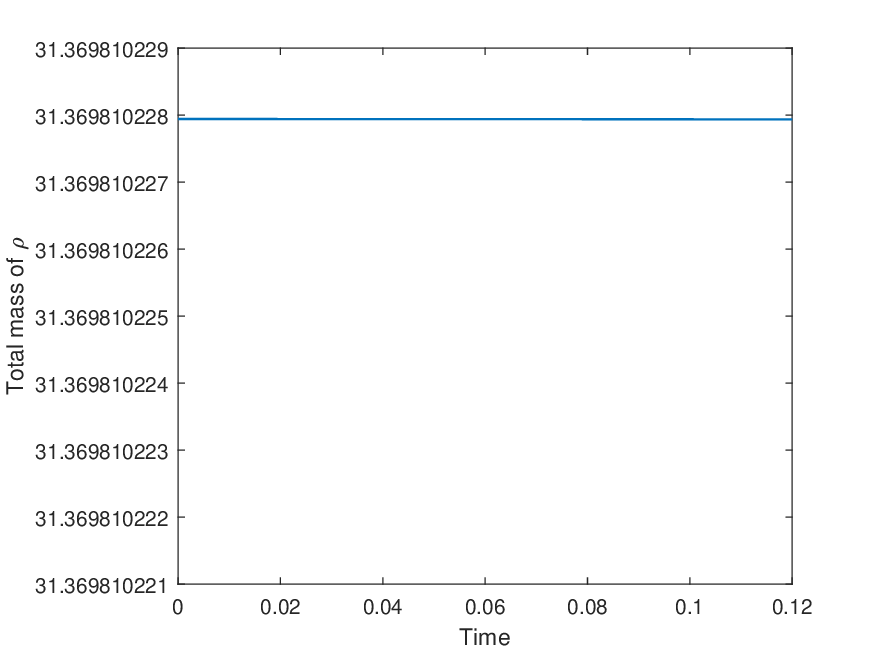}
	\caption{Evolutions of cell density, chemoattractant concentration, discrete energy and total cell mass for the PC-BCFD scheme on non-uniform grids with $M=40$ and $\gamma=1.285$ for Example \ref{exam:grate8pi:pp}. }
	\label{fig:PP-rhoBU-22M40-mid1285}
\end{figure}
\begin{figure}[!ht]
	\centering
	\includegraphics[width=0.32\linewidth]{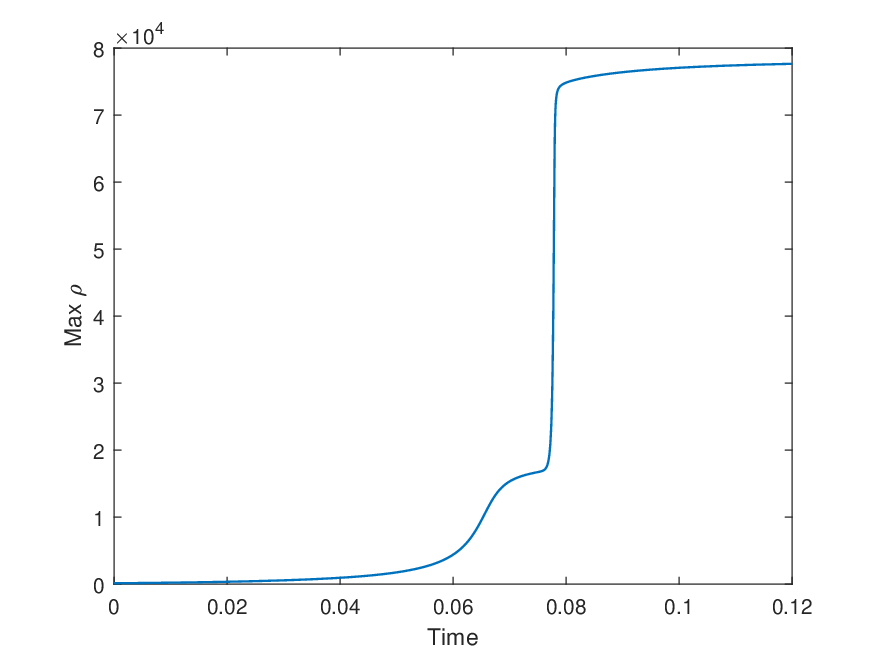}
	\includegraphics[width=0.32\linewidth]{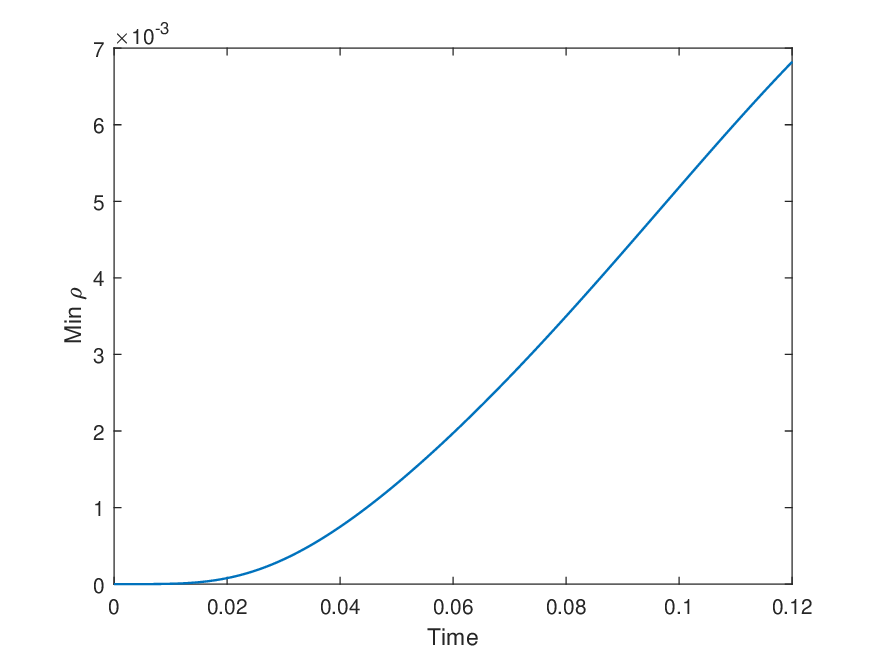}
	\includegraphics[width=0.32\linewidth]{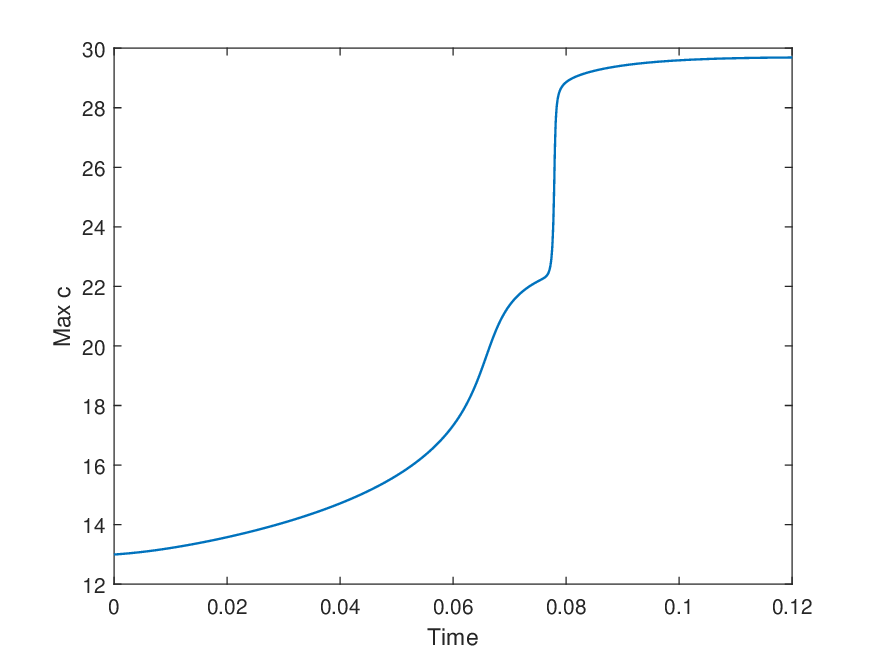}
	\includegraphics[width=0.32\linewidth]{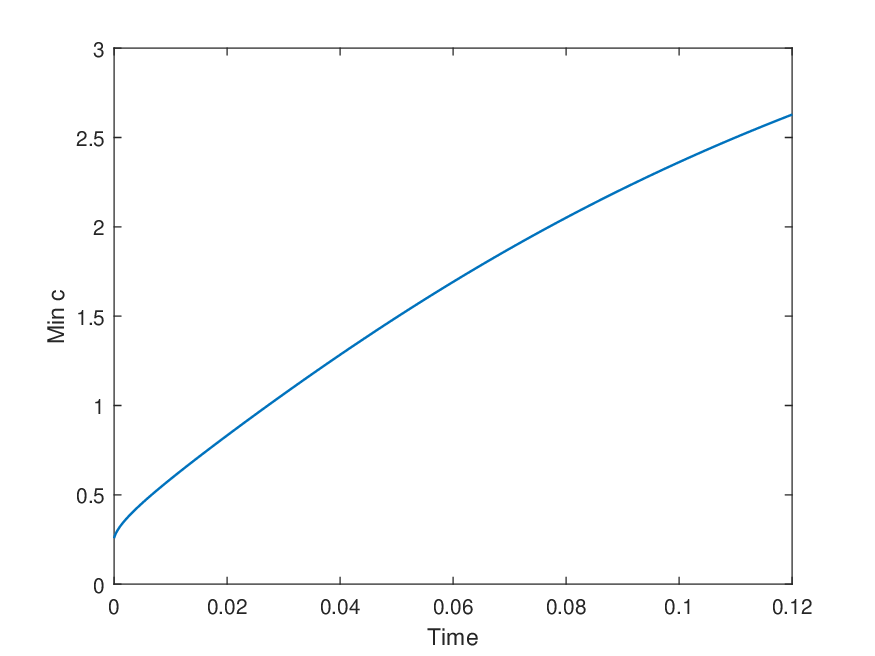}
	\includegraphics[width=0.32\linewidth]{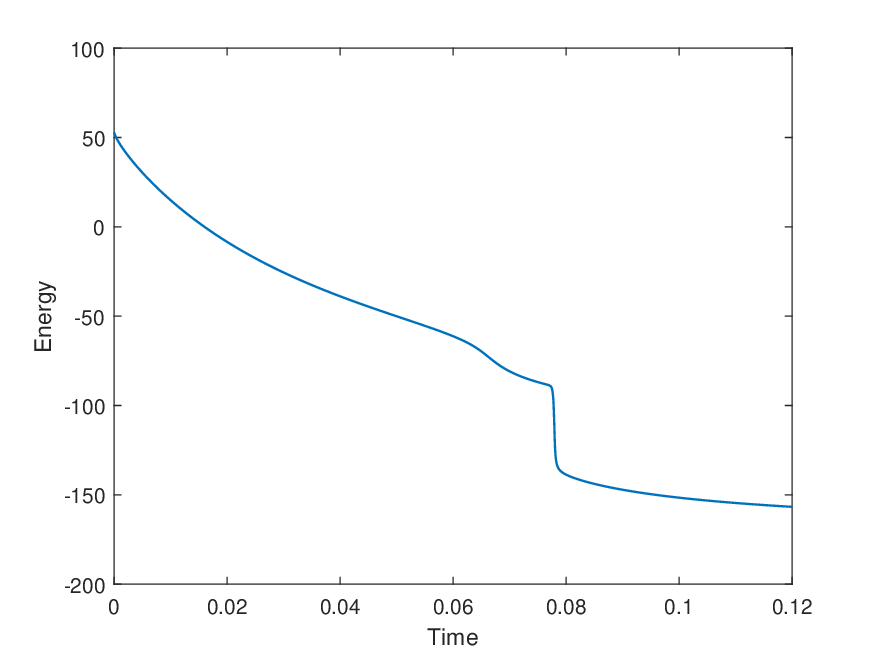}
	\includegraphics[width=0.32\linewidth]{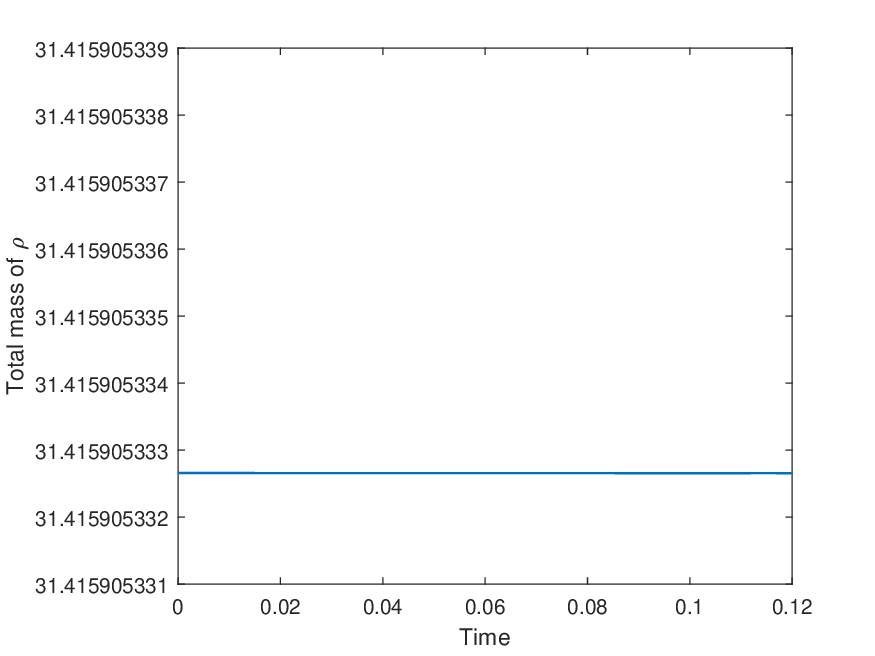}
	\caption{Evolutions of cell density, chemoattractant concentration, discrete energy and total cell mass for the PC-BCFD scheme on uniform grids with $M=100$ for Example \ref{exam:grate8pi:pp}. }
	\label{fig:PP-rhoBU-22M100-uni}
\end{figure}
In light of the expected rapid increase in cell density, we set $\tau=5\times 10^{-5}$ in the simulation. Evolutions of the cell density $\rho_h$, the chemoattractant concentration $c_h$, the discrete energy, and the total cell mass are illustrated under three grid configurations: uniform grids with $M=40$, non-uniform grids with $M=40$ and $\gamma=1.285$ described by \eqref{grid:mid}, and uniform grids with $M=100$. The simulation results for 
the PC-BCFD scheme are shown in 
Figs. \ref{fig:PP-rhoBU-22M40-uni}--\ref{fig:PP-rhoBU-22M100-uni}. We observe that:
\begin{enumerate}[(i)]
    \item The second-order scheme also uphold the fundamental physical principles of positivity preservation, energy dissipation, and mass conservation across both uniform and non-uniform spatial grids, although provable unconditional positivity preservation and energy dissipation are not available. These highlight the need for further theoretical investigations to design a positivity-preserving linear second-order numerical scheme.
    
    \item Nearly identical blow-up phenomena are observed on the coarser non-uniform grids ($M=40$, $\gamma=1.285$) and the finer uniform grids ($M=100$). In contrast, the simulated blow-up time is significantly delayed on the coarser uniform grids ($M=40$). This further underscores the necessity of using non‑uniform spatial grids in such simulations.

    \item As mentioned in Remark \ref{rm:positivity}, the time-step conditions in Theorem \ref{thm:pc-posi-preserve} are sufficient but not necessary. For instance, when simulation on the uniform grid with $M=100$, the time stepsize $\tau=5\times 10^{-5}$ indeed satisfies the sufficient condition $\tau\le h^2/2=2\times 10^{-4}$ for the concentration. In contrast, for the middle-refinement grid with $M=40$ and $\gamma=1.285$, we have $h^2/(2\sigma^6)\approx 2.34\times 10^{-6}$, which is much smaller than the actual time stepsize used in the computation. Nevertheless, the minimum values of both $\rho_h$ and $c_h$ remain positive throughout the simulation. 
\end{enumerate}

Furthermore, the contour plots and corresponding side views of $\rho_h$ at time instants $t = 0.01, 0.08$ and $0.12$ for
the PC-BCFD scheme are presented in 
Figs. \ref{fig:PP-rhoBU-22M100-uni-solution}–\ref{fig:PP-rhoBU-22M40-mid1285-solution}, in which similar simulation results are observed on the finer uniform grids ($M = 100$) and coarser non-uniform grids ($M = 40$, $\gamma = 1.285$). These again demonstrate the efficiency and accuracy of the non-uniform grid BCFD methods in modeling such blow-up phenomena.
%
\begin{figure}[!h]
	\centering
	\includegraphics[width=0.32\linewidth]{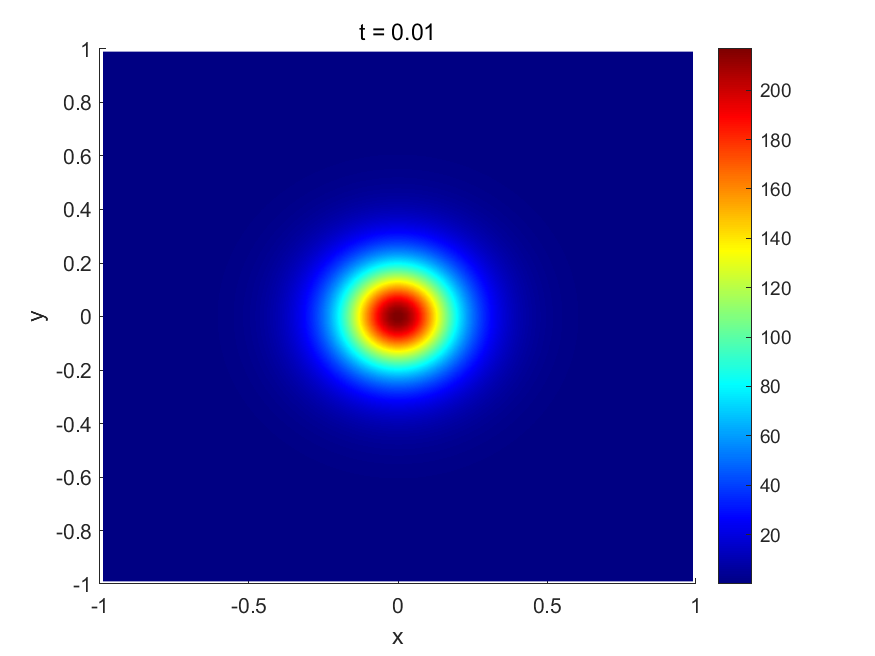}
	\includegraphics[width=0.32\linewidth]{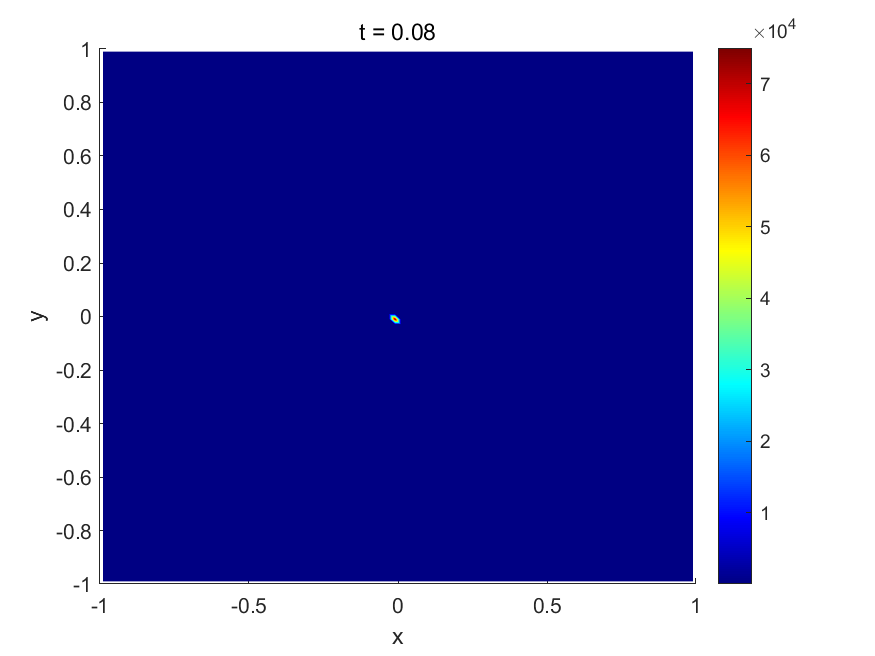}
	\includegraphics[width=0.32\linewidth]{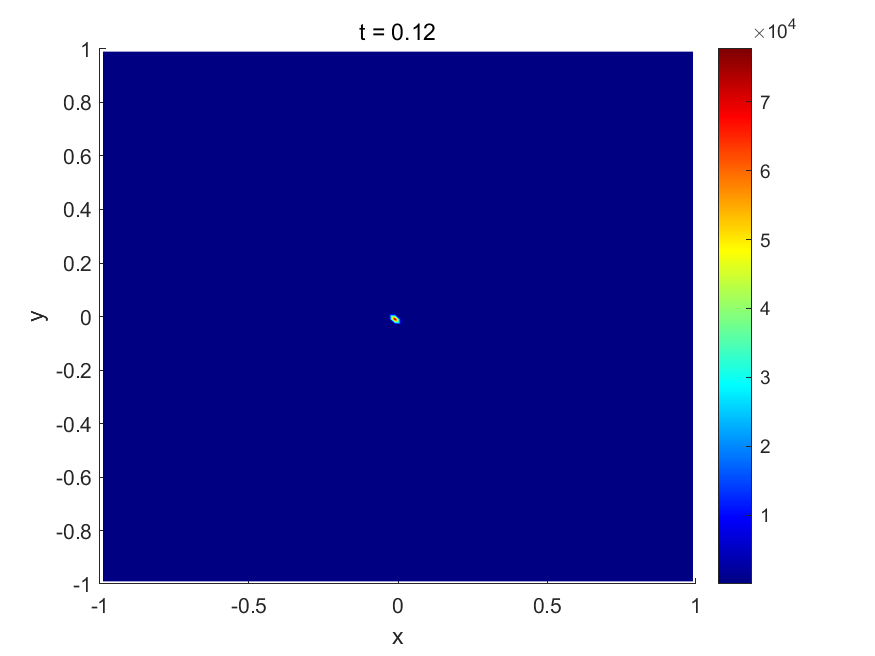}
	\includegraphics[width=0.32\linewidth]{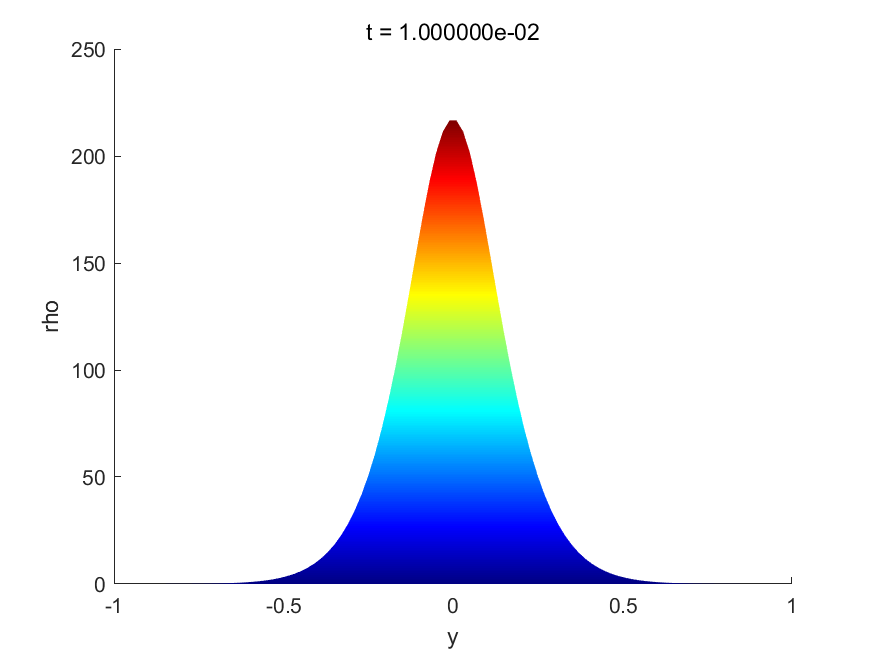}
	\includegraphics[width=0.32\linewidth]{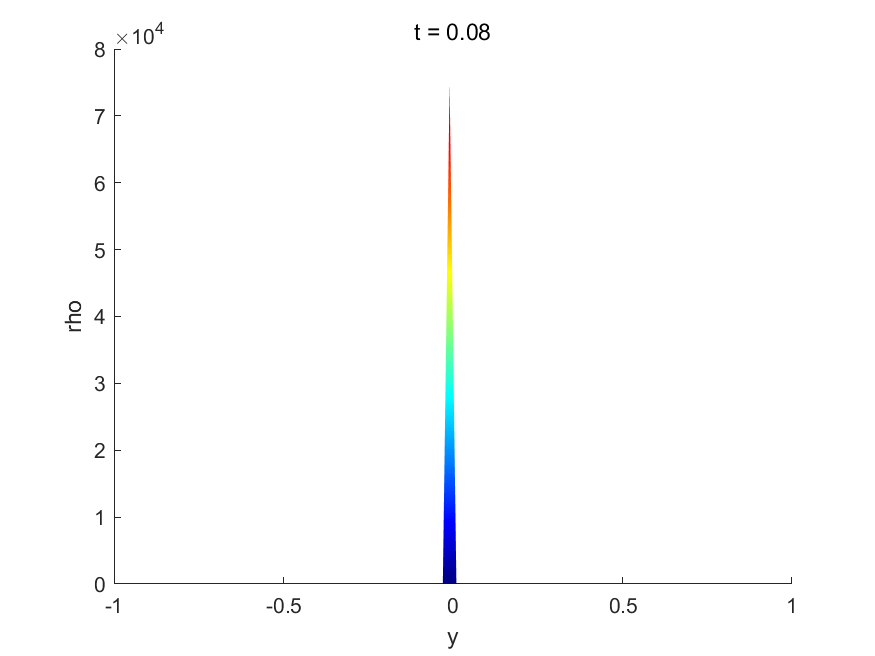}
	\includegraphics[width=0.32\linewidth]{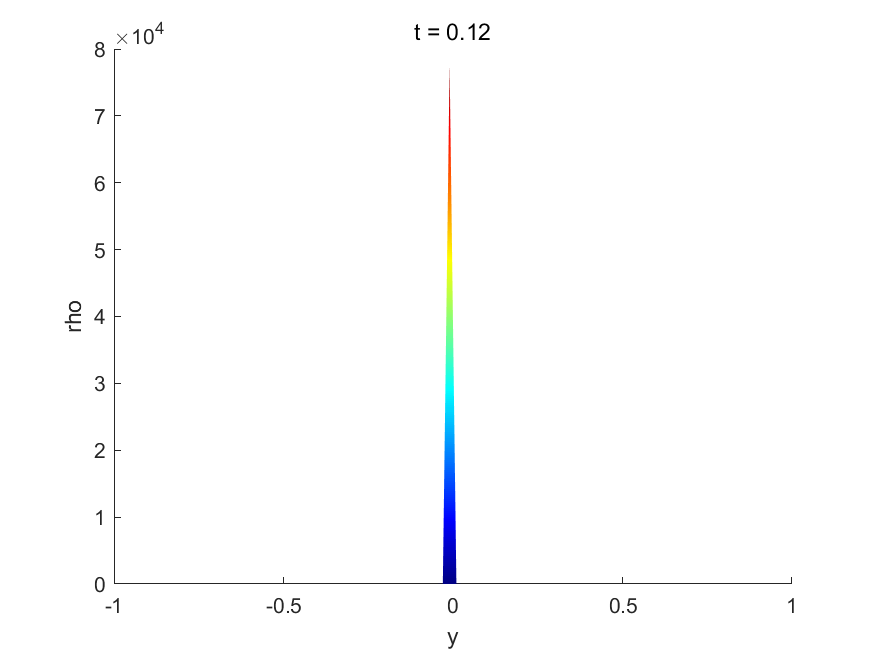}
	\caption{The contour plots (top) and side views (bottom) of $\rho_h$ for the PC-BCFD scheme  on uniform grids with $M=100$ at different time instants for Example \ref{exam:grate8pi:pp}.}
	\label{fig:PP-rhoBU-22M100-uni-solution}
\end{figure}
\begin{figure}[!h]
	\centering
	\includegraphics[width=0.32\linewidth]{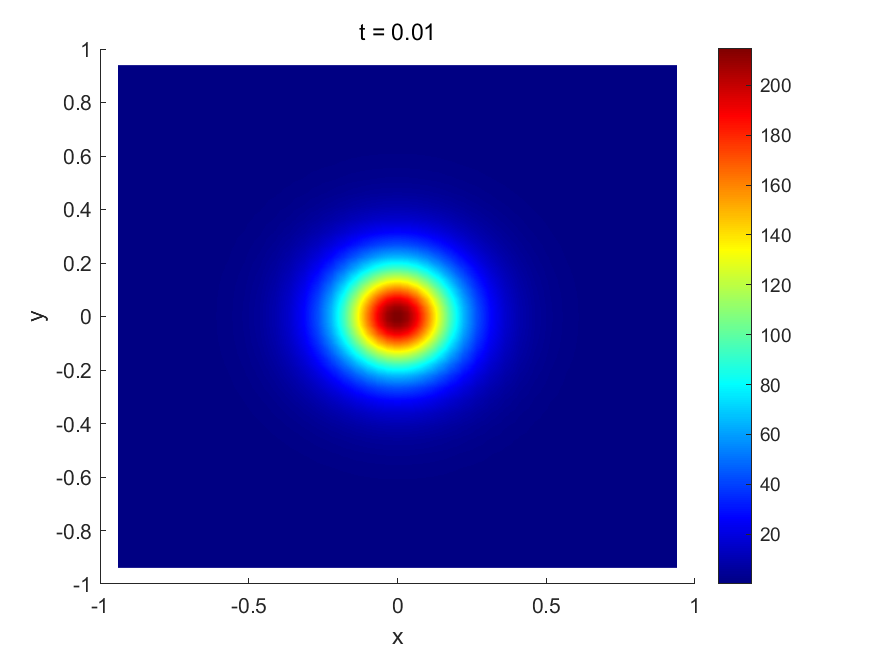}
	\includegraphics[width=0.32\linewidth]{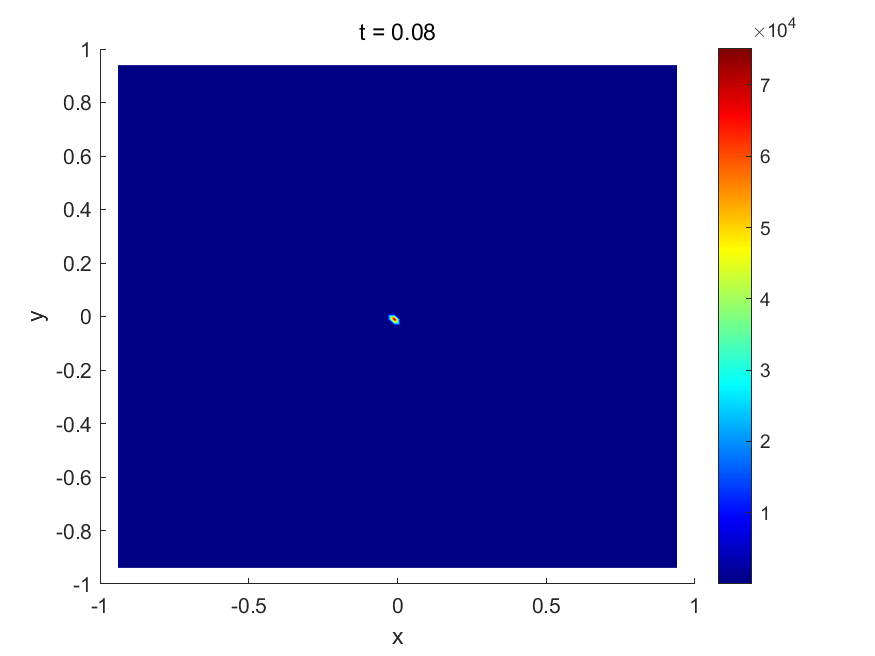}
	\includegraphics[width=0.32\linewidth]{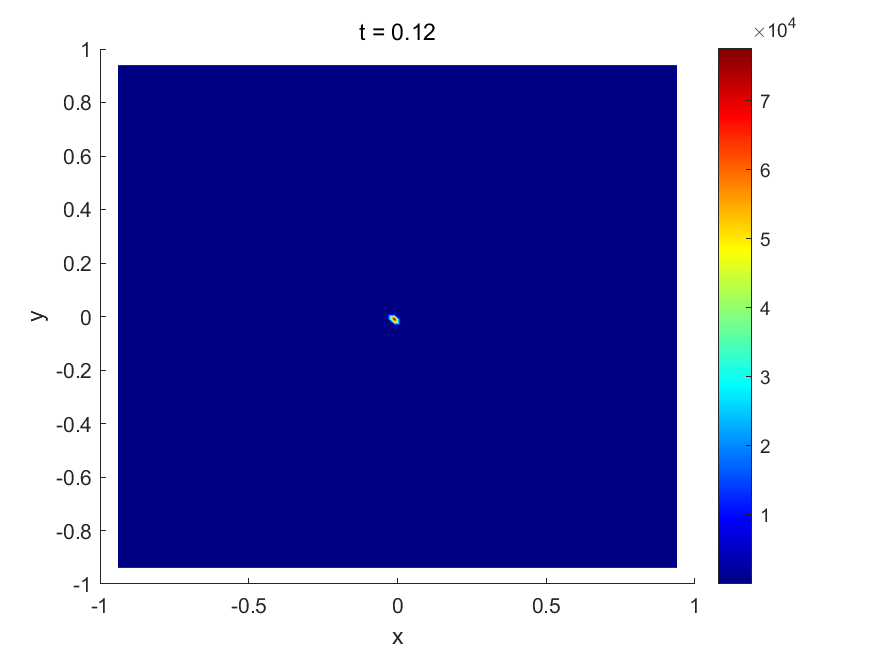}
	\includegraphics[width=0.32\linewidth]{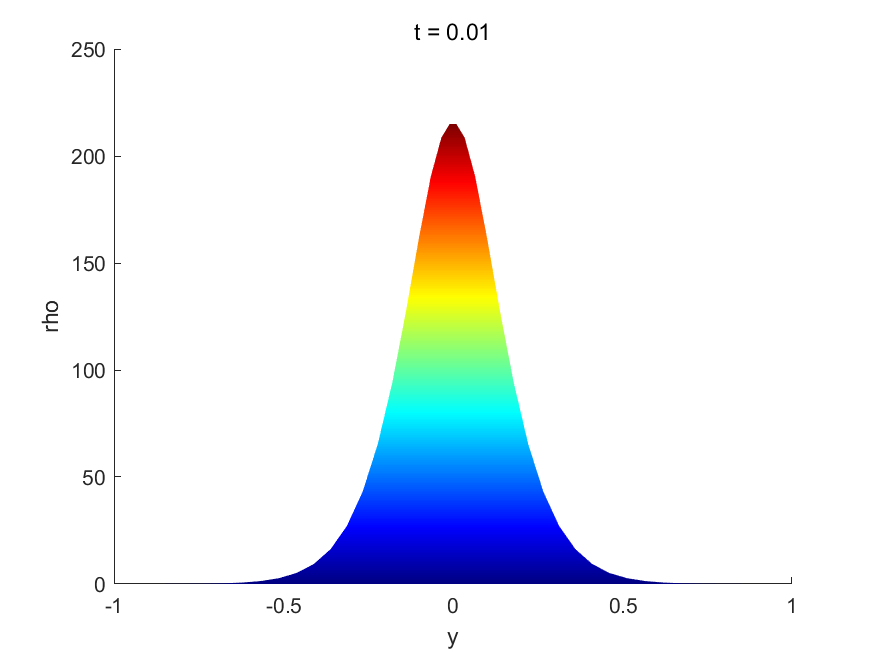}
	\includegraphics[width=0.32\linewidth]{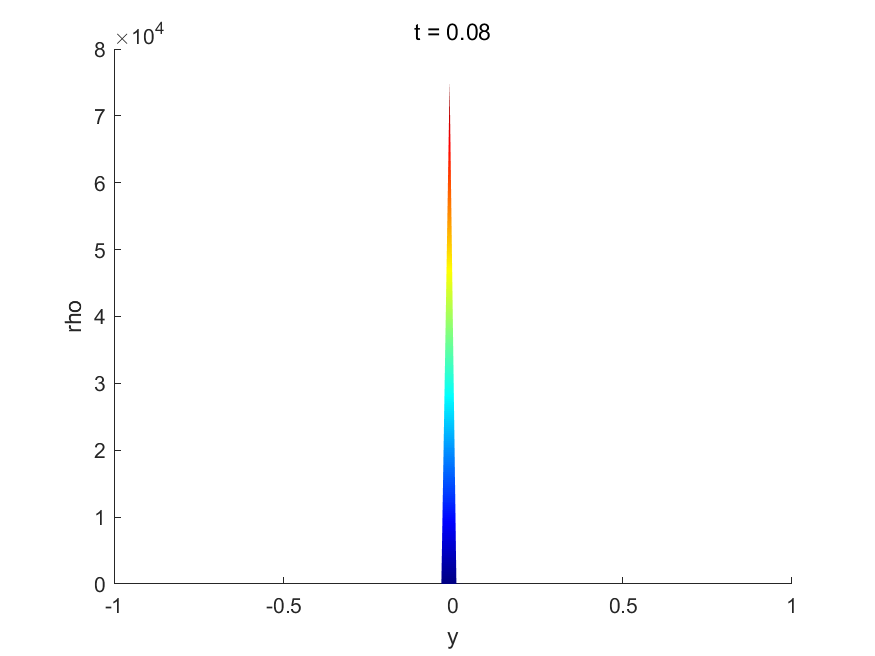}
	\includegraphics[width=0.32\linewidth]{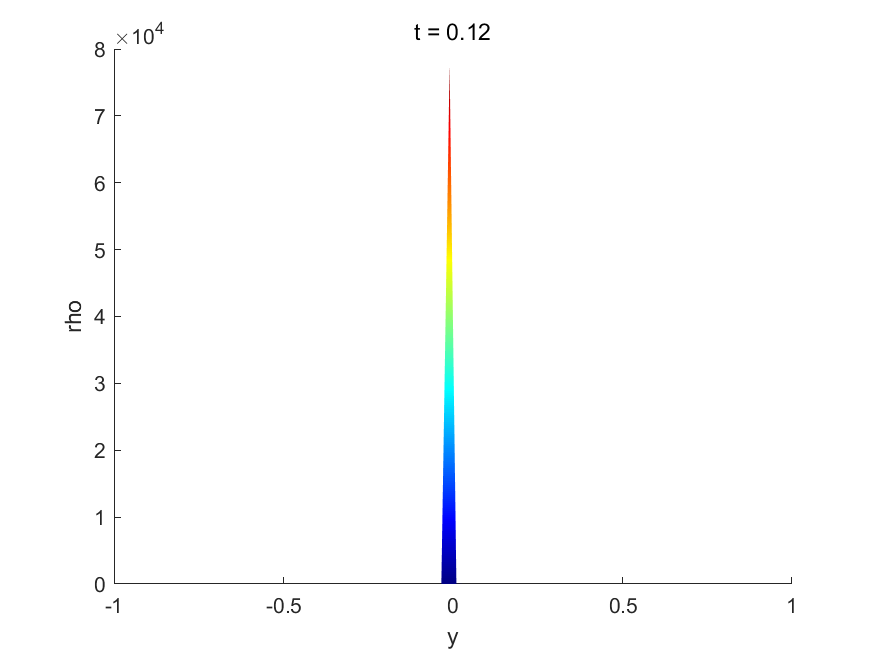}
	\caption{The contour plots (top) and side views (bottom) of $\rho_h$ for the PC-BCFD scheme  on non-uniform grids with $M=40$ and $\gamma=1.285$ at different time instants for Example \ref{exam:grate8pi:pp}.}
	\label{fig:PP-rhoBU-22M40-mid1285-solution}
\end{figure}

\begin{example}\label{exam:ppcheck}
 We finally consider a more challenging positivity test with extremely small minimum values of the initial density and concentration. More precisely, we take
\begin{equation*}
	\rho^o(x, y)=1000\exp\left(-100(x^2+y^2)\right),\quad
			c^o(x, y)=50\exp\left(-50(x^2+y^2)\right),
\end{equation*}
in the domain $\Omega=(-1,1)^2$. The corresponding initial total cell mass is $	M_\rho(0)\approx 31.36>8\pi$ .
\end{example}

\begin{figure}[!t]
	\centering
	\includegraphics[width=0.32\linewidth]{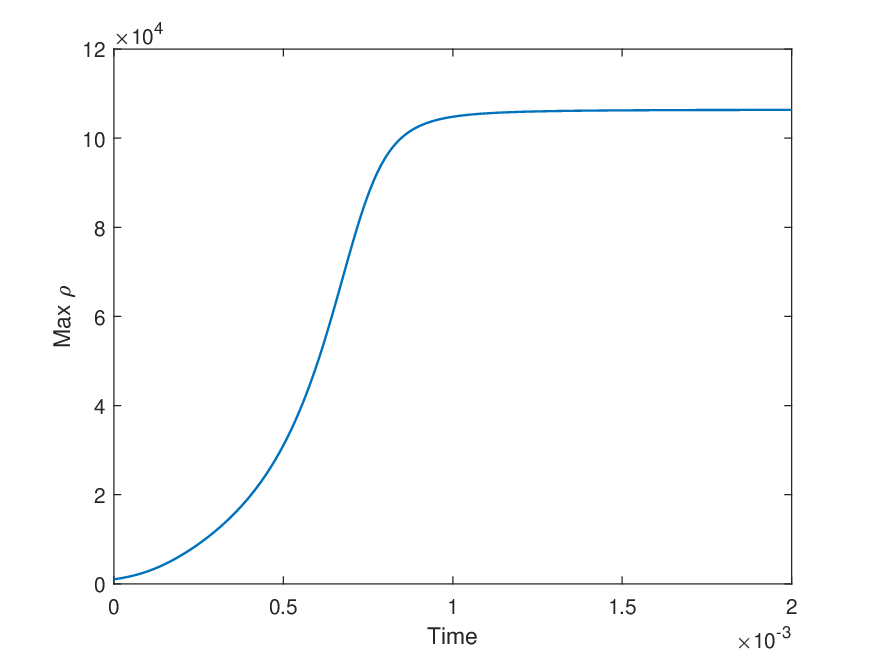}
	\includegraphics[width=0.32\linewidth]{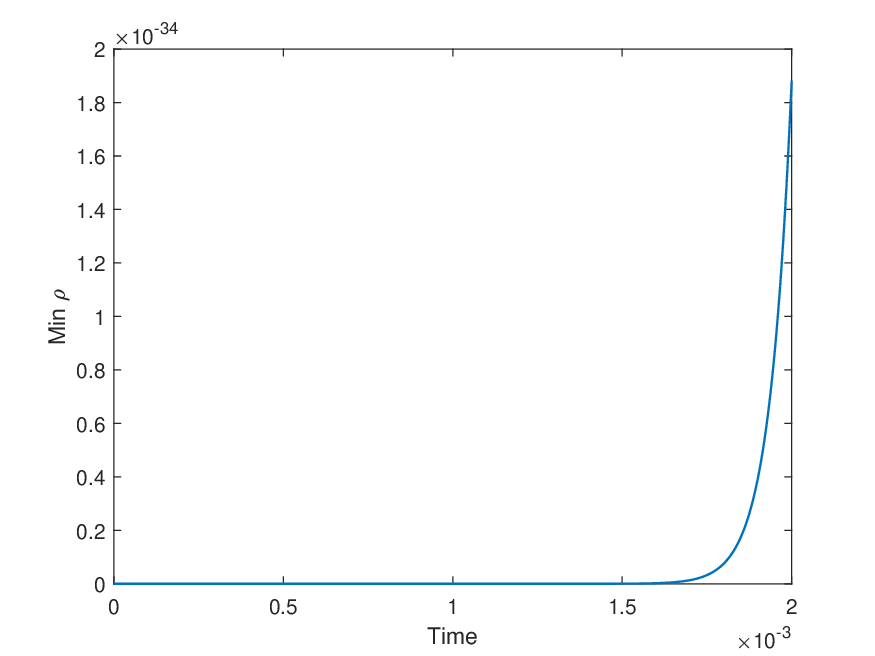}
	\includegraphics[width=0.32\linewidth]{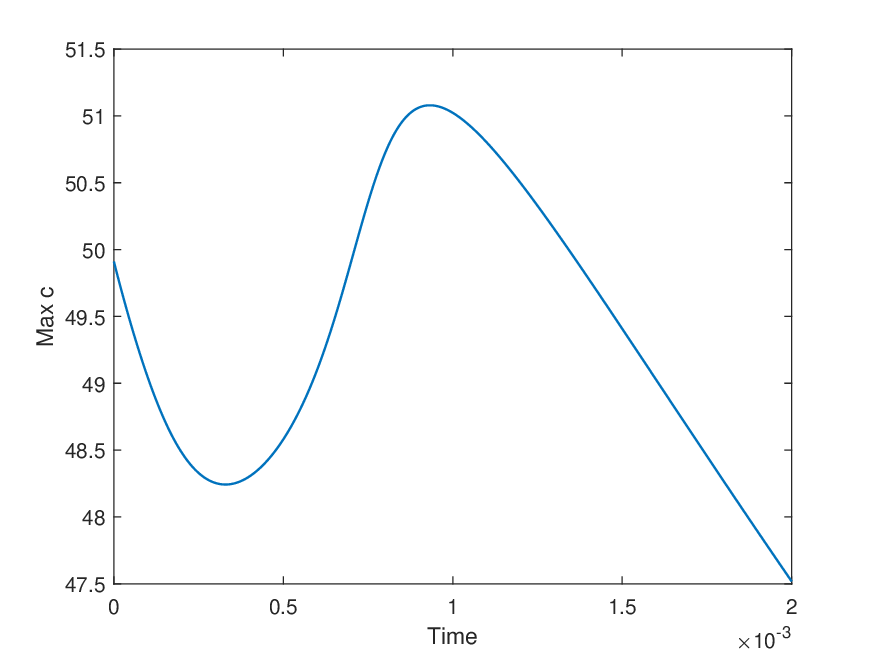}
	\includegraphics[width=0.32\linewidth]{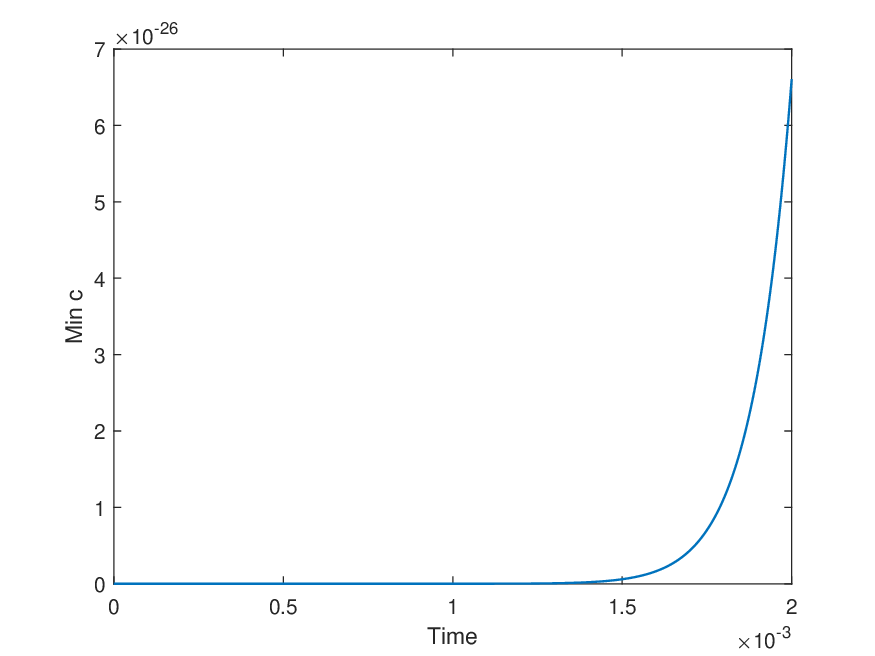}
	\includegraphics[width=0.32\linewidth]{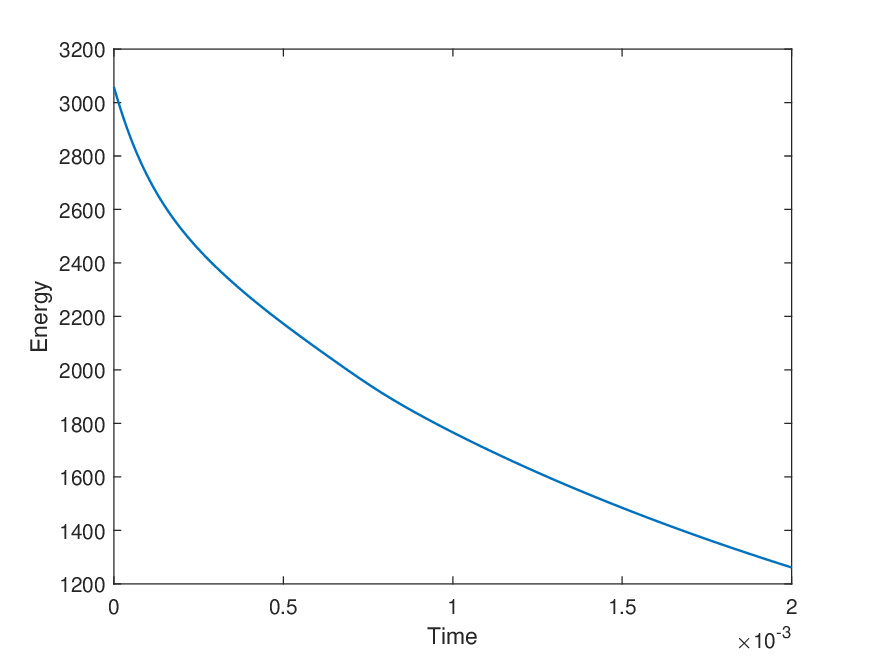}
	\includegraphics[width=0.32\linewidth]{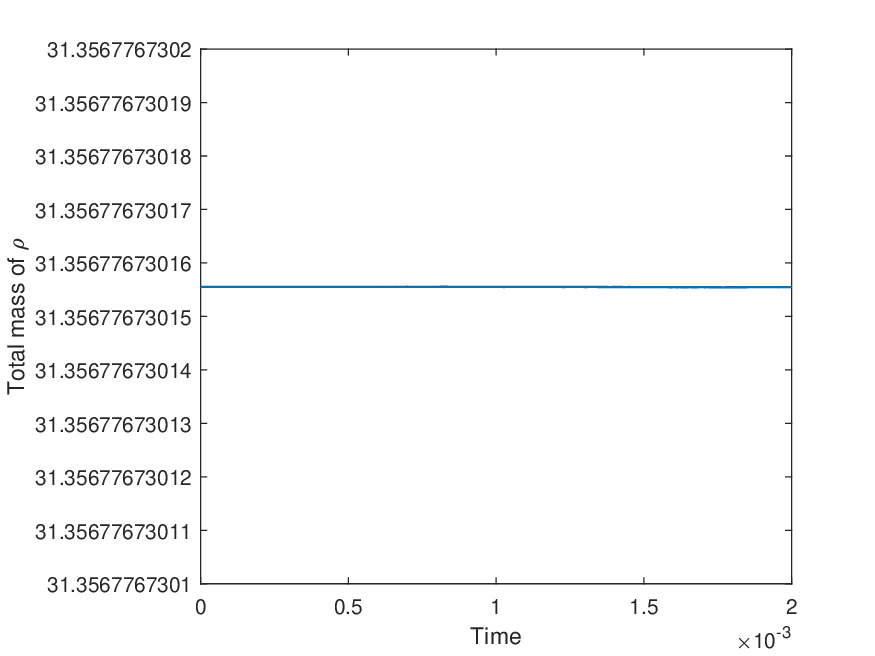}
	\caption{Evolutions of cell density, chemoattractant concentration, discrete energy and total cell mass for the BE-BCFD scheme on non-uniform grids with $M=80$ and $\gamma=1.285$ for Example \ref{exam:ppcheck}. }
	\label{fig:PP-rhoBU-12M80-mid1285}
\end{figure}
\begin{figure}[!t]
	\vspace{-10pt}
	\centering
	\includegraphics[width=0.32\linewidth]{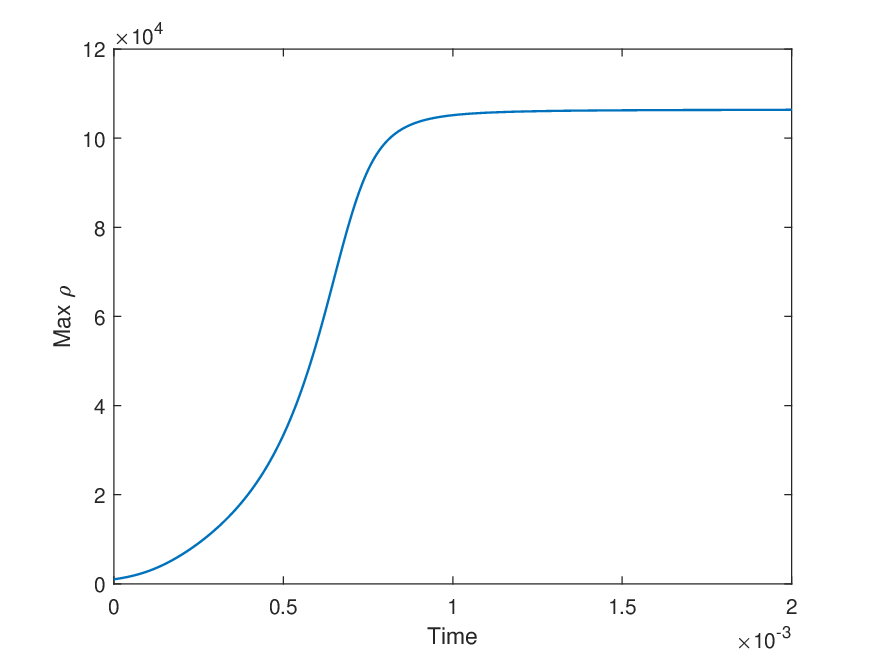}
	\includegraphics[width=0.32\linewidth]{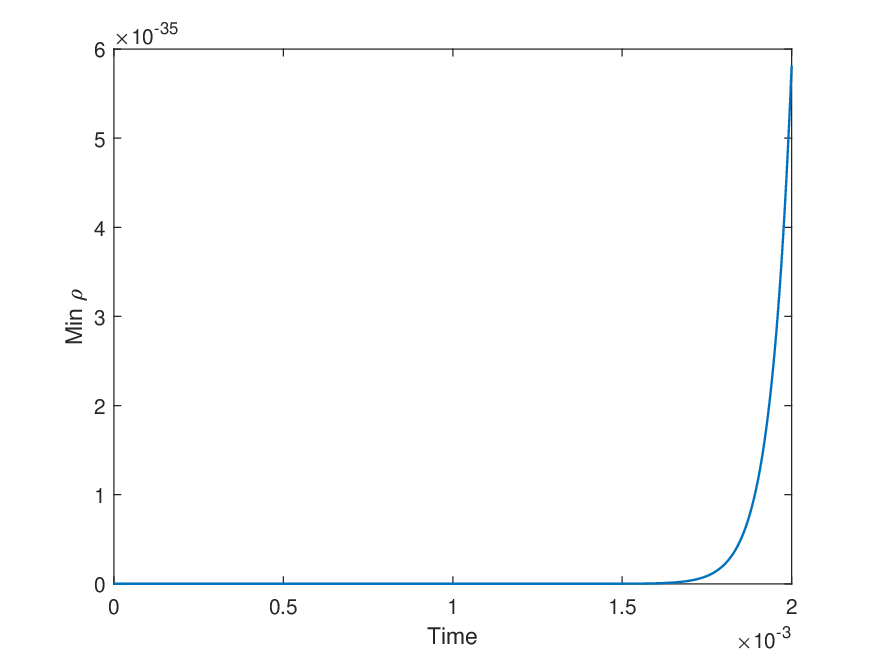}
	\includegraphics[width=0.32\linewidth]{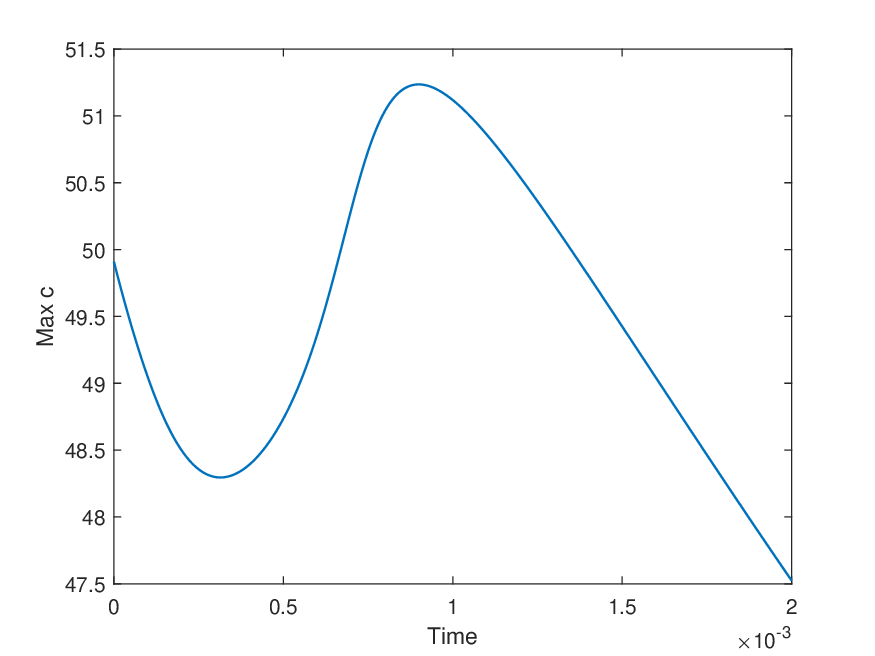}
	\includegraphics[width=0.32\linewidth]{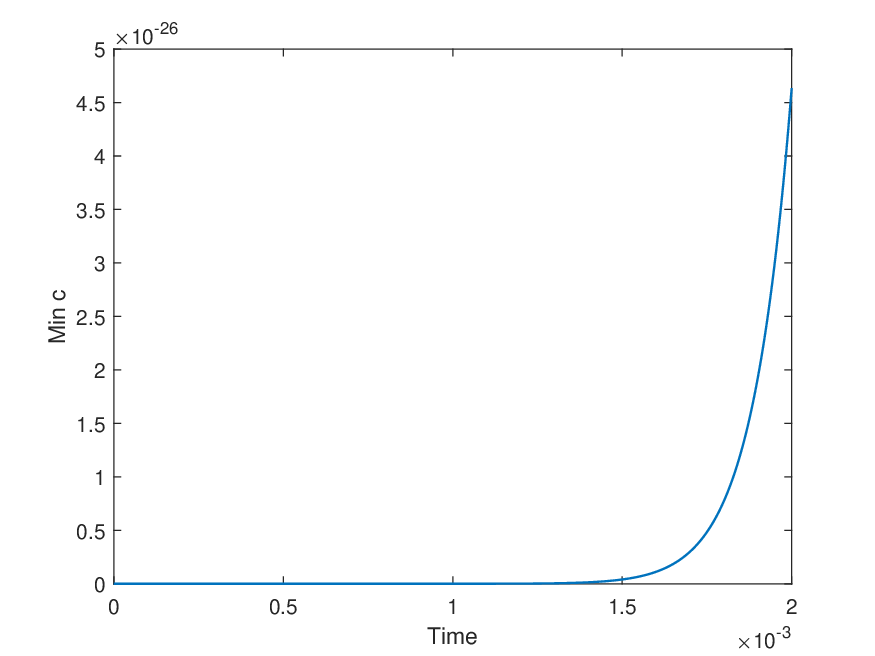}
	\includegraphics[width=0.32\linewidth]{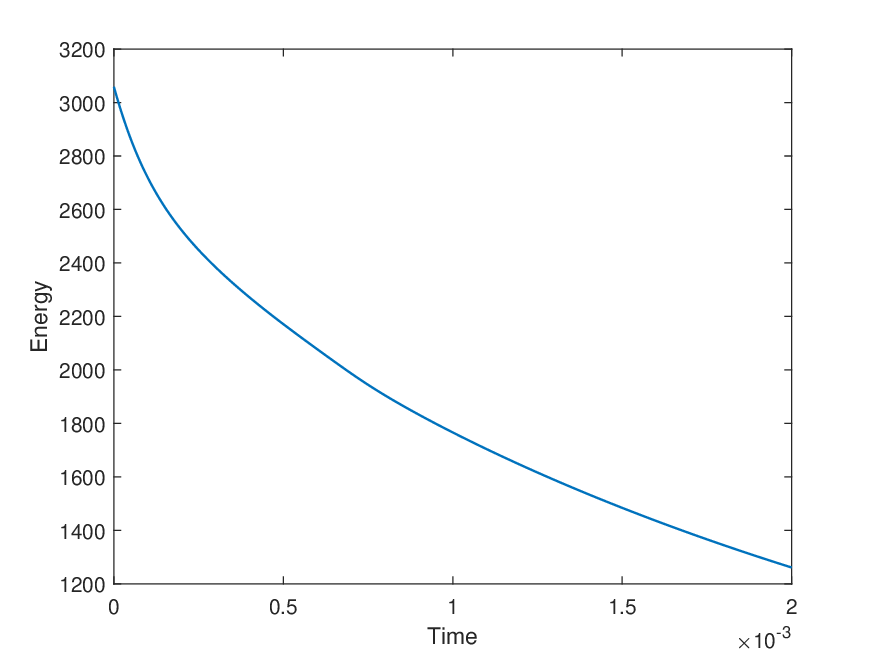}
	\includegraphics[width=0.32\linewidth]{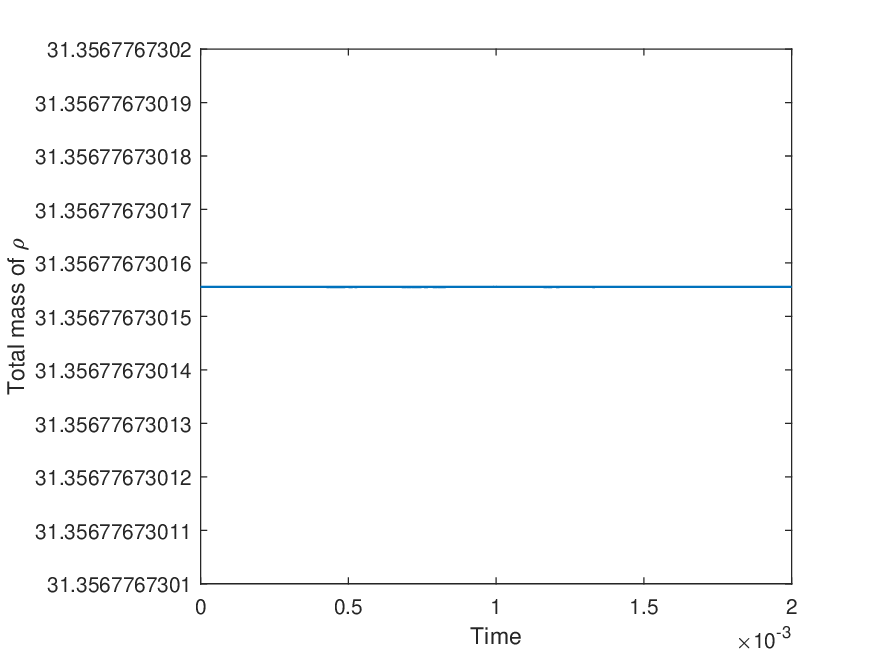}
	\caption{Evolutions of cell density, chemoattractant concentration, discrete energy and total cell mass for the PC-BCFD scheme on non-uniform grids with $M=80$ and $\gamma=1.285$ for Example \ref{exam:ppcheck}. }
	\label{fig:PP-rhoBU-22M80-mid1285}
\end{figure}
Although the initial data are strictly positive, their minimum values become extremely close to zero. This makes the positivity-preserving test particularly challenging. We use the non-uniform grids generated by \eqref{grid:mid} with $M=80$ and $\gamma=1.285$. Setting the time stepsize to $\tau=5\times 10^{-6}$, we compute the cell density and chemoattractant concentration up to $T=2\times 10^{-3}$. The evolutions of the cell density $\rho_h$, the chemoattractant concentration $c_h$, the discrete energy, and the total mass of $\rho_h$ for the BE-BCFD and PC-BCFD schemes are shown in Figs. \ref{fig:PP-rhoBU-12M80-mid1285} and \ref{fig:PP-rhoBU-22M80-mid1285}, respectively. We observe that: (i) Both schemes preserve the positivity of $\rho_h$ and $c_h$ throughout the simulation. Even though their minimum values drop to machine-precision levels, the numerical results still show no sign of spurious negative undershoots. (ii) The maximum density grows rapidly from about $9.96\times 10^2$ to approximately $1.06\times 10^5$, indicating strong chemotactic aggregation. (iii) The discrete energy decreases monotonically, and the total cell mass is conserved up to round-off error. 
	
In summary, this example demonstrates that both the BE‑BCFD and PC‑BCFD schemes robustly preserve positivity for highly concentrated positive initial data with extremely small minima. We also note that the sufficient positivity-preserving condition in Theorem \ref{thm:pc-posi-preserve} is far from being satisfied in this test. For the middle-refinement grid with $M=80$ and $\gamma=1.285$, we have approximately $h_{\min}^2/2\sigma^6\approx 2.26\times 10^{-10}$, which is much smaller than the time stepsize $\tau=5\times 10^{-6}$. Moreover,
	\[
	\max_{\Omega} c^o(x, y)\approx 49.91,\qquad
	\max_{\Omega} e^{c^o(x, y)}\approx 4.74\times 10^{21}.
	\]
Thus the sufficient time-step restriction for the positivity of $\rho_h$ becomes extremely restrictive. Nevertheless, the computed minimum values of both $\rho_h$ and $c_h$ remain positive, further suggesting that the theoretical time-step condition is only sufficient and, in practice, can be quite liberal.
	
\section{Conclusions}\label{sec:conclusion}
In this paper, we have developed two structure-preserving BCFD schemes for the classical Keller--Segel chemotaxis system on staggered non-uniform grids; see \eqref{BE-BCFD:rewrite}--\eqref{BE-BCFD:IBc:rewrite} and \eqref{PC-BCFD:rewrite}--\eqref{PC-BCFD:IBc:rewrite}. The main conclusions are as follows:
\begin{itemize}
\item The proposed schemes preserve the key physical properties of the Keller–Segel system at the discrete level on general non-uniform spatial grids. Specifically, the temporal first-order scheme is rigorously shown to unconditionally preserve positivity and mass conservation, and it also satisfies a discrete energy-dissipation law. For the temporal second-order scheme, however, only positivity and mass conservation are proved; its energy dissipation is observed numerically but has not yet been theoretically established.

\item The proposed schemes retain the expected spatial and temporal accuracy on general non-uniform spatial grids. Furthermore, the staggered non-uniform BCFD formulation naturally accommodates local mesh refinement near the blow-up region (see Fig.~\ref{middle_grid}), which improves the simulation efficiency and accuracy for the Keller–Segel chemotaxis system while preserving the desired physical properties. 
\end{itemize}

In Ref.~\cite{XF'25}, we established a superconvergence result for a mass-conservative BCFD method on general non-uniform spatial grids for the Keller–Segel system \eqref{model:ks}. However, extending the convergence analysis to the structure-preserving schemes derived from the Slotboom reformulation considered in this paper remains an important and nontrivial direction for future investigation. Moreover, developing an unconditional positivity-preserving and energy-dissipating linear second-order finite difference scheme remains challenging.

\section*{CRediT authorship contribution statement}
\textbf{Jie Xu} : Methodology, Formal analysis, Investigation, Software, Writing--original draft.
\textbf{Hongfei Fu} : Methodology, Conceptualization, Supervision, Writing--review and editing, Funding acquisition.

\section*{Declaration of competing interest} The authors declare that they have no competing interests.	
\section*{Data availability}  Data are available upon reasonable request.
\section*{Acknowledgements}
This work was supported in part by the National Natural Science Foundation of China (No. 12131014), by the Shandong Provincial Natural Science Foundation (No. ZR2024MA023), and by the Shandong Provincial Key Laboratory of Stochastic System Control and Scientific Computing.

\bibliographystyle{elsarticle-num}
\bibliography{Ref_KS.bib}

\end{document}